\documentclass[12pt,a4paper,twoside]{amsbook}

\usepackage[utf8]{inputenc}

\usepackage[british]{babel}
\usepackage{amssymb,mathtools,stmaryrd}
\usepackage{enumerate}
\usepackage{setspace}

\calclayout 
\usepackage{hyperref}
\usepackage{cleveref}
\usepackage{fouridx}

\usepackage{times,fancyhdr,color}
\usepackage{latexsym}
\usepackage{amsfonts}
\usepackage{faktor}
\usepackage{tikz}
\usetikzlibrary{cd}
\usetikzlibrary{positioning}
\usepackage{bussproofs}
\usepackage{float}
\usepackage{multirow}
\usepackage[utf8]{inputenc}
\usepackage{epigraph}
\usepackage{soul}
\usepackage{amsmath, amsthm}
\usepackage{mathrsfs}

\usepackage{longtable}
\usepackage{booktabs}

\usepackage{imakeidx}

\numberwithin{section}{chapter}
\numberwithin{equation}{section}

\newcommand{\emi}[1]{\index{#1}\emph{#1}}%

\def\defthm#1#2#3#4{
	\newtheorem{#1}[theorem]{#3}
	\newtheorem*{#1*}{#3}
	\newtheorem{#2}[theorem]{#4}
	\newtheorem*{#2*}{#4}
	\crefname{#1}{#3}{#4}
	\crefname{#2}{#4}{#4}  
}

\newtheoremstyle{mythm}%
{10pt}
{}
{\itshape}
{}
{\bf}
{.}
{.5em}
{}%

\newtheoremstyle{mydef}%
{10pt}
{3pt}
{}
{}
{\bf}
{.}
{.5em}
{}%

\newtheoremstyle{myrmk}%
{10pt}
{3pt}
{}
{}
{\bf}
{.}
{.5em}
{}%

\theoremstyle{mythm}
\newtheorem{theorem}{Theorem}[section]
\newtheorem*{theorem*}{Theorem}

\defthm{corollary}{corollaries}{Corollary}{Corollaries}
\defthm{lemma}{lemmata}{Lemma}{Lemmata}
\defthm{proposition}{propositions}{Proposition}{Propositions}
\defthm{prop}{props}{Proposition}{Propositions}

\theoremstyle{mydef}
\defthm{definition}{definitions}{Definition}{Definitions}

\theoremstyle{myrmk}
\defthm{remark}{remarks}{Remark}{Remarks}
\defthm{example}{examples}{Example}{Examples}
\defthm{question}{questions}{Question}{Questions}
\defthm{claim}{claims}{Claim}{Claims}
\defthm{conjecture}{conjectures}{Conjecture}{Conjectures}

\newtheorem*{replemmax}{\reptitle}
	{\end{replemmax}}

\crefname{section}{Section}{Sections}
\crefname{theorem}{Theorem}{Theorems}

\makeatletter
\renewenvironment{proof}[1][\proofname] {\par\pushQED{\qed}\normalfont\topsep6\p@\@plus6\p@\relax\trivlist\item[\hskip\labelsep\bf#1\@addpunct{.}]\ignorespaces}{\popQED\endtrivlist\@endpefalse}
\makeatother

\newcommand{\ie}{\text{i.e.\ }}
\newcommand{\eg}{\text{e.g.\ }}

\newcommand{\etal}{\text{et al.\ }}

\newcommand{\etc}{\text{etc.\ }}

\crefname{diagram}{diagram}{diagrams}

\newcommand{\norm}[1]{\left\lVert#1\right\rVert}

\newcommand{\y}[1]{\operatorname{\textsf{y}}#1}

\providecommand{\noopsort}[1]{}

\makeatletter
\def\degreedate#1{\gdef\@degreedate{#1}}
\def\degree#1{\gdef\@degree{#1}}
\def\collegeordept#1{\gdef\@collegeordept{#1}}
\def\university#1{\gdef\@university{#1}}
\def\crest#1{\gdef\@crest{#1}}
\def\logo#1{\gdef\@logo{#1}}
\def\deptlogo#1{\gdef\@deptlogo{#1}}
\makeatother

\title{Computability and Tiling Problems}
\author{Mark Richard Carney}
\crest{\includegraphics[width=35mm]{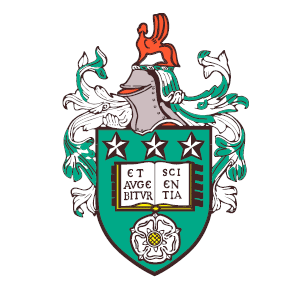}}
\logo{\includegraphics[width=50mm]{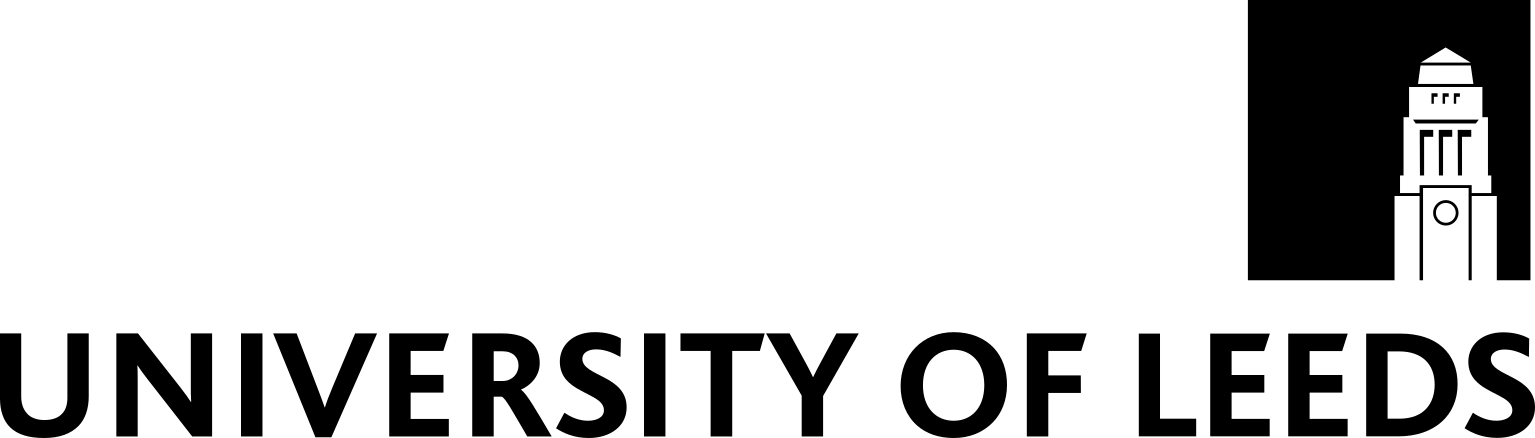}} 
\deptlogo{} 
\collegeordept{\href{https://physicalsciences.leeds.ac.uk/info/6/school_of_mathematics}{School of Mathematics}}
\university{\href{http://www.leeds.ac.uk}{University of Leeds}}

\degree{Doctor of Philosophy}
\newcommand{\submittedtext}{{Submitted in accordance with the requirements for the degree of}}

\degreedate{October 2019}

\usepackage{tikz}

\newcommand{\sampletile}[4]{\scalebox{.75}{
\tikz[scale=1.7,label distance=2.5mm]{
\draw[fill=white] (1.6,1.6) rectangle (0,0);
\filldraw[fill=white] (0,0) -- (0.8,0.8) -- (0,1.6) -- cycle;
\filldraw[fill=white] (0,0) -- (0.8,0.8) -- (1.6,0) -- cycle;
\filldraw[fill=white] (0,1.6) -- (0.8,0.8) -- (1.6,1.6) -- cycle;
\node at (0.8,0.8) [label=above:{#2},label=left:{#1},label=right:{#3},label=below:{#4}] {};}
}}

\newcommand{\sampletilefarlabels}[4]{\scalebox{.75}{
\tikz[scale=1.7,label distance=4.5mm]{
\draw[fill=white] (1.6,1.6) rectangle (0,0);
\filldraw[fill=white] (0,0) -- (0.8,0.8) -- (0,1.6) -- cycle;
\filldraw[fill=white] (0,0) -- (0.8,0.8) -- (1.6,0) -- cycle;
\filldraw[fill=white] (0,1.6) -- (0.8,0.8) -- (1.6,1.6) -- cycle;
\node at (0.8,0.8) [label=above:{#2},label=left:{#1},label=right:{#3},label=below:{#4}] {};}
}}

\newcommand{\sampletilenearlabels}[4]{\scalebox{.75}{
\tikz[scale=1.7,label distance=1.5mm]{
\draw[fill=white] (1.6,1.6) rectangle (0,0);
\filldraw[fill=white] (0,0) -- (0.8,0.8) -- (0,1.6) -- cycle;
\filldraw[fill=white] (0,0) -- (0.8,0.8) -- (1.6,0) -- cycle;
\filldraw[fill=white] (0,1.6) -- (0.8,0.8) -- (1.6,1.6) -- cycle;
\node at (0.8,0.8) [label=above:{#2},label=left:{#1},label=right:{#3},label=below:{#4}] {};}
}}

\newcommand{\sampletilefill}[4]{\scalebox{.75}{
\tikz[scale=1.7,label distance=2.5mm]{
\draw[fill=#3] (1.6,1.6) rectangle (0,0);
\filldraw[fill=#1] (0,0) -- (0.8,0.8) -- (0,1.6) -- cycle;
\filldraw[fill=#4] (0,0) -- (0.8,0.8) -- (1.6,0) -- cycle;
\filldraw[fill=#2] (0,1.6) -- (0.8,0.8) -- (1.6,1.6) -- cycle;
{};}
}}

\usetikzlibrary{shapes.geometric}

\newcommand\fillshapesimple[1]{ 
    \begin{scope}
        \clip #1;
    \end{scope}
    \draw[line width=0.2mm] #1;
}

\newcommand\lozengetile[5]{
    \begin{scope}
        \clip #1;
    \end{scope}
    \draw[line width=0.1mm] #1;
\draw[line width=0.2mm] (30:0.66) node {#2};
\draw[line width=0.2mm] (15:1.33) node {#3};
\draw[line width=0.2mm] (330:0.66) node {#5};
\draw[line width=0.2mm] (345:1.33) node {#4};
}
\newcommand\hexagontile[4]{
  \begin{scope}
    \clip (0:0) -- (180:2) -- (120:2) -- (0:0) -- (120:2) -- (60:2) -- (0:0) -- (60:2) -- (0:2) -- (0:0) -- (0:2) -- (300:2) -- (240:2) -- (180:2) -- cycle;
  \end{scope}
  \draw[line width=0.2] (0:0) -- (180:2) -- (120:2) -- (0:0) -- (120:2) -- (60:2) -- (0:0) -- (60:2) -- (0:2) -- (0:0) -- (0:2) -- (300:2) -- (240:2) -- (180:2) -- cycle;
\draw (150:1) node {$#1$};
\draw (90:1) node {$#2$};
\draw (30:1) node {$#3$};
\draw (270:1) node {$#4$};
}

\newcommand\numberthis{\addtocounter{equation}{1}\tag{\theequation}}

\makeindex

\begin{document}

\pagenumbering{roman}

\begin{titlepage}
\makeatletter
\begin{picture}({50,25})
\put(-75,75){\hbox{\@deptlogo}}
\end{picture}
\begin{picture}({50,25})
\put(210,75){\hbox{\@logo}}
\end{picture}
\begin{center}
	{ \Huge {\bfseries {\@title}} \par}
	{\large \vspace*{20mm} {{\@crest} \par} \vspace*{20mm}}
	{{\Large \authors} \par}
	{\large \vspace*{1ex}
		{{\@university} \par}
		\vspace*{1ex}
		{{\@collegeordept} \par}
		\vspace*{20mm}
		{{\submittedtext} \par}
		\vspace*{1ex}
		{\it {\@degree} \par}
		\vspace*{2ex}
		{\@degreedate}}
\end{center}
\makeatother
\null\vfill
\end{titlepage}

\addtocounter{page}{-1}
\cleardoublepage
\pagestyle{plain}
\begin{center}
	\vspace*{1.5cm}
	{\Large \bfseries Intellectual Property Statement}
\end{center}
\vspace{0.5cm}
\begin{quote} 
		The candidate confirms that the work submitted is his own and that
		appropriate credit has been given where reference has been made to
		the work of others.
		\\
		\\
		This copy has been supplied on the understanding that it is copyright
		material and that no quotation from the thesis may be published
		without proper acknowledgement.
		\\
		\\
		The right of \authors ~to be identified as Author of this work
		has been asserted by him in accordance with the Copyright, Designs
		and Patents Act 1988.
		\\
		\\
		\copyright\; October 2019 The University of Leeds and \authors.
\end{quote}
\setcounter{page}{0}
\addtocounter{page}{-1}
\pagestyle{headings}
\cleardoublepage
\centerline{\Huge{Abstract}}
\addcontentsline{toc}{section}{Abstract}
In this thesis we will present and discuss various results pertaining to tiling problems and mathematical logic, specifically computability theory.

We focus on Wang prototiles, as defined in \cite{GrunbaumTP}. We begin by studying Domino Problems, and do not restrict ourselves to the usual problems concerning finite sets of prototiles. We first consider two domino problems: whether a given set of prototiles $S$ has total planar tilings, which we denote $TILE$, or whether it has infinite connected but not necessarily total tilings, $WTILE$ (short for `weakly tile'). We show that both $TILE \equiv_m ILL \equiv_m WTILE$, and thereby both $TILE$ and $WTILE$ are $\Sigma^1_1$-complete. We also show that the opposite problems, $\neg TILE$ and $SNT$ (short for `Strongly Not Tile') are such that $\neg TILE \equiv_m WELL \equiv_m SNT$ and so both $\neg TILE$ and $SNT$ are both $\Pi^1_1$-complete.

Next we give some consideration to the problem of whether a given (infinite) set of prototiles is periodic or aperiodic. We study the sets $PTile$ of periodic tilings, and $ATile$ of aperiodic tilings. We then show that both of these sets are complete for the class of problems of the form $(\Sigma^1_1 \wedge \Pi^1_1)$. We also present results for finite versions of these tiling problems.

We then move on to consider the Weihrauch reducibility for a general total tiling principle $CT$ as well as weaker principles of tiling, and show that there exist Weihrauch equivalences to closed choice on Baire space, $C_{\omega^\omega}$. We also show that all Domino Problems that tile some infinite connected region are Weihrauch reducible to $C_{\omega^\omega}$.

Finally, we give a prototile set of 15 prototiles that can encode any Elementary Cellular Automaton (ECA). We make use of an unusual tile set, based on hexagons and lozenges that we have not see in the literature before, in order to achieve this.

\cleardoublepage \addcontentsline{toc}{section}{Dedication}
\vspace*{\stretch{1}}
\begin{center}
\Large{\textit{Dedicated to Prof. S. Barry Cooper}}
\end{center}
\vspace*{\stretch{2}}


\cleardoublepage
\centerline{\Huge{Acknowledgements}}
\addcontentsline{toc}{section}{Acknowledgements}

I wish to thank my supervisors Dr. Paul Shafer and Prof. Michael Rathjen, both of whom have guided and inspired me along this journey. In particular, Dr. Shafer's regular engagement and inspiring passion for computability, logic, and mathematics, matched with his rigorous approach and choice quality coffee, has been one of the most profound privileges to work with. The generosity of time, expertise, and guidance from both of my supervisors has made this Ph.D. possible.

I am grateful to my colleagues and friends in the Logic Group at the University of Leeds for their support in answering questions and commenting on ideas, in particular:  Giovanni Solda, Emanuele Frittaion, Alberto Marcone, Marta Fiori Carones, John Truss, Stan Wainer, Andrew Brooke-Taylor, Charles Harris, Anton Freund, Martin Krombholz, Bjarki Geir Benediktsson, John Howe, Jakob Vidmar, Richard Matthews, Rosario Mennuni, James Gay, James Riley, Anja Komatar, Cong Chen, Richard Whyman, Sarah Sigley, Cesare Gallozzi, and Petra Staynova. Of note is Emanuele Frittaion, whose observation regarding Weihrauch degrees after a seminar I gave spawned the work in Chapter 5, and Andrew Brooke-Taylor for the references on Vop\v{e}nka's Principle in Chapter 4.

My thanks to the University of Leeds and the School of Mathematics for their hosting me, and to the EPSRC for their financial assistance. It has been gratefully received.

My thanks to my parents Pauline and Richard, my brother Edward, and my family - in particular my mum, who has lingered through several conversations I am certain she was not expecting to understand all of. My love and thanks to you all for supporting me in many ways.

I acknowledge the plentiful support of my friends and colleagues in the technology scene, including the participants at Leeds Hackspace, DC151, and DC11331, and the DEFCON and BSides family. I honour the marvellous work of the baristas of Leeds.

Lastly, my thanks to my first supervisor on this Ph.D., Prof. S. Barry Cooper, to whom this thesis is dedicated. Without you, Barry, I would not have been so encouraged and prepared to undertake this task in the beginning.

Cheers, Barry, and thank you, my friends.


\cleardoublepage
\addcontentsline{toc}{section}{Contents}
\tableofcontents


\clearpage
\addcontentsline{toc}{section}{List of figures}
\markboth{List of figures}{}
\listoffigures
\clearpage
\addcontentsline{toc}{section}{List of tables}
\markboth{List of tables}{}
\listoftables
\cleardoublepage

\pagenumbering{arabic}
\chapter*{Introduction}
\label{chap:intro}
\setcounter{equation}{0}
\renewcommand{\theequation}{\thechapter.\arabic{equation}}

In this thesis we will explore the connections between tiling problems and logic, specifically in relation to, and through the lens of, computability theory. 

\section*{Background to the Thesis}

Broadly speaking, the tiling problems we study fall into two categories, for given prototile set $S$:

\begin{enumerate}
\item Domino Problems - the question of whether $S$ tiles the plane.
\item Tiling Properties - do all/any $S$-tilings have some specific property, \eg are they all periodic or aperiodic?
\end{enumerate}

We will construct well defined versions of both of these problems, and study their relationships to various areas of computability theory. 

This thesis builds on results that the author first presented in their MSc dissertation \cite{MCarneyMSc} as part of their MSc Mathematics at the University of Leeds. In that work, we presented some ways to code various results in computability, as well as elementary cellular automata, into sets of Wang prototiles.

In building on these results, we explore with much more depth the ways in which the classes of tiling problems listed above relate to various aspects of computability. We ask questions along the following lines:
\begin{itemize}
\item What are the computable parts of a given tiling problem?
\item How do tiling problems fit into existing computability hierarchies?
\end{itemize}
We also present improved versions of the Elementary Cellular Automata tilings using an original tile schema that we have constructed for this purpose.

\subsection*{Motivations}

There are some very interesting results in the literature regarding tiling problems and logic, and in general the aim is to determine both what conditions can be met by some given prototile set, and conversely whether there exist prototile sets that exhibit particular properties that are of interest. 

We will look at both finite and infinite sets of prototiles and determine results for both of these classes of possible tiling problems. Specifically, we are interested in formulating answers to the question:

\begin{center}
``What is the relative difficulty for a given problem about tile sets and tilings?''
\end{center}

This question, as the literature belies, is far from a foregone conclusion. The construction of a prototile set is intrinsically linked to the various patterns and behaviours of that set's tilings in the plane.

Given the well-studied logical strength of other combinatorial principles, we hope to expand the logical and mathematical vocabulary in this respect for tiling problems.

\subsection*{Computability and Tiling Problems}

In 1964 (see \cite{GrunbaumTP}) Wang proved that if a prototile set of Wang tiles - diagonally quadrisected square tiles - can tile any arbitrarily large finite portion of the plane, then it can tile the whole plane. This is a fairly straightforward compactness argument, and does indeed use K\"onig's lemma (cited as `K\"onig's Infinity Lemma' in \cite{GrunbaumTP}) to achieve the result, which we present in Chapter 2, Theorem \ref{thm:WangExtension}. 

Following on from this work, Wang continued to ask interesting questions regarding tiling problems. Indeed, many of the interesting results regarding tilings spawns from a conjecture due to Hao Wang in the early 60's:

\begin{conjecture}
It is necessary, as well as sufficient, that if a set of prototiles $S$ is periodic, it tiles the plane.
\end{conjecture}

Seeking an answer to this question, Berger in \cite{berger1966} formulated the first set of aperiodic Wang tiles - a prototile set consisting of 20,426 tiles that has only aperiodic tilings of the plane. This completely disproved Wang's conjecture, and demonstrated that periodicity is sufficient, but not necessary for a prototile set to tile the plane - thereby negating the conjecture.

Berger's refutation of Wang's conjecture was surprising, and laid the groundwork for further results in creating aperiodic prototile sets for a decade - the most well known of which are probably Penrose tilings. A summary of this work is given at the start of Chapter 4.  

In addition to creating the first aperiodic prototile sets, Berger was also the first to formulate the connection between Wang tilings and Turing Machines. The ultimate result was that the domino problem for finite sets of Wang prototiles, namely
\begin{center}
``Does a finite set of Wang prototiles $S$ tile the plane?''	
\end{center}
and the halting problem
\begin{center}
``Does a given Turing Machine $M$ halt on given input $x$?''
\end{center}
are equivalent, and these formed the central results of his thesis.

This equivalence was highly motivational for the current work we have regarding prototile sets and mathematical logic, as we can include the Domino Problem class of tiling problems for finite sets of prototiles as having the normal form of some $\Sigma^0_1$ formula - or the negation of one, if we desire an infinite planar tiling.

\section*{The Current Literature on Tiling Problems and Logic}

Firstly, we will summarise results in the literature that relate areas of logic to theorems and ideas about tiles, tilings, and prototile set properties and constructions.

Although Berger showed early on that Wang tiles are related to the undecidability of the Halting Problem, developments of using and studying tilings in mathematical logic is comparatively recent. 

Beginning with Harel in \cite{Harel1983}, who showed how problems of `high undecidability', \ie problems in $\Pi^1_1$, can be expressed as tiling problems. This is achieved in the plane by means of a set of carefully constructed Wang prototiles. Harel then built on this work in \cite{Harel1986} developing more full relationships between prototile sets and theorems about well/illfounded trees. Indeed, \cite{Harel1983} is cited by many texts in the field of Dynamic Logic - with Harel providing a chapter on this in the Handbook of Philosophical Logic \cite{Harel1984}. 

In `On the Convenience of Tilings' \cite{Boas97Conv}, van Emde Boas showed how various complexity classes are captured in specific tiling boundary results. Starting with an effective formulation of Turing Machines as prototile sets, van Emde Boas shows that a Wang prototile set that is unbounded vertically and horizontally is \textbf{NP}-complete, owing to the fact that a Turing Tape is realized left to right, whilst successive stages of a computation are realized vertically. Similarly, van Emde Boas continued by showing that a `corridor' tiling - a tiling that is of bounded width but unbounded height - is complete for \textbf{PSPACE}.

Following Durand's work on tilings and quasiperiodicity in \cite{Durand1999}, the work of Durand, Levin, and Shen \cite{Durand2008} showed that for every prototile set admits either no tiling or some tiling with $\mathcal{O}(n)$ Kolmogorov complexity of its $(n\times n)$-squares. Thatis to say, the string taken to describe any given square in the tiling has a complexity linearly related to the size of the square. This work was a continuation of their study of computational complexity paradigms and how they relate to tile sets and their planar tilings.

In Durand, Romashenko, and Shen \cite{Shen2010}, we find a significant development in the underlying theory of tilings - the existence of fixed point-based tilings. This work married up the work on Wang tiles with the previous work by Penrose and Amman on aperiodic Penrose tilings - see \cite[Chapters 10,11]{GrunbaumTP} for full presentations and discussions of these earlier works.

With these results in hand, recent work on $\Pi^0_1$ sets and tilings by Brown-Westrick in \cite{Westrick2017} utilised these self-similar Turing Machine tilings from \cite{Shen2010} in order to show that effectively closed subshifts of the distinct square shift are all sofic \cite[Theorem 1, 2]{Westrick2017}.

The study of tilings has, naturally from the above, been found and utilised in symbolic dynamics. A full introduction is found in the aforementioned Harel \cite{Harel1984}, with some interesting results being found recently in the work of Delvenne and Blondel \cite{Delvenne2004} where it is shown that by means of tiling problems, an analogue of Rice's theorem for computable functions is possible, giving that certain properties of dynamical systems are undecidable. As an extension to this result (Theorem 1 in \cite{Delvenne2004}), it is shown that topological entropy (as defined in \cite[Sec. 4.3, p.140]{Delvenne2004}) is undecidable for Turing Machines and tilings alike. Simpson in \cite{Simpson2007} also gave the following insight into tiling problems and their relation to mathematical logic, writing in \cite{Simpson2007} that:

\begin{quote}
``In the study of 2-dimensional subshifts of finite type, it has been useful to note that they are essentially the same thing as \emph{tiling problems} in the sense of Wang [ in \cite{Wang1990}].''
\end{quote}

Indeed, Levin's address, given as the Kolmogorov Lecture in 2005 at the University of London - see \cite{Levin2005} - gave some detail on the use of enumerable tilings in order to
prove that $2$-adic shifts and reflections can be enforced by a prototile set.

It is interesting to note that \cite{Delvenne2004} makes use of the notion of \emph{quasi-periodicity} - the property that every pattern $u$ of the tiling, there exists a $k$ such that any given $(k \times k)$ patch of tiles contains $u$. This notion is an interesting interim property that bridges the gap between fully periodic and fully aperiodic - see section \ref{subsec:quasiperiodic} for further details.

Adjacent to this work in mathematical logic, papers by Kari \cite{Kari1996} and later Culik \cite{CULIK1996245} showed how theorems about cellular automata that compute non-repeating reals can be converted into prototile sets to give very small sets of aperiodic prototiles. This work was generalised by Jeandel and Rao in \cite{Rao2015} to give the smallest possible set of aperiodic Wang prototiles, with a very small size of 11 prototiles to achieve this. They also proved through various means - both mathematically and with computational assistance - that this prototile set was smallest possible, and also had the property that if we were to remove any single tile from the prototile set, we no longer have tilings of the plane. Thereby, this prototile set either tiles aperiodically or fails to tile at all.

Having given this outline of the general view of tiling problems with respect to mathematical logic and related fields, we are now in a position to outline our contribution to this field.

\section*{Outline of the Thesis and Main Results}

Here we give an overview of the outline of the thesis, the main points in each chapter, and an account of the original work we are presenting in this volume.

\subsection*{Overview and Outline of the Thesis}

In chapter 1 we give a full background to the underlying mathematical logic and machinery we will use throughout the thesis. We give many definitions and present theorems generally without proofs, indicating sources along the way should they be necessary to the reader. We introduce precise definitions of Turing Machines as well as basic computability results that will be used later on. We also define various notions of reducibility in preparation for our work in Chapter 3.

We also give the background theory of computable trees as computable subsets of Baire space and Cantor space that form the backbone of many of our results in later chapters. We also give background results concerning the $\Pi^1_1$-completeness of Kleene's $\mathcal{O}$ which we shall use in later chapters. We finish this chapter with overview material for how computable trees, ordinals, and the arithmetical and analytic hierarchies hang together mathematically.

In chapter 2 we give an overview of core results regarding tilings and prototile sets. We give proofs of the Extension Theorem and state formally the first of our core tiling problems - the Domino Problem. We then give a proof of the undecidability of the Domino Problem by means of the computable conversion of any Turing Machine into a set of prototiles in such a way that their tilings tiling the plane iff the given Turing Machine on input $x$ does not halt. 

We introduce here the notion of a tile schema - a way of describing specific placement of colours from chosen colour sets. This allows us to describe (infinite) prototile sets by means of carefully chosen colour sets and schema tile construction such that the resultant product of combining these gives prototile sets whose tilings carry the specific properties we are looking for. Though this method may seem convoluted \textit{prima facie}, we hope to demonstrate that this technique leads in fact to quite straightforward proofs for translating various principles and concepts into the combinatorial properties of a prototile set. 

We round off this chapter by noticing some interesting corollaries and propositions arising from this fact that are of similar ilk to other results in mathematical logic - principally the fact that there exist prototile sets such that their domino problem is undecidable by Peano Arithmetic. 

In chapter 3 we state the first run of our main results - $\Pi^1_1$- and $\Sigma^1_1$-completeness of specific domino problems. We consider domino problems that require all tilings to be total, as well as domino problems that do not require total tilings, but instead only require an infinite connected patch of the plane to be tiled. To prove these results of $\Pi^1_1$ and $\Sigma^1_1$ completeness, we utilise the completeness for these classes due to wellfounded and illfounded trees. We construct tile schemas for each, and then demonstrate the completeness by means of $m$-reductions between our classes of prototile sets and ill-/well-founded trees.

With Chapter 4 we depart from domino problems, and instead consider the problems regarding whether or not the tilings for a given prototile set are all periodic, all aperiodic, or some mixture of the two. We state the fundamental results, with background references provided for this rather interesting class of problems.

We demonstrate that these notions are simultaneously $\Pi^1_1$ and $\Sigma^1_1$, as well as prove that, in fact, the questions of periodicity and aperiodicity for infinite sets of prototiles are both complete for the class of problems of the form $(\Pi^1_1 \wedge \Sigma^1_1)$. We also show that the set of all finite prototile sets whose tilings are aperiodic is $\Pi^0_1$, which is a surprising result.

Chapter 5 is an extension of this notion of computable reductions into the realm of Weihrauch reducibility. We give a feature rich presentation of the definitions and notions of Weihrauch reducibility, and state some core results. We then give intuitions for the core concepts in this theory, and proceed to derive Weihrauch equivalences between domino problems and closed choice on Baire space. 

Intuitively these results are motivated by realisation that all Wang tilings can be given by `tiling trees', first defined by Wang, for which closed choice realizers in Baire space can locate the infinite paths through, and from which we can recover a tiling of the plane. We can also consider that, given a non-deterministic prototile set - that is, for any prototile in the set there exist multiple possibilities for matching tiles in a given tiling - then having some choice principle in play is a natural conclusion. We give some exact results by means of Weihrauch equivalences.

The proposal for a new way of coding Elementary Cellular Automata (ECAs) into prototile sets is the subject of Chapter 6. Here, we demonstrate that for the 3-ary functions defining the behaviour of ECAs is naturally coded by a hexagon and lozenge based construction. With the requisite tiles to neaten up the upper edge of our tiling, we have a prototile set consisting of 15 tiles that very naturally give a way to represent the behaviour of ECAs in tilings of the half-plane by means of coding the first `input' row, and then making it such that the subsequent tilings of each row are exactly given by the underlying function of the given ECA.

We also show that such a prototile set is necessarily then chaotic and Turing Complete given correct choices for the ECA rule that we encode - Rule 30 and Rule 100 respectively for these results. Thus we have a nice and very small prototile set that carries with it a lot of possible mathematical capability.

Finally, we complete the thesis with an overview in Chapter 7 of the various open problems that we have found along the way - both in the literature and in the course of our research. We also aim to indicate the possible avenues for extending the results in this thesis further.

\subsection*{Summary of Original Work}

In this thesis, the following items are our original contributions:

\begin{itemize}
\item Our proof of theorem \ref{thm:TMTilings} is inspired by the form in \cite{Boas97Conv}, but is reshaped to match the structure of our later proofs. The observations leading up to corollary \ref{cor:PAGoodNotTile} have not been found in the literature, but are relatively straightforward to derive.
\item The results given in Chapter 3 are all original unless stated otherwise. Specifically, our main results are:
\begin{itemize}
\item Lemma \ref{lemma-pproc}
\item Theorem \ref{thm:TILE-ILL}
\item Theorem \ref{thm:nTILE-WELL}
\item Theorem \ref{thm:SNT-WELL}
\item Theorem \ref{thm:WTILE-ILL}
\end{itemize}
\item The results concerning $ATile$, $PTile$, $ATile_{FIN}$ and $PTile_{FIN}$ in Chapter 4 are all original:
\begin{itemize}
\item Theorem \ref{thm:ILL-ATile}
\item Theorem \ref{thm:ILL-PTile}
\item Theorem \ref{thm:WELL-PTile}
\item Theorem \ref{thm:WELL-ATile}
\item Theorem \ref{thm:X-ATile}
\item Theorem \ref{thm:X-PTile}
\item Corollary \ref{cor:ATilePTileComplete}
\item Theorem \ref{thm:ATileFIN-Pi01}
\item Theorem \ref{thm:PTileFINPi11}
\end{itemize}
\item The Weihrauch reductions for tiling problems in Chapter 5 are original:
\begin{itemize}
\item Theorem \ref{thm:CT-SW-ClosedC}
\item Theorem \ref{thm:CWPT-SW-ClosedC}
\item Theorem \ref{thm:CC-WKLstarCIPT}
\item Theorem \ref{thm:DPW-CC}
\end{itemize}
\item The main result in Chapter 6 is also original: Theorem \ref{thm:ECA-15-Hex}
\end{itemize}

\chapter*{Glossary of Sets and Constructions}
\label{chap:Glossary}
\setcounter{equation}{0}
\renewcommand{\theequation}{\thechapter.\arabic{equation}}

We give a table that details all of the major sets and operators that are used in this thesis, for convenience and for reference.

\begin{longtable}{p{2.5cm}|p{8.25cm}|p{1.75cm}}
\textbf{Name} & \textbf{Description} & \textbf{Thesis Ref.} \\
 \midrule
$m$-reducibility & Given two sets $A$ and $B$, $A$ is $m$-reducible to $B$, written $A \leq_m B$, if there exists some computable function $f: \omega \rightarrow \omega$ such that for all $x \in \omega$, $x \in A \iff f(x) \in B$ & \ref{def:mred}\\
\hline
Weihrauch Reducibility & Given two operators $f$ and $g$ on represented spaces, we say $f \leq_W g$, if there exist computable $H,K :\subseteq \omega^\omega \rightarrow \omega^\omega$ such that for any realizer $G \vdash g$, $F = K\langle id_{\omega^\omega}, GH \rangle$ is a realizer for $f$. & \ref{def:Weihrauch} \\
\hline
$WELL$ & The set of all indices $e$ such that $\varphi_e$ is the characteristic function of a well-founded tree $T \subseteq \omega^{<\omega}$. & \ref{def:WELL}\\
\hline
$ILL$ & The set of all indices $e$ such that $\varphi_e$ is the characteristic function of an ill-founded tree $T \subseteq \omega^{<\omega}$. & \ref{def:ILL}\\
\hline
$TILE$ & The set of all indices $e$ such that $\varphi_e$ is the characteristic function of an infinite Wang prototile set whose tilings are total in the plane. & \ref{def:TILE} \\
\hline
$WTILE$ & The set of all indices $e$ such that $\varphi_e$ is the characteristic function of an infinite Wang prototile set whose tilings are infinite, connected, but not necessarily total in the plane. & \ref{def:WTILE} \\
\hline
$SNT$ & The set of all indices $e$ such that $\varphi_e$ is the characteristic function of an infinite Wang prototile set whose connected tilings are all finite. & \ref{def:SNT}\\
\hline
$ATile$ & Set of all $e$ such that $\varphi_e$ is the characteristic function for a set of prototiles who planar tilings are all total and aperiodic. & \ref{def:ATile} \\
\hline
$PTile$ & Set of all $e$ such that $\varphi_e$ is the characteristic function for a set of prototiles who planar tilings are all total and periodic. & \ref{def:PTile} \\
\hline
$ATile_{FIN}$ & Set of all $e$ such that $\varphi_e$ is the characteristic function for a \emph{finite} set of prototiles who planar tilings are all total and aperiodic. & \ref{def:ATileFIN}\\
\hline
$PTile_{FIN}$ & Set of all $e$ such that $\varphi_e$ is the characteristic function for a \emph{finite} set of prototiles who planar tilings are all total and periodic. & \ref{def:PTileFIN}\\
\hline
AIT & The construction found in the proof of theorem \ref{thm:TILE-ILL} that creates an aperiodic prototile set given an ill-founded tree. & \ref{def:AITPIT}\\
\hline
PIT & The construction found in the proof of theorem \ref{thm:ILL-PTile} that creates an aperiodic prototile set given an ill-founded tree. & \ref{def:AITPIT}\\
\hline
$CT$ & The operator that takes some set of Wang prototiles as input and returns a total tiling of the plane. & \ref{def:ChooseTiling} \\
\hline
$CWPT$ & An operator that takes a set of Wang prototiles and returns a connected planar, but not necessarily total tiling. & \ref{ref:CWPT} \\
\hline
$CIPT$ & An operator that takes a prototile set $S$ that has total planar tilings, and returns an infinite `slice' of this tiling as a tiling of an infintie region of $\mathbb{Z}^2$. & \ref{def:CIPT} \\
\hline
$WIPT$ & An operator that takes a set of prototiles and return a tiling that has an infinite patch tiled within it, but we do not know where it is. & \ref{def:WIPT} \\
\hline
$DPW$ & The $DPW$ operator takes some set of prototiles and return a tiling that has an infinite connected patch within it. & \ref{def:DPW} \\
\hline
$C_{\omega^\omega}$ & Closed choice on Baire space - equivalent to finding a path through an ill-founded Baire space tree. & \ref{def:ClosedChoice} \\
\hline
$C_{2^\omega}$ & Closed choice on Cantor space - equivalent to Weak K\"onig's Lemma. & Sec \ref{sec:FurtherWeakTilingProblems} \\
\hline
$C_{\omega}$ & closed choice on the natural numbers - this takes a function $f: \omega \rightarrow \omega$ such that $range(f) \neq \omega$, and returns some point $n \notin range(f)$. & Sec. \ref{sec:FurtherWeakTilingProblems} \\
\end{longtable}

\chapter{Computability, Trees, and Preliminary Concepts}
\label{chap2}
\setcounter{equation}{0}
\renewcommand{\theequation}{\thechapter.\arabic{equation}}


\epigraph{The Analytical Engine has no pretensions whatever to originate anything. It can do whatever we know how to order it to perform…But it is likely to exert an indirect and reciprocal influence on science itself.}{\textit{Ada Lovelace, \\ in a Letter to Charles Babbage}}

In this chapter we will present the background theory for the rest of this volume. We will give definitions, theorems, and select proofs to lay the logical and mathematical groundwork for later chapters. 

\section{Preliminaries}

We will use the following standard notation throughout this work:

\begin{definition}
We shall make use of the standard logical notation:
\begin{itemize}
\item $\forall x$ and $\exists x$ for `for all $x$' and `there exists $x$' respectively.
\item $x \wedge y$ and $x \vee y$ for logical `$x$ AND $y$' and `$x$ OR $y$' respectively.
\item In general, variables and constants will be in lower case Roman lettering: $a,b,c,x,y,z,\ldots$
\item Lower case Roman letters such as $f,g,h,s,t,\ldots$ can also be used for function names.
\item In general, sets will be in upper case Roman lettering: $X, Y, Z, \ldots$
\item We shall use $A \rightarrow B$ to denote logical implication.
\item We shall use $A \cap B$ and $A \cup B$ to denote set intersection and union of $A$ and $B$.
\item We shall use $A \setminus B$ to denote the set $A$ with any elements found in $B$ removed, the standard set-minus.
\item Greek letters $\alpha, \beta, \gamma, \ldots$ shall be used primarily for ordinals, with the exception of $\varphi$ which is used for Turing Machines.
\item We shall use $\mathbb{N}, \mathbb{Z}, \mathbb{Q}, \mathbb{R}$ to mean the natural numbers, integers, rationals, and reals respectively.
\item For a given set $A$, let $\mathcal{P}(A)$ denote the \emph{powerset} of $A$ - the set of all subsets of $A$.
\item Unless otherwise indicated, our computable functions will be of the form $f: \omega \rightarrow \omega$.
\end{itemize}
\end{definition}

\begin{definition}[Cantor Pairing Function]
We shall use the standard Cantor pairing function to represent ordered pairs $\langle x,y \rangle$ as follows: \[ \langle x,y \rangle = \frac{(x+y)(x+y+1)}{2} + y \]
\end{definition}

We will shorten the notation for ordered $n$-tuples as $\langle x_1, x_2, \ldots, x_n \rangle$, with $\langle x,y,z \rangle = \langle \langle x,y \rangle , z \rangle$, and so forth. We fix this coding for the duration of this thesis, which will serve our definition of `computable' later.

We denote the set of natural numbers by its ordinal notation $\omega$, allowing for $\mathbb{N}$ to be used where it will avoid confusion.

\section{Computability}\label{sec:Computability}\index{Computability Theory}

We will use standard definitions, using \cite{Cooper2003computability} as our main reference text.

\begin{definition}\label{def:CompRel}
Let a \emi{computable relation} $R_e \subseteq \omega \times \omega)$ be a computable relation such that for some Turing Machine $e$, $$ R(x,y) \iff (\exists y) \varphi_e(x) = y $$ 
\end{definition}

\begin{definition}[First Layer of the Arithmetical Hierarchy]\label{def:CompRel}
We define the following notation for logical complexity of formulas as follows:
\begin{itemize}
\item If for all $x \in \omega$ we have $$ x \in A \iff (\exists y)R(x,y) $$ for a computable relation $R$, then we say that $A$ is a $\Sigma^0_1$ set, or $A \in \Sigma^0_1$.
\item If for all $x \in \omega$ we have $$ x \in A \iff (\forall y)R(x,y) $$ for a computable relation $R$, then we say that $A$ is a $\Pi^0_1$ set, or $A \in \Pi^0_1$.
\item If $A \in \Sigma^0_1 \cap \Pi^0_1$ then we say that $A$ is $\Delta^0_1$, or write $A \in \Delta^0_1$.
\end{itemize}
\end{definition}

Note that we rely on alternating existential/universal quantifiers, called \emph{prenex normal form}, in the structure of our formulae to properly ascertain which layer of any hierarchy we are at. Given this arithmetical hierarchy, we will later denote the `analytic' (also called `inductive') hierarchy in the same way, with a superscript of $1$ - $\Pi^1_1$, $\Sigma^1_1$, and $\Delta^1_1$. We will also find the following definition useful:

\begin{definition}[Skolem/Herbrand Normal Form]\index{Skolem/Herbrand normal form}
In the simplest form that we require in this thesis, a function is in Skolem (Herbrand) normal form if all of the existentially (universally) quantified terms are replaced by functions that take the preceding universally (existentially) quantified variable as input.

We always begin with formulae in prenex normal form. An example of Skolemisation is taking $$\forall x \, \exists y \, \forall z \, [P(x,y,z)]$$ and producing $$\forall x \, \forall z \, [P(x,f(x),z)]$$ for some \emi{Skolem function} $f$. Likewise, Herbrandization is taking some formula $$\exists x \, \forall y \, \exists z \, [P(x,y,z)]$$ and producing some $$\exists x \, \exists z \, [P(x,g(x),z)]$$ for some \emph{Herbrand function} $g$.
\end{definition}

\subsection{Turing Machines}\label{subsec:TMs}\index{Turing Machine}

We define a Turing Machine as follows:

\begin{definition}\label{def:TM}
A \emi{Turing Machine} (abbreviated to `TM') consists of a bi-infinite row of cells called the `tape', upon which are written symbols according to a `program' $P$ held in the TM `head' that moves sequentially along the tape. A program is a set of 5-tuples of the following form: $$ ( s,q,s',q',\{ L, R \} ) $$ where $s$ and $q$ are respectively the current symbol and state, $s'$ is the symbol to be written in place of $s$, and $q'$ is the next internal state for the TM to switch to. The final item instructs the head to move left or right, denoted $L$ or $R$ respectively.
\end{definition}

Before a TM is run, we set the input in symbols on the tape, set the head at position 0, and set the internal state to the starting state denoted $q_0$. We then allow the TM to operate along the input on the tape according to its program $P$.

Let us denote $\varphi_e(x)$ as the $e^{\text{th}}$ TM, under some chosen, effective enumeration of all possible Turing Machines, acting on input $x$. We say that our computation halts if we reach the reserved halting state, after which no more computation is performed. If such a computation halts, whatever is on the tape when it halts is considered the output. If the $e^{\text{th}}$ TM halts on input $x$ with output $y$, we write this $\varphi_e(x) \downarrow = y$. Where $\varphi_e(x)$ does not halt, we write $\varphi_e(x) \uparrow$.

\begin{definition}
A function $f : \omega \rightarrow \omega$ is \emi{computable} if there exists some $e$ s.t. $f = \varphi_e$.
\end{definition}

\begin{definition}[Halting Problem]\label{def:HP}\index{halting problem}
For any given TM $\varphi_e$ and some input $x$, is there a decidable method of determining if $\varphi_e(x)$ halts?
\end{definition}

\begin{definition}
There exists a Turing machine $U$ - the \emi{Universal Turing Machine} - which if given input $(e,n)$ can simulate $\varphi_e(n)$. That is to say, $\varphi_U(e,n) = \varphi_e(n)$.
\end{definition}

Alan Turing introduced these concepts in \cite{Turing1936a}, and determined that it the Halting Problem was in fact \emph{undecidable}, meaning that there is no universal Turing Machine that can decide it.

\subsection{Enumeration in Stages}

Given the discrete way in which we formulate Turing Machines, it is natural to press `stop' every now and again and see how our computation might be going. To do this, we can talk of successive stages of a computation, and the current configuration of the Turing Machine's tape at that particular point.

\begin{definition}\label{def:stagecomp}
For any TM $\varphi_e$:
\begin{itemize}
\item Let $\varphi_{e,s} (x)$ denote the computation $\varphi_e(x)$ carried out up to stage $s$.
\item Let $C_{e,s}$ denote the bi-infinite sequence corresponding to the tape configuration of $\varphi_e(x)$ at stage $s$ of the computation.
\end{itemize}
\end{definition}

\begin{theorem}[\protect{\cite[Thm. 5.2.10]{Cooper2003computability}}]\label{thm:DownarrowThm}
For any computation $\varphi_e(x)$, $$\varphi_e(x) \downarrow \iff (\exists s) \varphi_{e,s}(x) \text{ is in the HALT state} $$
\end{theorem}

\begin{proof}
If our computation has halted, then it has managed to reach the `HALT' state in the program. This necessarily means that a finite number of steps has been carried out before we halt. Thus, $s$ exists. 
\end{proof}

This gives the following corollary immediately:

\begin{corollary}[\protect{\cite[E. 5.2.14]{Cooper2003computability}}]
For any $e$, $\{ x : (\exists s) \varphi_{e,s}(x) \downarrow \}$ is a $\Sigma^0_1$ set.
\end{corollary}

\subsection{Core Background to Computability Theory}\label{subsec:ComputabilityTheory}\index{computability theory}

Computability Theory arose out of the work of G\"{o}del, Church, Turing, Kleene, P\'{e}ter, and Post - their foundational papers are collected in \cite{Davis:2004}. A core thematic idea arising out of this study, originally called `Recursion theory', was the {\em Church-Turing Thesis} defined as follows in \cite[p.42]{Cooper2003computability}:

\begin{definition}[Church-Turing Thesis]
For a given function $f$: $$ f \text{ is effectively computable } \iff f \text{ is recursive } \iff f \text{ is Turing computable.}$$
\end{definition}

This states that any algorithm we can come up with can be performed on a Turing Machine. As Cooper points out in section 2.5 in \cite{Cooper2003computability}, this gives us the security that our intuition for computability is matched with relevant details when it is needed. 

We can extend idea of what is computable to sets and trees, which we can initiate with the following definitions.

\begin{definition}
Let $\chi_A$ denote the \emi{characteristic function} of a set $A \subseteq \omega$.
\end{definition}

\begin{definition}
A set $A$ is \emi{computable} if the characteristic function $\chi_A$ is computable.
\end{definition}

That is to say that a set $A \subseteq \omega$ is computable if there exists $e$ such that for each $x \in \omega$

$$ \varphi_e(x) = \begin{cases}
			0 	& x \notin A \\
			1 	& x \in A  
			\end{cases}$$

We can also define what it is for a set to be \emi{computably enumerable}:

\begin{definition}[Computably Enumerable Sets]
We say that a set $A$ is \emi{computably enumerable}, or \emph{c.e.}, if $A = \emptyset$ or for some computable $f$, \[ A = range(f) = \{ f(0), f(1), f(2), \ldots \} \]
\end{definition}

There is an early result due to Post (see \cite[p.72]{Cooper2003computability}):

\begin{theorem}[\protect{\cite[Thm. 5.1.5]{Cooper2003computability}}]\label{thm:comp2ce}
If $A \subseteq \omega$ is computable, then $A$ is c.e.
\end{theorem}

\begin{proof}
Let $A$ be computable. Then we have a computable characteristic function $\chi_A$ that can decide for any $x \in \omega$ the question ``$x \in A$?", meaning there is a code $i$ such that $\varphi_i = \chi_A$. 

Given this $i$, we construct a Turing Machine that contains the machine given by $i$ and recursively answers the questions ``$0 \in A$?", ``$1 \in A$?", $\ldots$ in succession. For each positive answer to ``$x \in A$?" we enumerate $x$ into $A$, giving our result.
\end{proof}

In a similar way, we can prove other basic results, such as:

\begin{theorem}[\protect{\cite[Thm. 5.1.7]{Cooper2003computability}}]
$A$ is computable if and only if both $A$ and $\overline{A}$ are c.e.
\end{theorem}

\begin{proof}
($\rightarrow$) \, This follows from \ref{thm:comp2ce} above.
\newline
($\leftarrow$) \, If both $A$ and $\overline{A}$ are computably enumerable by computable functions $f$ and $g$ respectively, then we can construct $\chi_A$ by means of a TM that for all $x \in \omega$ computes both $f(x)$ and $g(x)$. Clearly one of these will give an answer, as both sets are c.e., and so $\chi_A$ is computable.
\end{proof}

However, the inverse arguments fail, which is where computability theory starts to get much more interesting.

\begin{theorem}[\protect{\cite[Thm. 5.3.1]{Cooper2003computability}}]
There exists a computably enumerable set that is is not computable.
\end{theorem}

We first define Post's Set:
\begin{definition}[Post's Set]\index{Post's Set}\label{PostsSet}
Let $ K = \{ e : \varphi_e(e)\downarrow \} $.
\end{definition}

\begin{proof}
We first note that Post's set $K$ is $\Sigma^0_1$, and thereby computably enumerable, as $$ e \in K \iff e \in W_e \iff \exists s \, \varphi_{e,s}(e) \downarrow .$$

However, to see that $K$ is not computable, it suffices to show that $\overline{K}$ is not computably enumerable. To see this, let $\overline{K}$ be computably enumerable for contradiction. Then $K = W_i$ for some $i \in \omega$, giving $$ x \in W_i \iff x \in \overline{K} \iff x \notin K \iff x \notin W_x $$ For $x=e$ this forces a contradiction by forcing different answers for ``$j \in \overline{K}$"? and ``$j \in W_j$?" for all $j \in \omega$.
\end{proof}

\subsection{Conventional theorems in Computability}

There are two standard, and very important theorems in computability - the $s$-$m$-$n$ theorem, and the recursion theorem, which we will give brief exposition and proofs of. These statements and proofs are based on \cite{Cooper2003computability} and \cite{Soare1987}.

\begin{theorem}[$s$-$m$-$n$ Theorem, \protect{\cite[Thm. 4.2.6]{Cooper2003computability}}]\label{thm:smn}
For every $m,n \geq 1$ there exists a 1-1 computable function $s^m_n$ of $m+1$ variables, such that for all $x, y_1, y_2, \ldots y_m$: $$ \varphi^{(n)}_{s^m_n(x,y_1,y_2,\ldots,y_m)} = \lambda z_1,z_2,\ldots,z_n [ \varphi^{m+n}_x(y_1,\ldots,y_m,z_1,\ldots,z_n) ] $$
\end{theorem}

Here, the notation of $\varphi_x^y$ denotes the machine with index $x$ that takes $y$-many inputs. This theorem is the only time we shall use this notation in this thesis - later, the subscript shall be used to denote Oracle sets.

Note, here we use the standard $\lambda$-notation for the substitution of $z_1, z_2, \ldots$ into our computable function. The notation of $m$ and $n$ in $s^m_n$ denote the number of parameters into the computable function $s$.

\begin{proof}[Proof sketch]
For $m=n=1$, let the TM $\varphi_{s^1_1(x,y)}(z)$ obtain $\varphi_x$, and then apply $\varphi_x(y,z)$. Such an $s = s^1_1$ is computable, as it is some effective procedure on $x$ and $y$. If it is not 1-1, then we can make it so by `padding' the process, and then letting the resultant $s'$ be s.t. $\varphi_{s(x,y)} = \varphi_{s'(x,y)}$, ordering our inputs $\langle x,y \rangle$ using a standard pairing function.
\end{proof}

\begin{theorem}[Kleene's Recursion Theorem, \protect{\cite[Thm. 4.4.1]{Cooper2003computability}}]\label{thm:recursion}
For every computable function $f$ there exists an $n$ - called the \emi{fixed point} of $f$ - s.t. \[ \varphi_n = \varphi_{f(n)} \]
\end{theorem}

\begin{proof}
Define the `diagonal' function $d(u)$ as follows: 
\[ \varphi_{d(u)} = 
\begin{cases}
	\varphi_{\varphi_u(u)} & \text{if } \varphi_u(u)\downarrow \\
	\uparrow & \text{otherwise}
\end{cases} \]

Note that by \ref{thm:smn}, $d(u)$ is 1-1 and total. $d$ is also independent of the $f$ that we are interested in.

For such a given $f$, let $i$ be the index given by \[ \varphi_i = f \circ d \]

\underline{Claim:} We claim that $n = d(i)$ is some fixed point for $f$. 

Note that, $f$ gives that $\varphi_i$ is total (as $d$ is total, above), so $\varphi_{d(i)} = \varphi_{\varphi_i(i)}$. Thus our result follows from the following equivalences: \[ \varphi_n = \varphi_{d(i)} = \varphi_{\varphi_i(i)} = \varphi_{fd(i)} = \varphi_{f(n)} \]
\end{proof}

In the previous proof, we constructed a function we described as \emph{diagonal}. Let \emi{diagonalization}, the construction of a diagonal function, be as follows: let $e$ be the index of $\varphi_e$, which we diagonalise $e$ by running $\varphi_e(e)$. 

This technique was first introduced by G\"odel in \cite{godel31} to give us unprovable statements, and was later used by Turing in \cite{Turing1936a} in relation to proving the non-computability of the halting problem. The set of \emph{Diagonally Non-Recursive} functions, or \emi{DNR}, is composed of all the computable functions $f$ such that $f(e) \neq \varphi_e(e)$ for all $e$, and is the subject of current study in modern mathematical logic. A thorough introduction and treatment can be found in \cite{Hirschfeldt2014}. 

We can also note that the numbers for which $\varphi_n = \varphi_{f(n)}$ need not be unique for any given $f$. 

\subsection{Computable Notions of Reducibility}

In speaking about computability, we often want to relativise two sets between each other. To do this, we will need the following definitions. We will begin with a more basic form of reducibility, called $m$-reducibility. This is defined in \cite{Cooper2003computability} as follows:

\begin{definition}[$m$-Reducibility]\label{def:mred}
Given set $A$ and $B$, we say that $A$ is \emi{$m$-reducible} to $B$, written $A \leq_m B$, if there is a computable function $f : \omega \rightarrow \omega$ such that for all $x \in \omega$: \[ x \in A \iff f(x) \in B \] If our function $f$ is injective, we say that $A$ is \emi{1-reducible} to $B$, written $A \leq_1 B$.
\end{definition}

Although $m$-reducibility was introduced \emph{after} Turing reducibility (see \ref{def:Turingred} below), it is a slightly easier-to-formulate version of reducibility between two sets. Cooper in \cite[p.103]{Cooper2003computability} gives the intuition for $m$-reducibility as $A$ being in some sense ``at least as computable" as $B$. 

From the definition \ref{def:mred} above, we can derive that \[ A \leq_m B \iff \overline{A} \leq_m \overline{B} \] which follows from the fact that $A = f^{-1}(B)$, and following from a general fact about pre-images we get that $\overline{A} = f^{-1}(\overline{B})$

Additionally, we can prove relatively straightforward theorems that give a good flavour of how theorems around $m$-reducibility are carried out:

\begin{theorem}[\protect{\cite[Thm. 7.1.2]{Cooper2003computability}}]
The ordering $\leq_m$ is:
\begin{enumerate}
\item reflexive.
\item transitive.
\item if $A \leq_m B$ and $B$ is computable, then $A$ is computable.
\item if $A \leq_m B$ and $B$ is c.e., then $A$ is c.e.
\end{enumerate}
\end{theorem}

\begin{proof}
\item
\textbf{1. - Reflexive} Clearly $A \leq_m A$ as for all $x$, $f(x) = x$ is computable. \qed 
\newline
\textbf{2. - Transitive} Let $A \leq_m B$ be given by $f$, and $B \leq_m C$ be given by $g$. We can get $A \leq_m C$ by \[ x \in A \iff f(x) \in B \iff g(f(x)) \in C \] so $A \leq_m C$ by $g \circ f$. \qed
\newline
For the next two proofs, let $A \leq_m B$ by a computable $f$. \newline
\textbf{3.} If $B$ is computable, then $\chi_A = \chi_B \circ f$, which is computable. \qed
\newline
\textbf{4.} Let $B \in \Sigma^0_1$, with $$x \in B \iff \exists y R(x,y)$$ for a computable relation $R$. Then $$x \in A \iff \exists y R(f(x), y)$$ giving us immediately that $A \in \Sigma^0_1$ also.
\end{proof}

Following on from these normal forms, we can prove that not just computable sets are $\Sigma^0_1$, but also computably enumerable sets are $\Sigma^0_1$ complete. This important intuition will be complimented by successive results in later sections.

\begin{theorem}[\protect{\cite[Thm. 5.1.5]{Cooper2003computability}}]
The following are equivalent:
\begin{enumerate}
\item $A$ is c.e.,
\item $A \in \Sigma^0_1$.
\end{enumerate}
\end{theorem}

\begin{proof}
\item
\textbf{$1 \rightarrow 2$} Let $A$ be c.e. - if $A = \emptyset$, then $x \in A \iff \exists x (x = x+1)$. Let $A = range(f)$ for some computable function $f$. Then \[ x \in A \iff \exists s (f(s) = x) \] where $f$ is now a computable relation between $s$ and $x$. 
\newline
\textbf{$2 \rightarrow 1$} Let $A \in \Sigma^0_1$ such that there is a computable $R$ giving \[ \exists y R(x,y) \iff x \in A \] we then construct a TM $e$ such that on input $y$, it will search through all possible $x \in \omega$ and $R(x,y)$ (computable) with the following outcomes:
$$ \varphi_e(y) = 
\begin{cases}
	x & \text{if } R(x,y) \\
	\uparrow & \text{otherwise}
\end{cases} $$
	
Thus, $ (\exists x) \varphi_e(y) = x \iff x \in A$ with $A$ also being c.e.
\end{proof}

\subsection{Turing Reducibility and the Jump Operator}

Although $m$-reducibility is incredibly useful, we can generalise it to a notion of \emph{Turing reducibility} by means of the following definitions - first proposed by Turing in 1939, but following the outline in \cite{Cooper2003computability}.

\begin{definition}[Oracle Turing Machines]
We define an \emi{oracle Turing machine} to be a normal Turing machine, but with access to an extra tape - called the \emi{oracle} - and makes use of \emph{query quadruples} $(q_i, S_k, q_j, q_k)$ that allow the Turing machine to behave as follows. Let $\varphi_e^A(x)$ be the $e^{th}$ TM on input $x$ and oracle $A$:
\begin{itemize}
\item The TM computes as before until it encounters a query quadruple.
\item The TM, then in state $q_i$, will read the current value on the work tape, call it $n$, and then query the oracle tape to ask $\text{is } n \in A?$.
\item Depending on the output of the query, the TM will then:
\begin{itemize}
\item State $q_j$ if $n \in A$.
\item State $q_k$ if $n \notin A$.
\end{itemize}
\end{itemize}
\end{definition}

Note, this definition does not require our oracle sets to be computable nor enumerable - just that they are there. In fact, it is explicitly why oracle Turing Machines were introduced - in order to analyse questions like ``is the halting problem all there is?'' Essentially, we can now ask ``What can we compute knowing the characteristic function of a, not necessarily computable, set $A$?" This breakthrough from Turing allowed us to reason about problems `beyond' the halting problem, by talking about \emi{Turing reducibility}.

\begin{definition}[Turing Reducibility]\label{def:Turingred}
We say that a set $A$ is \emi{Turing reducible} to a set $B$, written $A \leq_T B$ if for some $e$, $$\chi_A = \varphi_e^B$$.
\end{definition}

It is worth noting, however, that Turing reducibility is finer than $m$-reducibility, as evidenced by the following result:

\begin{theorem}[\protect{\cite[Thm. 4.2.6]{Cooper2003computability}}]
There exists $A$ and $B$ s.t. $A \leq_T B$ but $A \nleq_m B$.
\end{theorem}

\begin{proof}
Consider $C$ a non-computable computably enumerable set, with $\overline{C}$ its compliment. It is clear that $$ C \leq_T \overline{C} $$ but as $C$ is non-computable, we also have that $$ C \nleq_m \overline{C} $$
\end{proof}

The outcome of Turing's work was the Turing hierarchy, which is defined by taking successive `jumps' which we define as follows.

\begin{definition}
Let $A, B$ be given sets:
\begin{itemize}
\item We write $A \equiv_T B$ if $A \leq_T B$ and $B \leq_T A$.
\item We define the \emi{Turing degree} - also called the \emph{degree of unsolvability} - for some $A \subseteq \omega$ to be \[ deg(A) =_{def} \{ X \subseteq \omega : X \equiv_T  A \} \] 
\item We can write $\boldsymbol{\mathcal{D}}$ for the collection of all such degrees, and can define the partial ordering $\leq$ on $\boldsymbol{\mathcal{D}}$ induced by $\leq_T$ as: \[ deg(B) \leq deg(A) \Longleftrightarrow_{def} B \leq_T A \]
\end{itemize}
\end{definition}

It follows from this and some other results that three is in fact a partial order on $\mathcal{D}$, however this is beyond the scope of this thesis. Returning to Post's set, $K$, we state the following theorems - omitting proofs that can be found in \cite{Cooper2003computability}. 

\begin{definition}
For $n,e \in \omega$, let $HALT = { (n,e) : \varphi_e(n)\downarrow}$.
\end{definition}

\begin{theorem}[\protect{\cite[Thm. 5.3.1]{Cooper2003computability}}]
$$ HALT \leq_T K $$
\end{theorem}

Thus, $K$ is incomputable, and so things that $K$ reduces to are also necessarily incomputable. We also need the following idea of \emph{index sets}.

\begin{definition}
Let $\mathcal{A}$ be a set of partial computable functions - or of computably enumerable sets. The \emi{index set} of $\mathcal{A}$ is then the set $A$ of all the indices of elements of $\mathcal{A}$.
\end{definition}

\begin{theorem}[Rice's Theorem, \protect{\cite[Thm. 7.1.11]{Cooper2003computability}}]\index{Rice's theorem}
If $A$ is an index set - with $A \neq \emptyset$ and $A \neq \omega$ - then $K \leq_m A$ or $K \leq_m \overline{A}$. 
\end{theorem}

This result gives us the following corollary:

\begin{corollary}[\protect{\cite[Cor. 7.1.12]{Cooper2003computability}}]
Every non-trivial index set is incomputable.
\end{corollary}

This gives us a window into the core intuition behind Rice's important result on computable functions - that every non-trivial semantic property is fundamentally undecidable, by means of $m$-reducibility of $K$ into index sets.

\section{Computable Trees}

We denote Cantor space by $2^{\omega}$, and Baire space by $\omega^{\omega}$. For any alphabet $\Sigma$, we denote the set of strings $\sigma = (\sigma(0), \sigma(1), \ldots, \sigma(n-1))$ of length $n$ by $\Sigma^n$. We denote the set of arbitrary length finite strings by $\Sigma^{< \omega}$, and similarly for Cantor space we use $2^{< \omega}$, and for Baire space we shall use $\omega^{< \omega}$.

Let $|\sigma |$ denote the length of the string $\sigma \in \Sigma^{< \omega}$. We denote the initial segment of $\sigma$ of length $n$ by $\sigma \upharpoonright n$. For $\sigma$ and $\tau$, where $| \sigma | = i$ and $| \tau | = j$, we write $\sigma^{\frown} \tau$ for the string $(\sigma(0), \sigma(1), \ldots \sigma(i-1), \tau(0), \tau(1), \ldots, \tau(j-1))$, which we call the \emi{concatenation} of $\sigma$ and $\tau$. We write $\tau \prec \sigma$ if $\tau$ is an \emi{initial segment}, or \emi{initial substring}, of $\sigma$ - that is, there is some $n < | \sigma |$ such that for all $0 \leq i \leq n$ it holds that $\tau(i) = \sigma(i)$.

\subsection{Trees and $\Pi^0_1$ Classes}

The source for this section is Cenzer's chapter titled ``$\Pi^0_1$ Classes in Computability Theory" in \cite{Cenzer-griffor1999handbook}. 

\begin{definition}\label{def:trees}
A \emi{tree} is a set $T \subset \Sigma^{< \omega}$ that is closed under initial segments. That is, for all $\tau \in \Sigma^{< \omega}$ such that $|\tau| \leq |\sigma |$ it is true that $ \forall \sigma \in T \, (\tau \prec \sigma \rightarrow \tau \in T)$.
\end{definition}

We say that $\sigma$ is a {\em successor} to some $\tau \in T$ if there exists some $s \in \Sigma^{< \omega}$ s.t. $\sigma = \tau^{\frown} s$. If $\sigma \in T$ is a successor of some $\tau \in T$ and $| \sigma | = | \tau | + 1$ we say that $\sigma$ is an \emi{immediate successor} of $\tau$. 

\begin{definition}
We say that a tree $T$ is \emi{finitely branching} if for every $\tau \in T$ there are finitely many immediate successors in $T$.
\end{definition}

For every $T \subset 2^{< \omega}$ or $T \subset \Sigma^{< \omega}$ (for a finite alphabet $\Sigma$), $T$ can only be finitely branching.

\begin{definition}
We will make use of the following definitions for paths through a tree $T$:
\begin{itemize}
\item An \emi{infinite path} through $T$ is a sequence $(x(0), x(1), \ldots)$ such that $x \upharpoonright n \in T$ for all $n \in \omega$.
\item Denote by $[T]$ the set of infinite paths through $T$.
\end{itemize}
\end{definition}

We also state what it is for a set to be a $\Pi^0_1$ class, which is congruent with earlier definitions of $\Pi^0_1$ sets we stated earlier.

\begin{definition}\index{$\Pi^0_1$ Classes}
\begin{itemize}
\item A formula is $\Delta_0$ if it is a primitive recursive function.
\item A set $X \subset \omega^{\omega}$ is a $\Pi^0_1$ class if there is a $\Delta_0$ formula $\varphi(n,x)$ in the language of first order arithmetic such that $ x \in X \iff (\forall n) \varphi(n,x)$.
\end{itemize}
\end{definition}

A definition of Primitive Recursive Functions as well as other definitions we use here can be found in Cooper \cite[Sec. 2.1 p.12]{Cooper2003computability}.

The $\Pi^0_1$ classes may be described topologically as effectively closed subsets of the product space $\omega^{\omega}$. Early results in the study of $\Pi^0_1$ classes were carried out by Kleene, who proved the Kleene basis theorem in 1943. Further work was carried out by Kreisel, Shoenfield, Jockush, Soare, \etal.

The topology on Baire space, $\omega^{\omega}$ is determined by a basis of intervals given by $I(\sigma) = \{ x : \sigma \prec x \} $. A subset $P \subset \omega^{\omega}$ is closed iff $P = [T]$ for some tree $T$, hence our description of $\Pi^0_1$ classes as effectively closed subsets of Baire space.

Note that each interval given by $I$ is also closed, thus we can describe the intervals as \emi{clopen}. Note also that for Cantor space, $2^{\omega}$, the clopen sets are just the finite unions of intervals. 

Given these definitions we can state the core intuition for a $\Pi^0_1$ class as a tree in terms of some fixed initial segment $\sigma$ for which the $\Pi^0_1$ class is the set of points that are all possible extensions of $\sigma$ - the cone of extensions above this fixed initial segment.

We now wish to formalise the relationship between $\Pi^0_1$ classes and trees by means of the following Lemma: 

\begin{lemma}[\protect{\cite[p.41,Lem. 1.1]{Cenzer-griffor1999handbook}}]\label{lemma:trees-rec-rels}
For any class $P \subset \omega^{\omega}$, the following are equivalent:
\begin{enumerate}
\item $P = [T]$ for some computable tree $T \subset \omega^{< \omega}$.
\item $P = [T]$ for some primitive recursive tree $T$.
\item $P = \{ x : \forall n (R(n,x) )\}$ for some computable relation $R_e$,
\item $P = [T]$ for some $\Pi^0_1$ tree $T \subset \omega^{< \omega}$.
\end{enumerate}
\end{lemma}

Recall our definitions of computable relation (definition \ref{def:ArithH}) and tree (definition \ref{def:trees}) above.

\begin{proof}
A proof of this can be found in \cite[p.41]{Cenzer-griffor1999handbook}.
\end{proof}

Armed with this characterisation, we can equate the enumeration of computable trees with effectively enumerated $\Pi^0_1$ classes, as demonstrated in the following lemma.

\begin{lemma}[\protect{\cite[p.41,Lem. 1.2]{Cenzer-griffor1999handbook}}]
There is a uniformly recursive sequence $T_e$ of primitive recursive trees such that, for every $\Pi^0_1$ class $P$, there is some $e$ such that it holds that $$P = [T_e]$$
\end{lemma}

\begin{proof}
Let $\pi_0, \pi_1, \ldots$ be a recursive enumeration of the primitive recursive functions such that $\pi_i : \omega \rightarrow \{ 0,1 \}$. Define the $e^{th}$ such tree by \[ \sigma \in T_e \iff (\forall \tau \preceq \sigma) \pi_e(\langle \tau \rangle_n ) = 1 \] where $\langle \tau \rangle_n = \langle n, (\tau(0), \tau(1),\ldots,\tau(n-1)) \rangle$.

$T_e$ is a tree, and if $T$ is a primitive recursive tree with characteristic function $\pi_e$, then $T = T_e$. By lemma \ref{lemma:trees-rec-rels}, every $\Pi^0_1$ class is thereby equal to one of the $[T_e]$. 
\end{proof}

\section{Kleene's $\mathcal{O}$ and $\Pi^1_1$-Completeness}

In this section we will outline results that give the relationship between well-founded trees and $\Sigma^1_1$-completeness. Our preliminary definitions are as follows.

Unless otherwise stated, the material in this section is based on \cite{Cooper2003computability} and \cite{sacks_2017}. 

\begin{definition}[Ordinal]\label{def:ordinal}
We define ordinals as follows:
\begin{itemize}
\item A \emi{totally ordered set} is a set $A$ with a relation $\leq$ such that the following hold:
\begin{itemize}
\item (Reflexivity) $\forall a \in A (a \leq a)$
\item (Antisymmetry) $(a \leq b \wedge b \leq a) \rightarrow a = b$
\item (Transitivity) $(a \leq b \wedge b \leq c) \rightarrow a \leq c$
\item (Comparability) $\forall a,b \in A (a \leq b \vee b \leq a)$
\end{itemize}
\item A \emi{well-ordered set} is a totally ordered set $A$ together with a relation $\leq$ such that every subset $S \subseteq A$ has a least element.
\item Two sets $A, B \subseteq \omega$ are said to be \emi{order isomorphic} iff there exists a bijection $f:A \rightarrow B$ between $A$ and $B$ such that for all $a_1, a_2 \in A$ \[ a_1 \leq a_2 \iff f(a_1) \leq f(a_2) \]
\item Two well-ordered sets $A,B \subseteq \omega$ have the same \emi{order type} iff they are order isomorphic.
\item An \emph{ordinal number} or \emi{ordinal} (in the language due to Cantor) is just an order type of some well-ordered set.
\end{itemize}
\end{definition}

\textbf{NB} - later, in definition \ref{def:well-ord}, we will formalise the difference between a totally-ordered and well-ordered set. Specifically that the well-foundedness of such as set forces the relation to be irreflexive and connected.

\begin{definition}[\bf{Ord}]
We denote the set of all ordinals - that is, the set of every possible order type - as \textbf{Ord}.
\end{definition}

We now have all the basic machinery we need to describe the computable, or recursive ordinals.

\subsection{Ordinal Notations and Kleene's $\mathcal{O}$}

The aim of Kleene's construction is to analyse the structure of the computable ordinals, by means of creating representations of each as natural numbers.

The resulting theory identified that the computable ordinals form an initial segment of \textbf{Ord}, sitting strictly below the least non-computable ordinal, which we shall call the \emph{Church-Kleene ordinal}, denoted $\omega_1^{CK}$.

We will begin this journey into categorising and enumerating the computable ordinals by first defining a way of formulating \emph{notations} for the ordinals. The core idea here is that we can construct things that represent ordinals - including successor ordinals and limit ordinals - but in a way that can be more easily manipulated and understood for our present purposes. 

\begin{definition}[Ordinal Notation Ordering]\index{ordinal notations}
We first define the ordering $<_\mathcal{O}$:
\begin{itemize}
\item If $x$ and $y$ are both notations for constructive ordinals, then let $x <_\mathcal{O} y$ be for ``$x$ is less than $y$ according to the ordering of notations."
\item Given an ordinal can have two different notations, $<_\mathcal{O}$ is not linear.
\end{itemize}
\end{definition}

We can regard $x <_\mathcal{O} y$ as a set of ordered pairs - thus it is the closure of a finite set $X$ under some $\Sigma^1_1$-closure condition $A(X)$ we we define below.

\begin{definition}
Let $X$ be a finite set, the closure condition $A(X)$ has three clauses:
\begin{enumerate}
\item $\forall u,v (\langle u,v \rangle \in X \rightarrow \langle v, 2^v \rangle \in X )$ (Successors)
\item $\forall n (\varphi_e(n) \downarrow \wedge \langle \varphi_e(n) , \varphi_e(n+1) \rangle \in X) \rightarrow \forall n ( \langle \varphi_e(n), 3 \cdot 5^e \rangle \in X)$ (Limits)
\item $\forall u,v,w (\langle u,v \rangle, \langle v,w \rangle \in X \rightarrow \langle u,w \rangle \in X) $ (Transitivity)
\end{enumerate}
\end{definition}

Thus, there is some least $X$ such that $\langle 1,2 \rangle \in X$, with $A(X)$. We let $<_\mathcal{O}$ be this least such $X$. 

\subsection{Kleene's $\mathcal{O}$}

We can now define Kleene's $\mathcal{O}$ as follows:

\begin{definition}[Kleene's $\mathcal{O}$]
Let $\mathcal{O}$ denote the set of notations for constructive ordinals. $\mathcal{O}$ forms the field of $<_\mathcal{O}$.
\end{definition}

We will use the following definition of notations, noting that they are all defined recursively for future purposes.

\begin{definition}
Let the function $| \cdot | : \mathcal{O} \rightarrow \text{\textbf{Ord}}$ be defined by transfinite recursion on $<_\mathcal{O}$ as follows:
\begin{align*}  |1| & =  0 \\
| 2^u | & =  |u| + 1 \\
| 3 \cdot 5^e | & = \lim_{n\to\infty} | \varphi_e(n) |
\end{align*}
\end{definition}

We can now define all of the constructive ordinals in the following manner.

\begin{definition}[Constructive Ordinals]
An ordinal $\delta \in \text{\textbf{Ord}}$ is a \emi{constructive ordinal} if $\delta = u$ for some $u \in \mathcal{O}$. 
\end{definition}

\subsection{Kleene's $\mathcal{O}$, and Well-foundedness}

We define well-foundedness as follows: 

\begin{definition}[Well-founded relations]\label{def:well-fdd}
A binary relation $R$ is \emi{well-founded} if there is no $f$ s.t. \[ \forall x (R(f(x+1),f(x)) \]
\end{definition}

We are now ready for the following theorem:

\begin{theorem}[\protect{\cite[Thm. 2.2]{sacks_2017}}]
\begin{enumerate}
\item $<_\mathcal{O}$ and $\mathcal{O}$ are $\Pi^1_1$
\item $<_\mathcal{O}$ is a well-founded partial ordering.
\item For $v \in \mathcal{O}$, the restriction of $<_\mathcal{O}$ to $\{ u | u <_\mathcal{O} v \}$ is linear.
\end{enumerate}	
\end{theorem}

Our proof comes directly from \cite{sacks_2017}.

\begin{proof}
\textbf{1.} A full proof of 1. can be found in \cite[p.9]{sacks_2017}
\newline
\textbf{2.} The following \emi{natural enumeration} of $<_\mathcal{O}$ is equivalent to a redefinition of $<_\mathcal{O}$ by means of transfinite recursion on ordinals, as follows:
\begin{itemize}
\item \textbf{Stage 0}: enumerate $1 <_\mathcal{O} 2$.
\item \textbf{Stage $\delta + 1$}: enumerate all $v <_\mathcal{O} 2^v$ and $u <_\mathcal{O} 2^v$ if $u <_\mathcal{O} v$ was enumerated at stage $\delta$.
\item \textbf{Stage $\lambda$ (limit)}: enumerate $\varphi_e(n) <_\mathcal{O} 3 \cdot 5^e$ and $u <_\mathcal{O} 3 \cdot 5^e$, if not enumerated at some earlier stage, if for each $n$ it holds that $\varphi_e(n) <_\mathcal{O} \varphi_e(n+1)$ was enumerated at an earlier stage, and if for some $n,u <_\mathcal{O} \varphi_e(n)$ was also enumerated at an earlier stage.
\end{itemize}

By induction on each stage $\gamma$, a pair enumerated at some stage $\gamma$ belongs to $<_\mathcal{O}$. On the other hand, the set of all pairs enumerated into $<_\mathcal{O}$ is a solution of $A(X)$, and so contains $<_\mathcal{O}$. 

By induction on $u <_\mathcal{O} v$ and $v <_\mathcal{O} w$, then $u <_\mathcal{O} v$ is enumerated at an earlier stage than $v <_\mathcal{O} w$. It then follows that $<_\mathcal{O}$ is well-founded, else there would otherwise be a descending infinite sequence of ordinals.

\leavevmode\\ \textbf{3.}
We prove this by induction on $<_\mathcal{O}$. Assume $u_1, u_2 <_\mathcal{O} v$, we check that one of the following hold:
\begin{itemize}
\item $u_1 <_\mathcal{O} u_2$,
\item $u_1 = u_2$, or
\item $u_2 <_\mathcal{O} u_1$.
\end{itemize}

If $v = 2^u$, then (1) above implies that $u_1, u_2 \leq_\mathcal{O} u$ and our result follows by induction. Else, if $v = 3 \cdot 5^e$, then we apply (2) to get the result.
\end{proof}
We can now prove the following facts about ordinal notations and their addition:

\begin{definition}\label{def:recI}
\begin{itemize}
\item Let $+_\mathcal{O}$ be such that if $a,b \in \mathcal{O}$, then $a +_\mathcal{O} b \in \mathcal{O}$ and $$|a +_\mathcal{O} b | = |a| + |b| $$
\item Let $h$ be a recursive function such that $$ \varphi_{h(e,a,d)} \simeq \varphi_e(a, \varphi_d(n)) $$
\item Let $I$ be a recursive function such that
\begin{align*}
a & \text{  if } b = 1 \\
\varphi_{I(e)}(a,b) \simeq 2^{\varphi_e(a,m)} & \text{  if } b = 2^m \\
3 \cdot 5^{h(e,a,d)} & \text{  if } b = 3 \cdot 5^d \\
7 & \text{  otherwise}
\end{align*}
\end{itemize}
\end{definition}

It is worth noting that, because our breaking up of $\mathcal{O}$ into notations for zero, successors, and limits is effective, $I$ above is recursive, even though $<_\mathcal{O}$ is non-recursive. Also, the clause for $I(e)$ is sensible even if $a,b \notin \mathcal{O}$.

\begin{theorem}[Kleene, \protect{\cite[Thm. 3.4]{sacks_2017}}]
The recursive function $+_\mathcal{O}$ has the following properties. For all $a$, and $b$:
\begin{enumerate}
\item $a,b \in \mathcal{O} \iff a +_\mathcal{O} b \in \mathcal{O}$.
\item $a,b \in \mathcal{O} \Rightarrow | a +_\mathcal{O} b | = |a| + |b|$.
\item $a,b \in \mathcal{O} \wedge b \neq 1 \Rightarrow a <_\mathcal{O} (a +_\mathcal{O} b)$.
\item $a \in \mathcal{O} \wedge c <_\mathcal{O} b \iff (a +_\mathcal{O} c ) <_\mathcal{O} (a +_\mathcal{O} b)$.
\item $a \in \mathcal{O} \wedge b=c \in \mathcal{O} \iff (a +_\mathcal{O} b) = (a +_\mathcal{O} c)$.
\end{enumerate}
\end{theorem}

\begin{proof}
Can be found in Sacks \cite{sacks_2017} I.3.4 p.13.
\end{proof}

Due to our computable approach, and the fact that our notations for ordinals are, in particular, very computable, we can get theorems such as the following:

\begin{definition}
Denote by $W_e$ the $e^{th}$ computably enumerable subset of $\omega$, the domain of $\varphi_e$.
\end{definition}

The intuition here is that $W_e$ is the `set of inputs that $\varphi_e$ halts on' - which is why we use the domain of $\varphi_e$ in our definition.

\begin{theorem}[Kleene, \protect{\cite[Thm. 3.5]{sacks_2017}}]\label{thm:pinit}
There exists a computable function $p$ such that for all $b \in \mathcal{O}$, $$W_{p(b)} = \{ a : a \leq_\mathcal{O} b \} $$
\end{theorem}

\begin{proof}
The required properties of $p$ are as follows:
\begin{align*}
	W_{p(1)} & = \emptyset \\ 
	W_{p(2^a)} & = \{a\} \cup W_{p(a)} \numberthis \label{C2:Eq1} \\
	W_{p(3 \cdot 5^d)} & = \bigcup\limits_{n \in \omega} \{ W_{p(\varphi_d(n))} : \varphi_d(n)\downarrow \} 
\end{align*}

By induction on $<_\mathcal{O}$ we get that any $p$ that satisfies all of \ref{C2:Eq1} will also satisfy our theorem. As such, we want to show the existence of such a computable $p$, specifically by means of effective transfinite recursion on $p$. Let $e_0$ be any G\"odel number for some TM, and let $i$ and $j$ be computable functions such that:
\begin{align*}
	W_{e_0} & = \emptyset \\
	W_{i(e,a)} & = \{a\} \cup W_{\varphi_e(a)} \numberthis \label{C2:Eq2} \\
	W_{j(e,d)} & = \bigcup\limits_{n \in \omega} \{ W_{\varphi_e(\varphi_d(n))} : n < \omega \}
\end{align*}

In \ref{C2:Eq2}, it is intended that when $\varphi_e(a) \uparrow$, that $W_{\varphi_e(a)} = W_{\varphi_e(\varphi_d(n))} = \emptyset$. We can now obtain a recursive $I$ (similar to \ref{def:recI} above) such that:
\begin{align*}
e_0 & \text{  if } b = 1 \\
\varphi_{I(e)}(b) \simeq i(e,a) & \text{  if } b = 2^a \\
j(e,d) & \text{  if } b = 3 \cdot 5^d \\
0 & \text{  otherwise}
\end{align*}

By theorem \ref{thm:recursion}, the fixed point theorem, $I$ necessarily has a fixed point $c$ where $\varphi_{I(c)} \simeq \varphi_c$. Let $p(b)$ be $\varphi_c(b)$. Then
\begin{align*}
e_0 & \text{  if } b = 1 \\
p(b) = i(e,a) & \text{  if } b = 2^a \\
j(e,d) & \text{  if } b = 3 \cdot 5^d \\
0 & \text{  otherwise}
\end{align*}

Given $i$ and $j$ are both computable and total, we get that \ref{C2:Eq2}$\rightarrow$\ref{C2:Eq1}.
\end{proof}
We are also able to obtain the existence of similar recursive functions:
\begin{theorem}[Kleene, \protect{\cite[Thm. 3.5]{sacks_2017}}]\label{thm:Kleeneq}
There exists a recursive function $q$ such that for all $b \in \mathcal{O}$, $$W_{q(b)} = \{ \langle x,y \rangle : x <_\mathcal{O} y <_\mathcal{O} b \} $$
\end{theorem}

\begin{proof}
Essentially the same as for the proof of theorem \ref{thm:pinit}, with the modification that we adjust the definition \ref{C2:Eq1} and \ref{C2:Eq2} preserve all of the pairs $\langle x,y \rangle$ s.t. $x <_\mathcal{O} y <_\mathcal{O} a$ in our recursive definition of $i$.
\end{proof}

\subsection{Recursive Ordinals and well-founded Relations}

We can now show that every c.e. subset of $\mathcal{O}$ is bounded in a ``highly effective manner."\cite[p.15]{sacks_2017} 

\begin{theorem}[\protect{\cite[Lem. 4.1]{sacks_2017}}]
There exists a computable $g$ such that for all $e$:
\begin{enumerate}
	\item $g(e) \in \mathcal{O} \iff W_e \subseteq \mathcal{O}$,
	\item $g(e) \in \mathcal{O} \Rightarrow |a| < |g(e)| \text{ for all } a \in W_e$.
\end{enumerate}
\end{theorem}

\begin{proof}
A proof can be found in Sacks \cite{sacks_2017} p.16.
\end{proof}

We now formalize our definition \ref{def:ordinal}.

\begin{definition}\label{def:well-ord}
A binary relation $R(x,y)$ is a \emi{well-ordering} if it is:
\begin{enumerate}
\item (Connected) $(\forall x,y) ( R(x,y) \vee R(y,x) \vee x=y )$
\item (Transitive) $(\forall x,y,z) ( R(x,y) \wedge R(y,z) \rightarrow R(x,z) )$
\item (well-founded) if $S \neq \emptyset$, $S$ is a subset of the field of $R$, then $\exists y \in S$ such that $(\forall x \in S) \neg R(x,y)$
\\ Note, that 3. implies:
\item (Irreflexive) $(\exists x) \neg R(x,x)$
\item (Antisymmetric) $(\forall x,y) (R(x,y) \rightarrow R(y,x))$
\end{enumerate}
\end{definition}

Given the well-foundedness of a well-ordering relation, we can define the \emi{height} of $R$ as follows:

\begin{definition}
\begin{itemize}
\item Let $R$ be a well-founded binary relation, then it has a height, denoted by $|R|$, measured by some ordinal. 
\item Let $\beta$ be an ordinal variable. $\mu \beta$ is then the ``least $\beta$ such that..."
\item $|x| = \mu \beta \, [ R(y,x) \rightarrow |y| < \beta]$
\item $|R| = \mu \beta \, \forall x \, [ x \in \text{ field of } R \rightarrow |x| < \beta ]$
\end{itemize}
\end{definition}

We can also enumerate computable relations:

\begin{definition}
Let $R_e$ denote $R_e(x,y) \iff \varphi_e(x,y)$.
\end{definition}

Thus, we can enumerate all computable relations. We shall let $$\textbf{Rel}=\{ R_e : e < \omega \}$$

\begin{lemma}[\protect{\cite[Lem. 4.3]{sacks_2017}}]\label{lemma:Rwellfdd}
There exists a computable $f$ such that, for all $e$:
\begin{itemize}
	\item $R_e \text{ is well-founded } \iff f(e) \in \mathcal{O}$, and
	\item $R_e \text{ is well-founded } \rightarrow |R_e| \leq |f(e)|$
\end{itemize}
\end{lemma}

This lemma gives rise to the following theorem due to Kleene and Markwald:

\begin{theorem}[Kleene-Markwald, \protect{\cite[Thm. 4.4]{sacks_2017}}]
The computable ordinals are equal to the constructive ordinals.
\end{theorem}

\begin{proof}
A proof can be found in Sacks \cite{sacks_2017} p. 18.
\end{proof}

\subsection{$\mathcal{O}$, Well-foundedness, and $\Pi^1_1$ Sets}

In this subsection, we will build on our theory and present the ordinal analysis of $\Pi^1_1$ Sets.

\begin{definition}
Let $\overline{f}(x) = \{ \langle i, f(i) \rangle : i < x \}$, essentially that, for $p_i$ being the $i^{th}$ prime, $p_0 = 2$: $$ \overline{f}(x) = \prod\limits_{i < x} p_i^{1+f(i)} $$ If $y = \overline{f}(x)$ for some $f$ and $x$, we say that $y$ is a \emi{sequence number}.
\end{definition}

This $\overline{f}(x)$ can be thought of as the code for the graph of $f \upharpoonright x$ - essentially, it is the code for the sequence $\langle f(0), f(1), \ldots, f(x-1) \rangle$, with $f(0) = 1$. We can denote the length of $x$ as $len(\overline{f}(x))$. We can thus view $y$ as $\langle y_0, y_1, \ldots , y_{len(y)-1} \rangle$.

If $y$ and $z$ are both sequence numbers, then we say that `$y$ is \emi{properly extended} by $z$', written $y \prec_{seq} z$ if $len(y) < len(z)$ and for all $i < len(y)$ we have that $y_i = z_i$.

\begin{definition}
Let \textbf{Seq} denote the set of all sequence numbers.
\end{definition}

\textbf{Seq} is a computable set, and $\prec_{seq}$ is a computable, antisymmetric, transitive binary relation. We can think of $(\textbf{Seq},\prec_{seq})$ as presenting Baire space, $\omega^\omega$ as a tree - which is why it is useful in the study of $\Pi^1_1$ sets.

We denote $S_R(y)$ to be the restriction of $(\textbf{Seq},\prec_{seq})$ to the sequence numbers $\overline{f}(x)$ such that $$S_R(y) = \forall i < x [ \neg R(\overline{f}(i),y) ] $$

The following proposition begins our connection between well-foundedness and formulae in the normal form $\Pi^1_1$:

\begin{proposition}[\protect{\cite[Prop. 5.3]{sacks_2017}}]\label{prop:wellfddR}
$\forall f \exists x (R(\overline{f}(x), y)) \iff S_R(y) \text{ is well-founded.}$
\end{proposition}

\begin{proof}
Fix some $y$.  $\neg (\forall x \, \exists x \, R(\overline{f}(x),y)$ if and only if there is some $f$ such that $\forall x \, \neg R(\overline{f}(x),y)$ if and only if there is some $f$ such that $f(0) > f(1) > f(2) > \ldots$ in an infinite descending sequence $S_R(y)$ if and only if $S_R(y)$ is not well-founded.
\end{proof}

We can now continue our analysis with the following normalisation of $\Pi^1_1$ predicates and theorems. We note that for any computable relation $R_1(f,x,y)$ we can find $e$ such that $\varphi_e^f(x,y) = 0 \iff R_1(f,x,y)$. Using this, we can prove the following, denoting by $WF$ the set of all well-founded trees - we now show the following lemma:

\begin{lemma}[\protect{\cite[Sec. 5.2]{sacks_2017}}]\label{lemma:wellfddTreesPi11}
For each $\Pi^1_1$ set P, $$ P \leq_m WF $$
\end{lemma}

\begin{proof}
Let some $B \in \Pi^1_1$. By the above, there is some computable $R$ such that for all $y$, $$ y \in B \iff \forall f \, \exists x \, R(\overline{f}(x),y) $$ By proposition \ref{prop:wellfddR}, we get $$ y \in B \iff S_R(y) \text{ is well-founded.} $$
\end{proof}

We can thus extend this lemma to a result due to Kleene:

\begin{theorem}[Kleene, \protect{\cite[Thm. 5.4]{sacks_2017}}]\label{thm:KleenePi11O}
For each $\Pi^1_1$ set $P$, $$P \leq_m \mathcal{O}$$
\end{theorem}

\begin{proof}
Let $B \in \Pi^1_1$. As for \ref{lemma:wellfddTreesPi11}, we have that there is some computable $R$ such that for all $y$, $$ y \in B \iff \forall f \, \exists x \, R(\overline{f}(x),y) $$ and again by \ref{prop:wellfddR}, we get $$ y \in B \iff S_R(y) \text{ is well-founded.} $$

Given the $S_R(y)$ is computable uniformly in $y$ - that is, we only require one TM with which to carry our the computation - we have that there exists a computable function $g$ such that $S_R(y) = R_{g(y)}$. Let $f$ be as in lemma \ref{lemma:Rwellfdd}, then $$ y \in B \iff f(g(y)) \in \mathcal{O} $$
\end{proof}

This gives us the following useful corollary:

\begin{corollary}[\protect{\cite[Cor. 5.5]{sacks_2017}}]\label{cor:OnotSigma11}
$\mathcal{O} \notin \Sigma^1_1$
\end{corollary}

\begin{proof}
This proof is structurally similar to one that a complete c.e. subset of $\omega$ is not computable.

We first note that for any set $S$ such that for some $A \in \Sigma^1_1$, if $S \leq_m A$, then $S$ must also be $\Sigma^1_1$. So, by \ref{prop:wellfddR}, we have that if $\mathcal{O}$ were $\Sigma^1_1$, then every $\Pi^1_1$ set would also be $\Sigma^1_1$. 

Thus it suffices that we can find some $A \in \Pi^1_1$ such that $A \notin \Sigma^1_1$. Define $Q(y)$ to be $\forall f \, \exists x \, \varphi_y^{f \upharpoonright x}(y)$. Suppose $( \neg Q(y) \in \Pi^1_1 )$, then $\neg Q(y)$ is equivalent to the statement that there exists $e$, $\forall f \, \exists x \, \varphi_e^{f \upharpoonright x}(y)$. So $\neg Q(y) \iff Q(y)$.
\end{proof}

We can now prove the result due to Spector about the $\Sigma^1_1$-boundedness of $\mathcal{O}$:

\begin{theorem}[Spector, 1955 - \protect{\cite[Cor. 5.6]{sacks_2017}}]
Let $X \subseteq \mathcal{O}$ and $X \in \Sigma^1_1$. There exists some $b \in \mathcal{O}$ such that $\forall x \in X \, (|x| \leq |b|)$.
\end{theorem}

\begin{proof}
As for the proof of \ref{thm:KleenePi11O}, we can replace $\mathcal{O}$ with $B$ and find a computable function $t$ such that for all $y$, $$ y \in \mathcal{O} \iff R_{t(y)} \text{ is well-founded} $$

Let our $Q(y)$ be $$ \exists z \, [ z \in X \wedge \exists f \, \forall u,v \, (R_{t(y)}(u,v) \rightarrow \langle f(u) , f(v) \rangle \in W_{q(z)}) ] $$ where $W_{q(z)}$ is as per \ref{thm:Kleeneq}. $Q(y) \in \Sigma^1_1$. If $Q(y)$ holds then we have $R_{t(y)}$ must be well-founded. 

Suppose that $b$ does not exist, then if $R_{t(y)}$ is well-founded, then by \ref{lemma:Rwellfdd} there is some $z \in X \subseteq \mathcal{O}$ such that $|R_{t(y)}| < |z|$, and thereby $Q(y)$ holds. But $y \in \mathcal{O}$ is $\Sigma^1_1$, contradicting \ref{cor:OnotSigma11}.
\end{proof}

We can thus get the following corollary:

\begin{corollary}[\protect{\cite[Ex. 5.7]{sacks_2017}}]\label{cor:CompTreesPi11}
The set of all well-founded computable trees is $\Pi^1_1$ complete.
\end{corollary}

It should be noted that in later chapters, \ref{lemma:wellfddTreesPi11} and \ref{cor:CompTreesPi11} will be particularly useful, as it will form the basis for our theorems that relate the well-foundedness of trees to tilings of the plane using infinite prototile sets (see Chap. 2 for definitions of these terms). 

\section{Trees, Ordinals, and the Arithmetical and Analytic Hierarchies}

We have outlined in previous sections various definitions that will be used in our work in later chapters. There are some deep and illuminating connections between these objects, which we hope to outline and illustrate in this section. Unless otherwise stated, all results can be found in \cite{Cenzer-griffor1999handbook} and \cite{Cooper2003computability}.

\subsection{Fundamental Results}\label{sec:KLemmaWKL}

Our formulation of K\"onig's lemma comes from \cite{Pudlak2013} and \cite{Kaye2007}.

\begin{lemma}[K\"onig's Lemma, \protect{\cite[Thm. 3.13]{Kaye2007}}]
Every infinite finitely branching tree has an infinite branch.
\end{lemma}

\begin{proof}
We prove this for $T$, a binary tree. For a string $\sigma$ with $|\sigma| = n$, let $$ T_\sigma =  \{ \tau \in T : \tau \upharpoonright n = \sigma \} \cup \{ \sigma \upharpoonright k : k < n \}$$ We shall call $T_\sigma$ the \emph{subtree of $T$ below $\sigma$}. Though it is easy to check that $T$ is a tree, it may or not be infinite. 

We want $\gamma \in T$ such that the tree $T_\gamma$ below $\gamma$ is infinite. Let this be our induction hypothesis. Suppose we have some $\gamma$, with $| \gamma | = n$ and $T_\gamma$ is infinite. Since our tree $T$ is binary, we have $$ T_\gamma = \{ \tau \in T:\tau \upharpoonright (n+1) = \gamma^\frown 0 \} \cup \{ \tau \in T: \tau \upharpoonright (n+1) = \gamma^\frown 1 \} \cup \{ \gamma \upharpoonright k : k \leq n \} $$

The third of these sets is clearly finite, so one of the first two - corresponding to '0' and '1'respectively - must be infinite, by our induction hypothesis. 

If the first of these is infinite, we set $\gamma(n+1) = \gamma^\frown 0$, and so we have $$ T_{\gamma(n+1)} = \{ \tau \in T: \tau \upharpoonright (n+1) = \gamma^\frown 0 \} \cup \{ \gamma^\frown 0  \} \cup \{ \gamma \upharpoonright k : k \leq n \} $$ which is infinite. In the other case, we do the same for $\gamma(n+1) = \gamma^\frown 1$, which gives us the same infinite tree $T_{\gamma(n+1)}$ as before.

In both cases, we have defined $\gamma(n+1)$ and proved our induction hypothesis for $n+1$.

\end{proof}

This lemma is rather famous throughout the mathematical \oe{}uvre - indeed, in other reference texts such as \cite{GrunbaumTP}, this theorem features in reference to the compactness of Wang tiles as ``K\"onig's Infinity Lemma". This is something we shall later make use of in proving this result in chapter 2.

K\"onig's Lemma applied to trees with a bound on the number of children for each node, then we say that this is \emph{Weak K\"onig's Lemma} (WKL). WKL is a very important principle studied in reverse mathematics, such as a compactness principle for Cantor space. This is not, however, within the scope of this thesis to study or present.

\subsection{Trees and Analytic Sets}

We start by defining the extendible nodes of a tree:

\begin{definition}
For a tree $T$, we define the set of \emi{extendible nodes} $Ext(T)$ by \[ \sigma \in Ext(T) \iff (\exists x) (x \in [T] \wedge \sigma \prec x) \]
\end{definition}

This definition allows us to collect all of the initial segments of the points $x$ that lie in some tree $T$. Our aim is to use this set to establish $Ext(T)$ as a basis for trees whose sets of paths are $\Pi^0_1$ sets. By this, we mean that any extension in $Ext(T)$ is a $\Pi^0_1$ set. We first establish what a basis is:

\begin{definition}
Let $\Theta \subseteq \mathcal{P}(\omega^\omega)$ be a collection of subclasses of $\omega^{\omega}$. A set $\Gamma \subset \omega^{\omega}$ is a \emi{basis} if every class $C \in \Theta$, there is some $x \in C$ such that $x \in \Gamma$.
\end{definition}

This gives us natural formulation for `basis theorems', such as the following extracted from \cite[p.52]{Cenzer-griffor1999handbook}.

\begin{theorem}[\protect{\cite[Remark p.51]{Cenzer-griffor1999handbook}}]
The class $\Delta^0_0$ of computable functions is a basis for the family of open subclasses of Baire space.
\end{theorem}

However, we will present the following result - the Kleene Basis theorem.

\begin{theorem}[Kleene Basis Theorem, \protect{\cite[Thm. 3.1]{Cenzer-griffor1999handbook}}]
For any tree $T$ such that a $\Pi^0_1$ class $P = [T] \neq \emptyset$, $P$ contains a member that is computable in $Ext(T)$.
\end{theorem}

\begin{proof}
The infinite path $x$ through $T$ can be computably defined by letting $x(0)$ be the least $n$ such that the sequence $(n) \in Ext(T)$. We continue the construction by letting, for every $k$, $x(k+1)$ be the least $n$ such that $(x(0),x(1),\ldots,x(k),n) \in Ext(T)$.
\end{proof}

We can also prove the following result:

\begin{theorem}[\protect{\cite[Thm. 3.3]{Cenzer-griffor1999handbook}}]
For any recursive tree $T \subset \omega^{< \omega}$, $Ext(T)$ is a $\Sigma^1_1$ set.
\end{theorem}

\begin{proof}
This follows from the following characterisation: \[ \sigma \in Ext(T) \iff (\exists x)(\forall n > |\sigma|) (x \upharpoonright n \in T \wedge \sigma \prec x \upharpoonright n) \]
\end{proof}

These results solidify the fundamental link that we will use later, specifically that the well- or ill-foundedness of a tree $T \subset \omega^{<\omega}$ is complete to $\Pi^1_1$ and $\Sigma^1_1$ formulae. This is a fact that is central to our results in chapter 3 and beyond.

\chapter{Tilings - Concepts and Results}
\label{chap3}
\setcounter{equation}{0}
\renewcommand{\theequation}{\thechapter.\arabic{equation}}


\epigraph{It is the shape that matters.}{\textit{Samuel Beckett \\ to Harold Hobson}}

This chapter presents previous results to do with the mathematical study of tiling problems. We present more general results first, and then focus on tiling problems for Wang prototiles that will occupy the rest of our study in this thesis.

\section{Tilings of the Plane}

In this chapter, we will give an overview of the notation, history, and important results concerning tiling problems. Unless otherwise indicated, we will use \cite{GrunbaumTP} and \cite{FuchsT} as our primary resources for material in this chapter.

\subsection{Preliminaries of Tilings}

We will use the following definitions of tilings in this thesis. Note we restrict ourselves to tilings on the plane $\mathbb{R}^2$. 

\begin{definition}[Tiles]
A \emi{tile} is a closed polygon that covers some finite potion of the plane.
\end{definition}

Topologically, each tile is a closed subset of the plane, and is homeomorphic to a disc. As such, we can define \emi{tilings} as follows:

\begin{definition}[Tilings]
Tilings will generally take the following forms:
\begin{itemize}
\item Tiles form a \emi{complete tiling} if the union of these subsets is the full plane.
\item Tiles form a \emi{partial tiling} if there are points in the plane that are not contained in any subset.
\end{itemize}
\end{definition}

For complete tilings, each point $p \in \mathbb{R}^2$ will find itself in one of two situations. Either we have that:
\begin{enumerate}
	\item $p$ is to the interior of at most one tile, or
	\item $p$ is on the edge join of two tiles.
\end{enumerate}
As a consequence of this, tiles in a complete tiling have pairwise disjoint interiors, and there are no gaps between the tiles in the tiling. 

To make it easier to consider the relationship between a tiling and the tiles that constitute it, we can define sets of prototiles as follows:

\begin{definition}[Prototile Sets]
For a given tiling $\mathcal{T}$,
\begin{itemize}
\item A \emi{prototile} set $\mathcal{S} \subset \mathcal{T}$ is a set of tiles such that for every tile $t \in \mathcal{T}$ there is an $s \in \mathcal{S}$ that is congruent to $t$.
\item A prototile set $\mathcal{S}$ is called \emph{minimal} if for all $s_i,s_j \in \mathcal{S}$, \[ s_i \text{ is congruent to } s_j \iff s_i = s_j \]
\end{itemize}
\end{definition}

Later in this thesis we will consider only \emph{minimal tilings}, where we have substituted geometric requirements with a regular polygonal lattice with edge conditions. But for now, we will proceed with all the above definitions.

\subsection{The Extension Theorem}

The Extension Theorem is a compactness-like argument that is an important result from the literature, a version of which will become very useful later in this volume.

We will start with some definitions for related and useful concepts we will use in theorem \ref{thm:extension}. For this section we will assume that all prototile sets are finite, although we will relax this requirement for our further work in tiling problems later in this volume.

\begin{definition}
Given a tiling $\mathcal{T}$, $t_i, t_j \in \mathcal{T}$, the \emi{Hausdorff distance} $h(t_i, t_j)$ between two tiles is defined as \[ h(t_1,t_2)  = \max\left\{ \sup_{a \in t_1} \inf_{b\in t_2} \norm{a - b} ,\sup_{b\in t_2} \inf_{a\in t_1} \norm{a - b} \right\} \]
\end{definition}

From this definition it follows that where for some tiles $t_1, t_2 \in \mathcal{T}$, we have that $h(t_1, t_2) = 0 \implies t_1 = t_2$. 

\begin{definition}[Patch Tiling]
A \emi{patch} is the union of a number of tiles covering some non-total portion of the plane $R \subset \mathbb{R}^2$.
\end{definition}

The usual intuition for patch tilings is that they are finite portions of the plane, however we will also use this wording to denote infinite connected regions of the plane that are not total. Where the context requires we will talk of `infinite patches' and `finite patches', but generally speaking, we use this looser definition of `patch tiling' than is generally used in the literature.

\begin{definition}
We say that a set of prototiles $\mathcal{S}$ \emi{tiles over} a finite subset $X$ of the plane if there is a finite patch tiling $P_\mathcal{S}$ such that for all $x \in X$, $x \subset P_\mathcal{S}$, with each $t \in P_\mathcal{S}$ congruent to some $s \in \mathcal{S}$.
\end{definition}

Where we have finite patches as a bounded tiling, these are then also topologically equivalent to a disc.

\begin{definition}
A sequence of tiles $t_1, t_2, t_3, \ldots$ \emi{converges} to a limit tile $t$ if $\lim_{i\to\infty} h(t_i,t) = 0$.
\end{definition}

\begin{definition}[Circumparameter]
$U$ is a \emi{circumparameter} of a prototile set $\mathcal{S}$ if for every $t \in \mathcal{S}$, $t$ is contained in some disc of radius $U$.
\end{definition}

\begin{definition}[Inparameter]
Analogously we have that $u$ is an \emi{inparameter} of $\mathcal{S}$ if for each $t \in \mathcal{S}$, there exists a disc of radius $u$ that can be wholly inscribed within $t$.
\end{definition}

Now that we have covered the base definitions we require for this section, we will proceed to prove some general theorems in the theory of tilings. Our aim here is to state the geometric and topological arguments that are commonly used to analyse general properties of tilings derived from finite prototile sets. We begin with the following lemmas:

\begin{lemma}[Bolzano-Weierstrass Theorem]\label{lemma:ConvSeq}
Let $S$ be a closed bounded area in $\mathbb{R}^2$, and let $z_1,z_2, z_3, \ldots $ be a sequence of points in $S$. There is a subsequence of $z_{i_1}, z_{i_2},\ldots$ that converges to some point $z \in S$.
\end{lemma}

Note, such a limit $z$ need not be unique. 

\begin{proof}
Let $z_i$ for $i \in \omega$ be our sequence $z_1, z_2, z_3, \ldots$, and let $S_0$ be a bounded region in $\mathbb{R}^2$.
	
First we bisect $S_0$. By pigeonhole principle, we have that at least one of these pieces contains infinitely many $z_i$. Call this piece $S_1$, and repeat the subdivision infinitely. The same density must apply to at least one of any subdivided region, so we can choose a sequence of pieces $S_2$, $S-3$, \ldots containing infinitely-many $z_i$ in each subsequent piece.

From our eventual infinite sequence $S_0, S_1, S_2, \ldots$ we can choose any sequence of points, with each successive $z_i$ coming from $S_i$. These points converge closer to some limit point $z$.
\end{proof}

\begin{theorem}[Selection Theorem, \protect{\cite[p.154]{GrunbaumTP}}]\label{thm:Selection}
Let $t_1, t_2, \ldots$ be an infinite sequence of tiles such that all $t_i$ are congruent - by translation and rotation - to (bounded) $t$, that is fixed. If every $t_i$ contains point $p$, then the sequence contains a convergent subsequence whose limit tile $t'$ is congruent to $t$, with $p \in t$.
\end{theorem}

\begin{proof}
Choose $t_n \cong t_0$ for each $n \in \omega$. As such, each point $p \in t_n$ identifies some point $q_n \in T$. By \ref{lemma:ConvSeq}, there is a convergent subsequence $q_{i_1}, q_{i_2}, \ldots \rightarrow q$ inside $t$. 

Intuitively, this limit point $q$ is taken from a point $p$ that is `common' to all tiles where they translated, but \emph{not} rotated, and placed over each other. When separated out, this is our sequence of $q_i$'s, where each tile is labelled spiralling out from our $t_0$ - much like the `snake' proof in classical set theory.

If we position this $t$ such that $q$ is over the coordinate $(0,0) \in \mathbb{R}^2$, then we can notice that all of our translations are rotated about $q$. So the position of each tile is the translation $q-q_i$ followed by some rotation $\alpha_{i_n}$ (mod $2\pi$).
	
As such, by using the same reasoning in lemma \ref{lemma:ConvSeq}, we can we can gather a subsequence of rotation angles $\alpha_{i_1}, \alpha_{i_2} , \ldots$ which converges (modulo $2\pi$). Let $\alpha$ be the limit of this sequence, and so $t_{i_1}, t_{i_2}, \ldots \rightarrow t'$, which is a copy of $t$ rotated by $\alpha$ and with $q$ coincident with $p$.
\end{proof}

Note, this theorem will fail if such a $p$ does not exist, for example. That said, we will use the following special case later:

\begin{corollary}[\protect{\cite[p.154]{GrunbaumTP}}]
Let $t_0, t_1, t_2, \ldots$ converge to some $t$, if $d(t_i, t) \rightarrow 0$ as $i \rightarrow \infty$. Then $t$ is congruent to $t_0$. 
\end{corollary}

We can now prove the Extension Theorem, which is a fundamental, general result about tilings.

\begin{theorem}[Tiling Extension Theorem, \protect{\cite[Thm. 3.8.1]{GrunbaumTP}}]\label{thm:extension}\index{extension theorem}
Let $\mathcal{S}$ be a finite set of prototiles - each of which is a closed topological disc. If $\mathcal{S}$ tiles over arbitrarily large discs, then there exist $\mathcal{S}$-tilings of the plane. 
\end{theorem}

The proof will follow the one found in \cite[p.151]{GrunbaumTP}.

\begin{proof}
Let $\mathcal{S}$ be a finite set of prototiles, and let $U$ be the common circumparameter, and $u$ be the common inparameter. Consider the lattice $\Lambda$ of all points who regular Cartesian coordinates are $(nu,mu)$ for $m,n \in \mathbb{Z}$. $\Lambda$ therefore has some point in each prototile in $\mathcal{S}$. Let $L_0, L_1, L_2, \ldots$ be the full sequence of these points, spiralling out from some chosen $L_0$, say $(0,0)$. 

For any positive $r \in \mathbb{N}$, let $D(L_0, r)$ be the disc of radius $r$ centred on the point $L_0$. Let $P(r)$ be the finite patch of tiles from $\mathcal{S}$ that covers $D(L_0,r)$. When $r$ is large enough for $D(L_0,r)$ to contain some $L_s$, let $t_{rs}$ denote the tile of $P(r)$ that covers $L_s$. If, however, $L_s$ lies on an edge or a vertex point, then we can choose any tile in $P(r)$ that is incident to $L_s$. 

By the Selection Theorem (\ref{thm:Selection}) we have that, given $\mathcal{S}$ is finite, the sequence $t_0, t_1, t_2, \ldots$ has a subsequence $t'_{0}, t'_{1}, \ldots$ of tiles that are congruent to $t'_0$. This sequence will also contain an infinite subsequence $S_0 = t'_{i_0}, t'_{i_1}, \ldots$ that is convergent, and whose limit tile $t'_0$ will also contain $L_0$.

We now consider the sequence of tiles $t_{r1}$ containing $L_1$, restricting attention to values of $r$ that correspond to tiles in $S_0$. We can carry out the same line of argument as we just did to acquire $S_1$ of tiles all congruent to $t_1$, containing $L_1$, and convergent to a limit tile $t'_1$. 

Let $\mathcal{T} = \{ t'_0, t'_1, t'_2, \ldots \} $, deleting any duplicates as necessary in our selection. Ultimately, we want to show that $\mathcal{T}$ forms an $\mathcal{S}$-tiling of the plane. To show this, let $p$ be any point of the plane. We want to show that $p$ belongs to at least one $t'_i$, but does not belong to the interior of any other $t'_j$. 

Let $D(p,u)$ be the disc centred at $p$, of circumparameter radius $u$. Let $L_m$ be the point of $\Lambda$ in $D(p,u)$ with greatest index. We want to restrict our attention to the sequence of finite patches $P(r)$ as $r$ ranges through value corresponding to the subsequences $S_m$, specifically $\mathcal{T}_r = \{ t_{r0}, t_{r1}, t_{r2}, \ldots t_{rm} \}$.

As $r \rightarrow \infty$, $\mathcal{T}_r$ converges to the set $\mathcal{T}' = \{ t'_0, t'_1, \ldots \}$. Since all of the tiles in $\mathcal{T}_r$ have disjoint interiors, and all contain $p$, the same is true of each member of $\mathcal{T}'$. Thus, $\mathcal{T}$ is an $\mathcal{S}$-tiling of the plane.
\end{proof}

\section{The Domino Problem}

Whilst theorem \ref{thm:extension} gives us a notion of compactness that we can express through tiles, we then come to a more general question about tilings, known as the `Domino Problem'.

\begin{definition}[the Domino Problem]\label{def:domprob}\index{domino problem}
For any given set of prototiles $S$, does there exist an $S$-tiling of the plane?
\end{definition}

By theorem \ref{thm:extension} we know that if we can extend any finite patch $S$-tiling, we can get a tiling of the plane, but the Domino Problem asks us to consider whether there is any finite patch that cannot be tiled. 

When considering the Domino Problem for various sets of tiles, it is possible to modulate various requirements on how we cover the plane. For example, we might not consider trivial sub-tilings of some $S' \subset S$, or we might permit `small' holes that are strictly smaller than any polygon $t \in S$, such that we can consider them `small enough' in the limit. 

Given this definition, it is common to consider these conditions on a Domino Problem for some prototile set, unless explicitly indicated otherwise. Given a set of prototiles $S$:

\begin{itemize}
\item We will \emph{not} require that $\forall t \in S$, $t$ is used at least once in each $S$-tiling of the plane.
\begin{itemize}
\item This requirement is sometimes used to prevent trivial sub-tilings of the plane of some $S' \subset S$, mentioned above. However, when we come to dealing with encoding a Turing Machine into prototile a set, we need to allow that our TM will not enter every state on every input.
\end{itemize}
\item We will require that, for lattice regular polygonal tilings such as Wang tiles (defined in \ref{def:Wang}) we do not admit rotations of the tiles.
\begin{itemize}
\item Although this increases our prototile sets significantly, it make our later more functional definitions much more straightforward.
\end{itemize}
\item That our tilings are complete tilings.
\end{itemize}
Though we will make use predominantly of lattice-based tilings in this thesis, we wish to prevent gap from occurring in our tilings. As such, our resultant tilings can be thought of as total functions over the plane via coverings given by mapping each point in the lattice to a copy of a prototile.

These requirements can serve as to simplify our tilings, definitions, and constructions of prototiles. As noted, although the number of prototiles will increase, the complexity of our tiling functions will significantly reduce.

However, in this thesis, we will attempt to make our tilings as `free' as possible. Although this gives us slightly larger prototile sets, it serves to give us some better insight into the equivalence between the logical complexity of some statement, the computable trees arising from these statements, and the computable prototile sets that code paths of these computable trees into planar tilings.

\subsection{Wang Tiles}

To properly analyse the Domino Problem, we wish to reduce the complexity of our tilings to some `bare minimum', in line with the requirements above. As such we will make use of Wang tiles, first introduced by Hao Wang in \cite{Wang1990}, which we define as follows:

\begin{definition}[Wang Tiles]\label{def:Wang}\index{Wang tiles}
Let \emi{Wang tiles} be square tiles, diagonally quadrisected, such that ordered 4-tuples of the form $\langle l,u,r,b \rangle$ can be represented by:

\begin{center}
\sampletile{$l$}{$u$}{$r$}{$b$}
\end{center}
Where $l,u,r,b$ each stand for left, upper, right, and bottom respectively.
\end{definition}

We keep our previous definitions of `prototiles', `prototile sets', and `tilings'. Given this prototile definition, we will need to consider what happens when the edges of our tiles are to meet. Given a set $\mathcal{S}$ of Wang prototiles:

\begin{definition}
Given two Wang tiles $w, u \in \mathcal{S}$, such that $w = \langle l_w, u_w, r_w, b_w \rangle$ and $u = \langle l_u, u_u, r_u, b_u \rangle$:
\begin{itemize}
\item The \emi{edge meets} between these tiles are the comparisons between meeting edges, where one of the following applies:
\begin{itemize}
\item $l_w$ is next to $r_u$,
\item $u_w$ is next to $b_u$,
\item $r_w$ is next to $l_u$,
\item $b_w$ is next to $u_u$
\end{itemize}
\item The \emi{match criteria} for Wang tiles are the requirements that for any edge meet, the edge symbols match. Explicitly, one of the following holds: 
\begin{itemize}
\item if $l_w$ is next to $r_u$, then $l_w = r_u$
\item if $u_w$ is next to $b_u$, then $u_w = b_u$
\item if $r_w$ is next to $l_u$, then $r_w = l_u$
\item if $b_w$ is next to $u_u$, then $b_w = u_u$
\end{itemize}
\end{itemize}
\end{definition}

\begin{figure}[t]
  \centering
  \begin{tikzpicture}
  \begin{scope}[scale=2]
\foreach \x/\y/\l/\u/\r/\b in {1/0/$\cdot$/$\cdot$/$\cdot$/$u$,
	       0/1/$\cdot$/$\cdot$/$l$/$\cdot$,
	       1/1/$l$/$u$/$r$/$b$,
	       2/1/$r$/$\cdot$/$\cdot$/$\cdot$,
	       1/2/$\cdot$/$b$/$\cdot$/$\cdot$}
      {
      \draw[fill=white] (1+\x,1-\y) rectangle (0+\x,0-\y);
      \filldraw[fill=white] (0+\x,0-\y) -- (0.5+\x,0.5-\y) -- (0+\x,1-\y) -- cycle;
      \filldraw[fill=white] (0+\x,0-\y) -- (0.5+\x,0.5-\y) -- (1+\x,0-\y) -- cycle;
      \filldraw[fill=white] (0+\x,1-\y) -- (0.5+\x,0.5-\y) -- (1+\x,1-\y) -- cycle;
      \node at (0.5+\x,0.5-\y) [label=above:{\u},label=left:{\l},label=right:{\r},label=below:{\b}] {};
     }
\end{scope}
\end{tikzpicture}
    \caption{Edge Conditions in the von Neumann Neighbourhood surrounding a Wang tile.}
    \label{fig:vNN}
\end{figure}
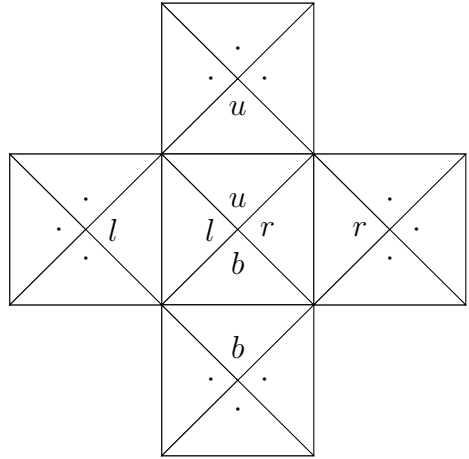

Intuitively we use the von Neumann neighbourhood surrounding the tile as the basis for our matching and placement conditions for each Wang prototile. This means that we only ever consider the 4-place valency for each tile and for each position in our $\mathbb{Z}^2$ lattice following the rules we stated above. When we come to code cellular automata, we will still only consider the von Neumann neighbourhood over the usual Moore neighbourhood.

From this construction of Wang tiles, we can now envisage our tilings as projection functions $$f_{\mathcal{S}} : \mathbb{Z}^2 \rightarrow \mathcal{S}$$ This characterisation will be useful when we explore computable tilings later in this thesis. Thus, the following definition is natural:

\begin{definition}\index{Total Wang tilings}
Given a set of Wang prototiles $\mathcal{S}$, we say that an $\mathcal{S}$-tiling of the plane is a \emi{total tiling} if for $f: \mathbb{Z}^2 \rightarrow \mathcal{S}$, and $f$ enforces the edge-meet criteria for the von Neumann neighbourhood of every point in $\mathbb{Z}^2$.
\end{definition}

Given for every point $(x,y) \in \mathbb{Z}^2$ there is some $s \in \mathcal{S}$ such that $f(x,y) = s$ and $f$ ensures that $s$ observes and meets all of the match criteria for its neighbours in the plane.

Our notion of a `total Wang tile tiling' is indeed a direct analogue for complete tilings we defined earlier. The slight change in terminology is to facilitate the intuition we will use later in this thesis that a complete tiling generated by a computable function must be total on $\mathbb{Z}^2$, and so is in this sense a total function. Thereby, total functions give total tilings, and total tilings must come from total functions.

Thus, a total tiling from a Wang prototile set is analogous to a complete tiling we considered previously. When we consider computable sets of Wang prototiles, this definition will be equivalent to a computable function $\varphi_e$ being total. 

Wang proved a version of the Extension Theorem for Wang tiles - known as \emph{Wang's theorem}. Our statement and proof are taken from \cite[p.600]{GrunbaumTP}.

\begin{theorem}[\protect{\cite[p.600]{GrunbaumTP}}]\label{thm:WangExtension}
Let $\mathcal{S}$ be a finite set of Wang prototiles. If it is possible, of arbitrarily large values of $n$, to assemble $n \times n$ blocks of tiles satisfying the edge-matching conditions, then there is an $\mathcal{S}$-tiling of the plane.
\end{theorem}

We should reiterate that we only admit \emph{translations} of Wang prototiles - we do not permit rotations of Wang prototiles into tilings of the plane. If we did, this theorem would be immediate from the Extension theorem, theorem \ref{thm:extension}. Additionally the proof will make explicit use of the face that $\mathcal{S}$ is a finite set of prototiles.

\begin{proof}
Given a set of prototiles $\mathcal{S}$, with $| \mathcal{S} | = r$. We can construct a graph-theoretic tree in the following manner. We start with a single root node $n^0$. Level 1 is formed of $n^1_1, \ldots, n^1_r$ corresponding to each of the tiles in $\mathcal{S}$. Similarly at each level $k$, we add nodes $n^k_1, \ldots,n^k_{r_k}$ corresponding to adding a ring of tiles around each of the previous blocks.

We then form the tree by joining all of level 1 nodes to the root node. For any level $k$, we connect any of the $n^k$ to the nodes in $n^{k+1}$; if, for any $n^k_i$, $n^{k+1}_j$ contains the block represented by $n^k_i$, and the tiles on the outer edge of block $n^{k+1}_j$ match all the edge-matching criteria for the exterior of the tiles represented in $n^k_i$ are met by the inner edge criteria $n^{k+1}_j$. If this holds, then $n^k_i \text{ and } n^{k+1}_j$ are connected.

Each successive block can be thought of as an extension of the previous block by an outer `square ring' of tiles from $\mathcal{S}$ that surround the outside of the block.

Thus, we can reduce the question of an $\mathcal{S}$-tiling now to whether each level $k$ is connected to each $k+1$. If the answer is in the negative, then  there exists some $n$ such that there can be no patch of Wang tiles greater than $n \times n$ that can be extended to a full planar tiling.

If the answer is in the positive, then we have created a finitely branching infinite tree. By K\"onig's Lemma, there is necessarily an infinite path through our tree. By this construction, this path corresponds to an $\mathcal{S}$-tiling of the plane. 
\end{proof}

It is worth noting that although K\"{o}nig's Lemma is utilised in this proof, this is not necessary. Given our sets of prototiles are always finite, we only actually require Weak K\"{o}nig's Lemma - that an infinite bounded-branching tree necessarily has an infinite path - for this proof with some modification of our tiling tree as follows.

\begin{proof}[Alternative Tree Construction for proof of \ref{thm:WangExtension}]
Take some finite prototile set $S$, and consider each each point on $\mathbb{Z}^2$ by spiralling out from the centre point $(0,0)$ as before for the proof of theorem \ref{thm:extension}. We can construct a tree based on the valid tiles that could be placed at each successive point based on the 1 or 2 edge criteria defined by previously placed tiles. 

This tree is bounded by the size $S$, which is finite, thereby restricting the branching of our tree. A total planar tiling also corresponds to a path through this tree by the following observations:
\begin{itemize}
\item Each level on our tree corresponds to a point in $\mathbb{Z}^2$.
\item All edge-meet criteria are met by the construction of each branch.
\end{itemize}
Thus if our tree is infinite, there must be an infinite path by WKL, meaning there is a total planar tiling.
\end{proof}

Later in this thesis, we will entertain weaker notions of tiling the plane, and will draw more equivalences with properties and principles on trees in both Baire space and Cantor space.

\section{Undecidability of the Domino Problem}

One of Hao Wang's students, Robert Berger, proved in \cite{berger1966} the undecidability of the Domino Problem for finite sets of Wang prototiles. Whilst Berger's original created a prototile set of over 6,000 tiles, we present an updated proof where sets of `universal Turing Machine prototiles' number in the few hundred.

\begin{definition}
For a set $\mathcal{S}$ of prototiles, we denote ``There exists a complete $\mathcal{S}$-tiling of the plane" by $Tile(\mathcal{S})$.
\end{definition}

Note that $Tile(S)$ immediately has a $\Sigma^0_1$ normal form as the existence of an infinite sequence $s \in S^{\omega}$, such that $s$ is a sequence of tiles that covers each point in the lower-right quarter plane in $\mathbb{Z}^2$, thereby giving a total tiling of this quarter plane.

However, given we can extend any $S$ with tiles that fill in the other three quarter-planes, we can convert this $s$ to a total planar tiling.

\begin{theorem}[\protect{\cite[Thm. 3-3]{berger1966}}]\label{thm:TMTilings}
The Domino Problem for finite Wang prototile sets is $\Sigma^0_1$-complete.
\end{theorem}

We will prove this by showing that for any Turing Machine $\varphi_e$ there exists a set of prototiles $\mathcal{S}_e$ such that $$ \varphi_e(x) \downarrow \iff \neg Tile(\mathcal{S}_e) $$ In order to do this, we will need the following machinery:

\begin{definition}\label{def:schematile}
A \emi{schema tile} is a prototile that determines a set of prototiles for given sets of colours. That is, it determines the position of colours taken from one or more sets of colours.
\end{definition}

\begin{example}[Schema Tile Example]
Let $A = \{ a_1, a_2 \}$ and $B = \{ b_1 \}$ be sets of colours. Let $t$ be the schema tile, with $i \neq j$:
\begin{center}
\sampletile{$a_i$}{$b_i$}{$a_j$}{$b_i$}
\end{center}	
The prototile set $\mathcal{S}$ generated by $t$ will consist of the following tiles:
\begin{center}
\sampletile{$a_1$}{$b_1$}{$a_2$}{$b_1$} \sampletile{$a_2$}{$b_1$}{$a_1$}{$b_1$}
\end{center}
It is worth observing that this resultant prototile set can give total planar tilings.
\end{example}

Thus, we can talk about the following progression: $$ \text{schema tile} + \text{colours} \Rightarrow \text{prototile sets} \Rightarrow \text{planar tilings} $$ By careful control of our schema tiles, we can establish the overall `shape' or `behaviour' of our prototile sets, which in turn controls some desirable feature or features of our classes of planar tilings.  

The following proof is after \cite{Boas97Conv} and \cite{MCarneyMSc}, however it has been restructured in order to match the structure of proofs later in this thesis.

\begin{proof}[Proof of \ref{thm:TMTilings}]
We construct the following schema tiles with which we can emulate Turing Machines. Let $s \in \Sigma$ be colours representing symbols, $q_i \in Q$ be colours representing machine states, and $(s,q) \in \Sigma \times Q$ be colours corresponding to each symbol matched with each state. Let $B$ be a distinguished colour representing `blank', and $H$ be distinguished colour representing the halting state.

\noindent \textbf{Symbol tiles}
\begin{center}
\sampletile{$B$}{$s$}{$B$}{$s$}
\end{center}
\textbf{Head State tiles}
\begin{center}
\sampletile{$q$}{$s$}{$B$}{$(s,q)$} \sampletile{$B$}{$s$}{$q$}{$(s,q)$}
\end{center}
\textbf{Computational tiles}
For $s,s' \in \Sigma$ and $q, q' \in Q$, permitting $s = s'$ and $q=q'$,
\begin{center}
\sampletile{$B$}{$(s,q)$}{$q'$}{$s'$} \sampletile{$q'$}{$(s,q)$}{$B$}{$s'$}
\end{center}
\textbf{Halting tile}
\begin{center}
\sampletile{$B$}{$s$}{$H$}{$s'$}
\end{center}

Let $\varphi_e$ be some Turing Machine composed of 5-tuples, and let $\varphi_e(x)$ be the computation that we wish to represent in our planar tilings.

We first take every symbol in our Turing program, and represent each one by some $s \in \Sigma$. We then code each symbol in the tape by a symbol tile. The `blank' representing colour $B$ serves to line up our rows into representations of configurations $c_i$ for $i \in \omega$. We now colour all of the symbol tiles with each $s \in \Sigma$, and put these into $\mathcal{S}_e$.

Next, we need to assign each of the states in $e$ to a state $q \in Q$, and we are then ready to add the Head State and Computation prototiles to $\mathcal{S}_e$. To do this, we take each $s \in \Sigma$, and each $q \in Q$, and assign colours for each `TM state' $(s,q)$. The Head State tiles will accept a state from $q$ from left or right, and will merge this information into the bottom quadrant of the prototile. 

Next, we need to look to all of the 5-tuples $(s,q,s',q',\{L,R\}) \in e$. For each $(s,q)$ taken from $\Sigma \times Q$, we look to see which of these form the first two positions of a 5-tuple. We then create a prototile for $\mathcal{S}_e$ of the form of this tuple based off the schema, placing the exit state $q'$ on the left or right according to the last position of the 5-tuple.
	
\emph{E.g.} let $(1,a,1,a,L)$ and $(1,b,0,a,R)$ be valid 5-tuples from some given $\varphi_e$. We can represent them in $\mathcal{S}_e$ by means of the computation schema tiles as follows (respectively left and right): 
\begin{center}
\sampletile{$a$}{$(1,a)$}{$B$}{$1$}  \sampletile{$B$}{$(1,b)$}{$a$}{$0$}
\end{center}
Given this we colour all the necessary computation tiles - except for any 5-tuple that enters the halting state, which we will deal with below) - remove any unnecessary head state tiles, and add all these to the symbol tiles in $\mathcal{S}_e$.

In order to complete the representation of $\varphi_e$, we need to add the halting states. These are distinctive, 5-tuples, and for any given halting 5-tuple $(s, q, s', HALT, \{ L, R \})$, we represent these as:
\begin{center}
\sampletile{$B$}{$(s,q)$}{$H$}{$s'$}
\end{center}

In order to fully represent our computation $\varphi_e(x)$ we perform the following steps:
\begin{enumerate}
\item We first take the representation of $x$ in symbols from $\Sigma$ - let this be a string of symbols $s_0, s_1, \ldots, s_k$, where $k = |x|$.
\item We take $s_0$, the initial state of our TM $q_0$, and place the following tile in the first position at co-ordinate $(0,0)$:
\begin{center}
\sampletile{$q_0$}{$s_0$}{$B$}{$(s_0,q_0)$}
\end{center}
\item We then place the respective symbol tiles for $s_1, \ldots, s_k$ to the right of this tile on what will become the representation of the first configuration $c_0$ of $\varphi_e(x)$. We can also continue tiling this entire bi-infinite row according to the symbols on the rest of the TM tape.
\begin{itemize}
\item We will later use the index on each configuration $c_i$ to map to the lower quadrants of every even row of tiles $r_{2i}$ for checking later.
\end{itemize}
\item We now continue the computation by continuing the tiling - given the prototiles in use code each part of the computation, each row can be read off as a successive stage of the computation.
\item the Halting tiles are designed that they will block the tiling from tiling the plane to the right any further.
\end{enumerate}

We can check the following facts about our tiling computation:
\begin{itemize}
\item Given our TM is not a non-deterministic TM, there will be only one choice for each computation prototile on each row.
\item Each row $r_{2i}$ will correspond to some configuration $c_i$ in our computation, with the tape configuration being readable from the top quadrants of each tile on the row.
\item Given our first row setup, there will not be more than one TM head performing the computation.
\end{itemize}

Thus, our tiling problem $Tile(\mathcal{S}_e(x))$ is also represented by the problem $$\exists s \, \{r_{2s} \in \mathcal{S}_e(x) \text{ has a hole}\}$$ which is in turn equivalent to the statement $\exists s \, \varphi_{e,s}(x) \downarrow$. As such, the Domino Problem for finite Wang prototile sets is $\Sigma^0_1$-complete.
\end{proof}


%
%
%

%

\begin{corollary}[\protect{\cite[Cor. 4-1, p.36]{berger1966}}]
The Domino Problem is undecidable.
\end{corollary}

\begin{proof}
By \ref{thm:TMTilings}, it is clear that there exists a class of prototile sets corresponding to each TM enumerated by some $e$. By our construction, $$ \neg Tile(\mathcal{S}_e(x)) \iff \varphi_e(x) \downarrow $$
	
Thus, given the Halting Problem is undecidable, then the question of whether or not the corresponding $\mathcal{S}_e$-tilings tile the plane or not is also undecidable.
\end{proof}

The above re-proof of this classic result due to Berger is intended to illustrate our proof method in later chapters.

The original proof uses much more machinery, and a large set of prototiles for a Universal Turing Machine. This simplification makes plain the equivalence much more immediately, and lays a groundwork for our later results.

We will use this equivalence in the rest of this volume when we define \emph{computable prototile sets} and \emph{computable tilings} in the next chapter.

\subsection{Universal Turing Machine and TM Tilings}

Let Universal Turing Machines (UTMs) be minimal Turing Machine symbol and state sets, such that they can effectively emulate a Turing Machine of any size. 

\begin{definition}
Let a \emi{$(x,y)$-Universal Turing Machine} $\psi$, denoted $(x,y)$-UTM, be a Turing Machine that uses precisely $x$ active non-halting states, and $y$-many symbols on the tape, such that $\psi$ is Turing Complete.
\end{definition}

As such, we can think of them as being a pre-coded minimum requirement for any Turing Machine to operate. Let $\mathcal{S}_{UTM}$ denote a `library' of all possible states and symbols given by some UTM of a given number of states and symbols.
 
Due to the succinctness of our construction, it is reasonable to ask ``how big would a Turing prototile library be?". By colouring for all possible states, symbols, and state-symbol combinations we can get the following theorem:

\begin{theorem}[\protect{\cite[Chap. 3]{MCarneyMSc}}]
There exists a set, called the \emph{library}, of prototiles $\mathcal{S}$ with $|\mathcal{S}| = 625$, such that for every $\varphi_e$ there exists a set of prototiles $P_e \subset \mathcal{S}$ such that $P_e$ is a finite set of prototiles that represents $\varphi_e$ selected from $\mathcal{S}$.
\end{theorem}

The proof of this can be found in \cite{MCarneyMSc}, and involves colouring a full library of Turing tiles with the states and symbols of a $(2,5)-UTM$, known universal Universal Turing Machine.

Indeed, if we take Smith's as-yet unpublished proof that a $(2,3)$-TM is universal, \cite{Smith2007}, then we can get the following theorem:

\begin{theorem}[C. 2019]
There is a library set of Turing Machine encoding prototiles of size 105.
\end{theorem}

The proof comes from generating colours from a set of states $|\Sigma| = 2$ and a set of symbols $|Q| = 3$, obtaining $|\Sigma \times Q| = 6$, and then applying these colours to our Turing Tile schemas, and then counting all possible compositions.

%
%

\section{Implications of TM Tilings}

There are some interesting implications that arise out of the fact that every Turing Machine has a representation in tiles. We state the following processes and theorems from \cite{Cichon1983}, assuming that the definitions of Primitive Recursive Arithmetic (PRA) and Peano Arithmetic (PA) are already known:

\begin{definition}[\protect{\cite[Process 1]{Cichon1983}}]\label{def:Process1}
\begin{enumerate}
\item Given some $n \in \omega$, write this number as the sum of powers of $x$ (base-$x$ notation).
\item Increase the base of the representation by 1.
\item Subtract one from this new representation. 
\item Return to 2 and repeat this procedure.
\end{enumerate}
\end{definition}

\begin{definition}[\protect{\cite[Process 2]{Cichon1983}}]\label{def:Process2}
Same as \ref{def:Process1}, except that on step 1 we write $n$ as \emph{pure base} representation, that is we write $n$ in base $x$, and then continue this process for all the exponents.
\end{definition}

The difference between these two definitions is that process 1 (definition \ref{def:Process1}) will admit for $n = 244$ a representation of $3^5 +1$, whilst \ref{def:Process2} will go further to $3^{3+2} +1$. After one iteration of \ref{def:Process1} we get $(3^5 + 1) : \rightarrow 4^5$, whereas \ref{def:Process2} will give us $(3^{3+2} + 1) : \rightarrow 4^{4+2}$.

The algorithm in \ref{def:Process2} is due to Goodstein in 1944 in \cite{Goodstein1944}. \cite{Cichon1983} gives short, elegant proof of the following famous results originally due to Kirby and Paris \cite{ParisKirby82}:

\begin{theorem}[\protect{\cite[Thorem 1]{Cichon1983}}]\label{thm:Cichon1}
For any $n \in \omega$ and base $x$, \ref{def:Process1} terminates, but this fact is not provable in PRA.
\end{theorem}

\begin{theorem}[\protect{\cite[Thorem 2]{Cichon1983}}]\label{thm:Cichon2}
For any $n \in \omega$ and base $x$, \ref{def:Process2} terminates, but this fact is not provable in PA.
\end{theorem}

Denote by $ProvRec(PA)$ the Provably Recursive functions of PA. Cichon's \cite{Cichon1983} proof of \ref{thm:Cichon2} relies on demonstrating that some machine $\varphi_{Good}$ that computes \ref{def:Process2} is such that $$\varphi_{Good} \notin ProvRec(PA)$$ Given this fact, it is necessarily true that $$ PA \nvdash \forall n,x \, \exists s \, \varphi_{Good,s}(n,x) \downarrow = 0$$

Let $S_{Good}$ denote the Turing Machine tiling generated by the process outlined in the proof of theorem \ref{thm:TMTilings}. We get the following corollary:

\begin{corollary}[C. 2019]\label{cor:PAGoodNotTile}
It is necessarily the case that for all $n,x$ there exists an $s$ such that the row $r_{2s}$ has a hole, and so $\forall n,x \, [\neg Tile(S_{Good}(n,x))]$, however by \cite{Cichon1983} it is necessarily true that $$PA \nvdash \forall n,x \, [\neg Tile(S_{Good}(n,x))]$$
\end{corollary}

It is perhaps unexpected \emph{prima facie} that the Domino Problem would have the means to defy provability of mathematically strong theories such as PA. However, the long established relationships between tilings and computability cement that there exists sets of Wang prototiles that have interesting proof theoretic outcomes.

\chapter{$\Sigma^1_1$-Complete Tilings}
\label{chap4}
\setcounter{equation}{0}
\renewcommand{\theequation}{\thechapter.\arabic{equation}}


\epigraph{I could be bounded in a nutshell and count myself king of infinite space.}{\textit{Hamlet}}

In this chapter we present our main results that concern infinite sets of Wang prototiles, and relate these to problems on infinite trees in Baire space. Previous work in tilings has generally considered finite sets of prototiles - and this is a natural assumption to make about things that we ostensibly only consider to be of finitely-many possibilities. 

The difference, as we shall see, is that by allowing our tilings as functions $f:\mathbb{Z}^2 \rightarrow \mathcal{S}$ to range over infinite prototiles, the original Domino Problem \ref{def:domprob} becomes equivalent, after careful construction, to whether a tree corresponding to our tiling is well-founded or ill-founded. As we found that finite sets of prototiles are equivalent to the Halting Problem, so we construct this new equivalence in this chapter.

We then extend this result to a variation of the Domino Problem - the problem of `weakly tiling' the plane, as well as an analogous notion of `strongly not tiling' the plane.

\section{Computable Trees and Computable Tilings}

In the section that follows, we will need the following in order to prove theorem \ref{thm:TILE-ILL}. First, we define what we mean by computable tilings. Recall that we represent by $\langle l,u,r,b \rangle$ the Wang prototile
\begin{center}
\sampletile{$l$}{$u$}{$r$}{$b$}
\end{center}
We define a computable set of Wang prototiles as follows:

\begin{definition}\index{computable prototile sets}\index{computable tile sets}
Let $X \subset \omega$, and $\mathcal{S}$ be a set of Wang prototiles.
\begin{itemize}
\item Let $X_\mathcal{S} = \{ \langle c_l, c_u, c_r, c_b \rangle : \langle  c_l, c_u, c_r, c_b \rangle \text{ codes some prototile in } \mathcal{S} \}$.
\item We say that $\mathcal{S}$ is \emph{computable} if $X_\mathcal{S}$ is computable.
\item We say that an $\mathcal{S}$-tiling of the plane is computable if $f_\mathcal{S} : \mathbb{Z}^2 \rightarrow \mathcal{S}$ is computable.
\item We say that $\mathcal{S}$ is \emph{total} if for every point $(x,y) \in \mathbb{Z}^2$ and a tiling function $f: \mathbb{Z}^2 \rightarrow \mathcal{S}$, $f$ is total on $\mathbb{Z}^2$, all edge conditions are met for any $\mathcal{S}$-tiling.
\end{itemize}
\end{definition}

\section{$\Pi^1_1$ Properties of Tilings}

In this section we will cover previous work on the $\Pi^1_1$ nature of specified Domino Problems that inquire about the properties of tile occurrences in planar tilings.

\subsection{Harel's $\Pi^1_1$ Tilings}

David Harel in \cite{Harel1986} was interested in translations between various kinds of computable trees. The core idea is to formulate correspondences between finitely branching and countably infinitely branching trees and infinitely branching tress, one-to-one, such that the paths along the latter become ``$\varphi$-abiding" paths of the former, for $\varphi$ being some property of infinite paths.

Harel in \cite{Harel1986} proposes the following problem relating to Wang prototile sets:

\begin{definition}[Recurring Tile Problem]\index{recurring domino problem}\label{def:RecDP}
Given a set of prototiles $\mathcal{S}$, for $t \in \mathcal{S}$, does $t$ occur infinitely often in a tiling of the lattice $\mathbb{Z}^2$?
\end{definition}

This is a variation on the standard Domino problems that we have considered so far. Rather than ask ``do there exist planar tilings?" we ask ``do any planar tilings have a given property?" The property in this case is a weaker question than ``are all $S$-tilings periodic or aperiodic?" - something we will come to discuss later in this thesis.

Harel in \cite{Harel1986} goes on to prove the following theorem:

\begin{theorem}[\cite{Harel1986}, Theorem 6.3]\label{thm:RecDomProbSigma}
The Recurring Tile Problem is $\Sigma^1_1$-complete.
\end{theorem}

We first require the following definition and lemmas from \cite{Harel1986}:

\begin{definition}
A class $A$ is \emph{$\Sigma^1_1$-hard} if there is a computable way of converting any $\Sigma^1_1$ formula into some member of $A$.
\end{definition}

\begin{definition}
A tree $T$ is an $\omega$-tree if $T \subseteq \omega^{< \omega}$. A $k$-tree is a tree $T \subseteq \{0,1,\ldots,k-1\}^{< \omega}$ for some finite $k \in \omega$. If such a \emi{$k$-tree} $T$ is bounded by some $b \in \omega$ then it is a \emi{$b$-tree}. We say that a \emph{recurrence} in a $b$-tree is the repetition of some specific $i \in \{ 0, \ldots, k-1 \}$ along an infinite path. 
\end{definition}

For graph-theoretic trees, this is equivalent to some of the non-leaf nodes being marked, and a recurrence being infinitely many marked nodes along some infinite path in the tree. 

\begin{lemma}[\cite{Harel1986}, p.230]\label{lemma:InfTreeRec}
The set $A$ of computable well-founded $\omega$-trees is computably isomorphic to the set $B$ of computable marked recurrence-free $b$-trees.
\end{lemma}

This lemma then sets the scene for the following theorem:

\begin{theorem}[\cite{Harel1986}, Lemma 6.1]\label{lemma:AequivC}
Let $A$ be the set of computable well-founded $\omega$-trees, and let $C$ be the set of enumerated notation for all Non-deterministic Turing Machines (NTMs). Then $$ A \equiv_1 C $$
\end{theorem}

Recalling our definition of 1-reducibility in definition \ref{def:mred}, and let $A \equiv_1 B$ iff $A \leq_1 B$ and $B \leq_1 A$. A proof of this is found in \cite{Harel1986}. From here we get:

\begin{corollary}[\cite{Harel1986}, Corollary 6.2]
$C$ is $\Pi^1_1$ complete.
\end{corollary}

The intuition behind these results is to set the stage that the question:

\noindent
\textbf{C1}: ``for a given NTM $U$, does $U$ re-enter its starting state $q_0$ infinitely often?" 

\noindent 
is a $\Sigma^1_1$ link to our Recurring Tile Problem above (\textbf{RTP}). The proof of \ref{thm:RecDomProbSigma} thus proceeds as follows:

\begin{proof}[Proof of \ref{thm:RecDomProbSigma}]
To first see that \textbf{RTP} is $\Sigma^1_1$, let $\mathcal{S}$ and some $t \in \mathcal{S}$ be given. Construct and NTM $M$ that begins on a blank tape by initially constructing a blank tiling of $\mathbb{Z}^2$. At each step, $M$ iterates over the $\mathbb{Z}^2$ lattice in a spiral pattern, considering each point in turn. Non-deterministically, $M$ tries to tile each position with some tile from $\mathcal{S}$. $M$ rejects if the edge conditions fail to match, and signals a successful use of the tile $t$ by re-entering its starting state $q_0$. Otherwise, $M$ will never re-enter $q_0$. Thus, $M$ has the property \textbf{C1} iff $t$ occurs infinitely often in the $\mathcal{S}$-tiling.

The rest of the proof is showing that \textbf{RTP} is $\Sigma^1_1$-hard. This is done through the following three claims. First, define \textbf{R2} as follows:

\noindent
\textbf{R2} - Given $\mathcal{S}$ and $t \in \mathcal{S}$, can $\mathcal{S}$ tile the positive quadrant of $\mathbb{Z}^2$ with $t$ occurring infinitely often and with the borderlines coloured white?

\begin{claim}\label{claim:R2}
\textbf{R2} is $\Sigma^1_1$-hard.
\end{claim}

\begin{proof}[Proof of \ref{claim:R2}]
We sketch the following proof of this claim. By theorem \ref{lemma:AequivC} we have that for an NTM $M$ that computes from the right, the question of whether it enters its initial $q_0$ infinitely often will be a $\Sigma^1_1$-hard problem, as it will be equivalent to the well-foundedness of some $\omega$-tree. 

We then construct a tile set from a scheme such that for each $M$, the tile set we build from $M$ has the property \textbf{R2} iff $M$ has the property above.

Let $M$ be given, reserving $B$ as the `blank' symbol, and let $p,q$ be states, and $s,t$ be tape symbols, all in NTM quintuples as defined in chapter 1. Our prototile set $\mathcal{S}$ will consist of tiles generated by the schema defined in the proof of theorem \ref{thm:TMTilings}.

Given our translation of $M$ into tiles preserves the recurrent properties of $M$, if $M$ enters its starting state $q_0$ infinitely often, then the tile representing this will occur infinitely often in the tiling, so $\mathcal{S}$ satisfies \textbf{R2}, with the white borders guaranteed by substituting the blank colour $B$ for plain white quadrants in our prototiles.
\end{proof}

We modify \textbf{R2} to the following statement:

\noindent
\textbf{R3} - Given $\mathcal{S}$ and $t \in \mathcal{S}$, can $\mathcal{S}$ tile the positive quadrant of $\mathbb{Z}^2$ with $t$ occurring infinitely often?

\begin{claim}\label{claim:R3}
\textbf{R3} is $\Sigma^1_1$-hard.
\end{claim}

\begin{proof}[Proof of \ref{claim:R3}]

Note that the border requirement in the previous claim was intended to force the initial starting state tile giving $q_0$ to appear in the right place. Consider the following machine problem:

\noindent
\textbf{C2} - Given NTM $M$, is there some tape configuration and state such that the following computation does not halt and re-enters $q_0$ from the right onto a blank tape cell infinitely often?

\textbf{C2} is $\Sigma^1_1$-hard by theorem \ref{lemma:AequivC} and the observation that a machine can be run from any starting tape configuration and state. We now adjust our schema prototiles as follows in order to produce prototiles for our $\mathcal{S}$ as follows:

For all symbols $s \in \Sigma$:

\begin{center}
\sampletile{$\rightarrow$}{$s$}{$\rightarrow$}{$s$} \sampletile{$\leftarrow$}{$s$}{$\leftarrow$}{$s$}
\end{center}

For symbols $s, s^\prime \in \Sigma$ and $q_i, q_j \in Q$:

\begin{center}
	\sampletile{${\rightarrow \atop q_i}$}{$s$}{$\leftarrow$}{$(s,q_i)$} \sampletile{$\rightarrow$}{$s$}{${\leftarrow \atop q_i}$}{$(s,q_i)$}
\end{center}

\begin{center}
	\sampletile{${\rightarrow \atop q_j}$}{$(s,q_i)$}{$\rightarrow$}{$s^\prime$} \sampletile{$\leftarrow$}{$(s,q_i)$}{${ \leftarrow \atop q_j}$}{$s^\prime$}
\end{center}

Fix $t$ to be 

\begin{center}
	\sampletile{$\rightarrow$}{$$B$$}{${\leftarrow \atop q_0}$}{$(B,q_0)$}
\end{center}

The addition of the arrows forces patterns of the form $$\cdots \rightarrow \rightarrow \leftarrow \leftarrow \cdots  $$

This is intended to force only one state to appear on each row in our NTM tiling. Thus $t$ occurring just once forces exactly one state per row, and so $(\mathcal{S}, t)$ satisfies \textbf{R3} iff $M$ satisfies \textbf{C2}. 
\end{proof}

To complete our proof, we need to extend these tilings out from one quadrant to full planar tilings. First, note that our NTM tapes are bi-infinite two way tapes, so we can extend our $\cdots \rightarrow \rightarrow \leftarrow \leftarrow \cdots$ pattern to the left half of the plane easily.

Extending to the upper half-plane is trickier. Note that there is nothing that requires $M$ to have infinite computations in the forwards or backwards directions by default. We can fix the backwards direction by requiring that $M$ will return repeatedly into some state $q_i$, requiring that $q_i \neq q_0$.

Likewise, we can prevent $\mathcal{S}$ from having $t$ appear infinitely often upwards but nowhere appearing downwards by having $M$ hold a counter variable that is incremented each time $M$ enters $q_0$. Thus, a planar tiling with infinitely many $q_0$ in the upper half of the grid would indicate a computation that checks the presence of increasingly smaller positive integers, which is impossible. 

Thus, for these modified machines, $M$ satisfies \textbf{C2} iff $(\mathcal{S},t)$ satisfies \textbf{RTP}. This completes our sketch of this proof for \ref{thm:RecDomProbSigma} from \cite{Harel1986}.

\end{proof}

In the following sections, we will deviate from asking if the Recurring Tile Problem from definition \ref{def:RecDP} is $\Sigma^1_1$, and instead ask if we can find some $\Pi^1_1$ properties that are equivalent to the original Domino Problem (\ref{def:domprob}).

\section{Domino Problems for Infinite Computable Sets of Prototiles}

Next, we will define our class of prototiles sets with total planar tilings as to not restrict ourselves to finite sets of prototiles. To this end, we define the set $TILE$ that will range over infinite sets of Wang prototiles.

\begin{definition}\index{$TILE$}\label{def:TILE}
\begin{align*}
TILE = \{ e : \, & \varphi_e \text{ is the characteristic function of some infinite}  \\
	      & \text{ Wang prototile set whose tilings are total in the plane.} \}
\end{align*}
\end{definition}

It is natural from our definition of $TILE$ that for any $e \in TILE$, the tiling that is generated by $e$ must be connected and infinite. 

We also define set $ILL$ which we will use later to get our $\Sigma^1_1$-completeness of $TILE$.

\begin{definition}\label{def:ILL}
$$ ILL = \{ e : \, \varphi_e \text{ is the characteristic function of an ill-founded tree } T \subseteq \omega^{< \omega} \} $$
\end{definition}
Note that by proposition \ref{prop:wellfddR}, specifically the converse argument, $ILL$ is $\Sigma^1_1$-complete.

\subsection{Filter for Computable Trees}

In order to adequately satisfy \ref{thm:TILE-ILL}, it is critical that our computable functions $\Phi_e$ do indeed actually compute trees. As such, we will need the following lemma to `filter out' the functions that do not compute trees.

\begin{lemma}[C. 2019]\label{lemma-pproc}
There is a computable $g: \omega \rightarrow \omega$ such that for every characteristic function $\varphi_e$ of some set $T \subseteq \omega^{<\omega}$:
\begin{enumerate}
\item if $\varphi_e$ is a tree, then $\varphi_{g(e)}$ is the same tree.
\item if $\varphi_e$ is total but not a tree, then $\varphi_{g(e)}$ is not total.
\item if $\varphi_e$ is not total then $\varphi_{g(e)}$ is not total.
\end{enumerate}
\end{lemma}

\begin{proof}
For any $\varphi_e$ define $g(e)$ as follows:
$$ 
    \varphi_{g(e)}(\sigma)= 
\begin{cases}
    1 & \text{if } \forall \tau \subseteq \sigma \  (\varphi_e(\tau)=1) \\ 
    0 & \text{if } \exists \tau \subseteq \sigma \text{ s.t. } \\ 
    & \ \ \ \forall \eta (\eta \subset \tau \rightarrow \varphi_e(\eta) = 1 \land \tau \subseteq \eta \rightarrow \varphi_e(\eta)=0 ) \\
    \uparrow & \text{otherwise}
\end{cases} $$
\end{proof}

\section{$\Pi^1_1$ and $\Sigma^1_1$ Domino Problems}

We will now present our results that show some equivalences between the domino problem for infinite prototile sets and well-founded trees.

\subsection{Equivalences to $TILE$}

\begin{theorem}[C. 2019]\label{thm:TILE-ILL}\index{$TILE$ equivalence to $ILL$}
\[ TILE \equiv_m ILL \]
\end{theorem}

\begin{proof}
Firstly, we note that it follows from $\Sigma^1_1$-completeness of $ILL$ that anything $ILL$ is $m$-reducible to will be $\Sigma^1_1$-complete as well, and so anything in $ILL$ will likewise be found in the set we are reducing to. Thus, we get the converse $m$-equivalence essentially `for free' from this fact and a opposite argument to that found in lemma \ref{lemma:wellfddTreesPi11}.

As such, it suffices to prove $ILL \leq_m TILE$. For this, we will follow the shape of regular $m$-reducibility proofs, and show that there is a computable function $h$ such that $$\forall e (x \in ILL \iff h(x) \in TILE)$$.

We first fix the following colours/symbols:
\begin{itemize}
\item Let $\lambda$ denote the empty string, and let $\lambda^U, \lambda^D$ be unique colours.
\item Fix $M^L_0$ and $M^R_0$ as unique colours.
\item Fix unique colours for all $M_i$ for $i \in \omega$.
\item For $j \in \{1,2,3,4\}$ and $i \in \omega$, let each $c^j_i$ be unique colours.
\item Let $\alpha \in \omega^\omega$ be an infinite string, and for all $i \in \omega$ let $\sigma_i \in \omega^{< \omega}$ denote successive initial segments of $\alpha$ of length $i$ such that $\sigma_0 \prec \sigma_1 \prec \ldots \sigma_i \prec \ldots \prec \alpha$.
\item Let $\sigma_0 = \lambda$ by this notation.
\item For $\sigma \in \omega^{< \omega}$, let $\sigma^\frown n$ denote $\sigma$ concatenated with $n$ as defined before for some $n \in \omega$, and let $| \sigma |$ denote the length of $\sigma$.
\end{itemize}
With these defined, let $e \in ILL$ be given. We will construct the following schema tiles:

We start with the \textbf{root tile}:
\begin{center}
\sampletile{$M^L_0$}{$\lambda^U$}{$M^R_0$}{$\lambda^D$}
\end{center}

Next, we require {\bf column tiles}:

\begin{center}
\sampletile{$c^1_{i+1}$}{$\sigma_i^\frown n$}{$c^2_{i+1}$}{$\sigma_i$}
\sampletile{$c^4_{i+1}$}{$\sigma_i$}{$c^3_{i+1}$}{$\sigma_i^\frown n$}
\end{center}

We also define {\bf mid-row} tiles to be:

\begin{center}
\sampletilenearlabels{$M_{i+1}$}{$c^1_{i+1}$}{$M_i$}{$c^4_{i+1}$}
\sampletilenearlabels{$M_i$}{$c^2_{i+1}$}{$M_{i+1}$}{$c^3_{i+1}$}
\end{center}

We shall additionally define the following diagonal {\bf quadrant filling} tiles:

\begin{center}
\sampletile{$c^1_{i+1}$}{$c^1_{i+1}$}{$c^1_i$}{$c^1_i$}
\sampletile{$c^2_{i}$}{$c^2_{i+1}$}{$c^2_{i+1}$}{$c^2_i$}
\sampletile{$c^3_{i}$}{$c^3_{i}$}{$c^3_{i+1}$}{$c^3_{i+1}$}
\sampletile{$c^4_{i+1}$}{$c^4_{i}$}{$c^4_{i}$}{$c^4_{i+1}$}
\end{center}

We now construct a `library' $\mathcal{S}$ from which we will select the prototiles we need. To generate $\mathcal{S}$ we take all of the colours we fixed at the start of the proof, and colour the schema tiles as follows:
\begin{itemize}
\item We colour the root tile with the tuple $\langle M^L_0, \lambda^U, M^R_0, \lambda^D \rangle$ and put this tile into $\mathcal{S}$.
\begin{itemize}
\item \textbf{NB} - our root tile has distinctions for up/down and left/right in order to prevent trivial $S_e$-tilings using only the root tile.
\end{itemize}
\item For all the $c^j_i$ and $M_i$ colour the mid-row tiles.
\begin{itemize}
\item We must be careful to put the $M^L_0$ and $M^R_0$ tiles such that they will tile from the root tile.
\item specifically, we add the tiles $\langle M_1, c^1_1, M^L_0, c^4_1 \rangle$ and $\langle M^R_0, c^2_1, M_1, c^3_1 \rangle$.
\end{itemize}
\item For all $c^j_i$ colour all of the quadrant tiles, and put these into $\mathcal{S}$.
\end{itemize}

What now remains is to colour the column tiles and add the required ones to $\mathcal{S}$. To do this we will need to take our $e$ and ensure that it has been put through our pre-processing lemma \ref{lemma-pproc} in order to ensure it is a tree.

With this done, we have an $h$ that we will now use to construct a set of prototiles $S_e \subset \mathcal{S}$ as follows:
\begin{itemize}
\item Select all of the mid-row and quadrant filling tiles, along with the root tile, and add these into $S_e$. 
\item Next add all of the column tiles for all $\sigma_n \in \omega^{< \omega}$ such that $\varphi_e(\sigma_n)=1$.
\end{itemize}

We choose all of the column tiles such that there are two copies of each $\sigma_n \text{ such that } \varphi_e(\sigma_n) = 1$; one copy going up from the root tile, with $\sigma_0 = \lambda^U$ and one going down from the root tile with $\sigma_0 = \lambda^R$.

We now want to verify that for each $e \in ILL$ we will get an $S_e$ such that there exist $S_e$ tilings of the plane, giving $h(e) \in TILE$.

To see this, we first note that the quadrant tiles, root tile, and mid-row tiles form a near-complete tiling of the plane. Without the column tiles, we can tile the left and right halves of the plane, meaning that whether or not we have a total function $\Phi^{p}: \mathbb{Z}^2 \rightarrow S_e$ (defined below) is dependant on whether this central column is fully tiled. We now show that this is dependent on there being an infinite path through the tree computed by $\varphi_e$.

So show that this is the case, let $T_e$ be the tree computed by $\varphi_e$ - this is guaranteed by lemma \ref{lemma-pproc}. Given $e \in ILL$ it follows that there is an infinite $p \in [T_e]$. Thus, for all $n \in \omega$ there is some string $\sigma_n = p \upharpoonright n$. Given we added all of these $\sigma_n$ strings into $S_e$ as tiles that cover both the up and down directions from the root tile, $\varphi_{h(e)}$ will have contained all of the tiles that represent $\sigma_0 \prec \sigma_1 \prec \sigma_2 \prec \ldots p$ - in fact, there will be precisely two copies. Given $p$ is infinite, these column tiles will thus complete our tiling, making our $S_e$-tiling total in the plane.

Indeed, taking such a $p \in [T_e]$ as our oracle, for all $x,y \in \mathbb{Z}$, and given the output of $\varphi_{h(e)}$ from above as $S_e$, we define $\Phi^p$ as a fully as a total function $$\Phi^p : \mathbb{Z}^2 \rightarrow S_e$$ which can be fully defined algorithmically as follows:

\begin{itemize}
\item For $\Phi^p(0,0)$ will return the root tile, $\langle M^L_0, \lambda^U, M^R_0, \lambda^D \rangle$
\item For $\Phi^p(x,y)$, where $x,y \neq 0$, we will return the relevant quadrant tile.
\item For $\Phi^p(x,0)$ we will return the correct middle-row tile of the form:
\begin{itemize}
\item if $x$ is positive: $\langle M_{x-1},c^2_x,M_x,c^3_x \rangle$
\item if $x$ is negative: $\langle M_{x-1},c^1_x,M_x,c^4_x \rangle$
\end{itemize}
\item For $\Phi^p(0,y)$ we will use that $\sigma = p \upharpoonright y$, and then return the correct column tile of the form:
\begin{itemize}
\item if $y$ is positive: $\langle c^1_y, \sigma, c^2_y, \sigma \upharpoonright y-1 \rangle$
\item if $y$ is negative: $\langle c^4_y, \sigma \upharpoonright y-1, c^3_y, \sigma \rangle$
\end{itemize}
\end{itemize}

To show that $h(e) \in TILE \Rightarrow e \in ILL$ we first note that if $\Phi^p$ is total, then $\varphi_e$ must also be total - as such, if there are no gaps in our $S_e$-tiling following our construction of $S_e$, then it suffices to show that we can computably recover an infinite $p$ from an $S_e$-tiling for which we can assume that $e \in ILL$.

Let $\mathcal{I}$ be the class of all $S_e$-tilings of the plane. We take one total tiling $I \in \mathcal{I}$ - clearly existing by our assumption that $h(e) \in TILE$ - and try to recover an infinite path $p \in [T_e]$, where $T_e$ is again the tree computed by $\varphi_e$. Our goal is to use the tiling to show whether or not $e \in ILL$.

The following computable method will be our attempt to extract the path $p$ from our $S_e$-tiling:
\begin{enumerate}
\item If we choose the root tile, read upwards along the column of tiles, from which we can recover a path $p$. 
\item If we choose a mid-row tile, then we follow the descending chain of $M_i$ colours to the root tile, and then go to 1.
\item If we choose a quadrant tile, then for our given $i \in \omega$ from our chosen tile:
\begin{itemize}
\item If $c^1_i$ or $c^2_i$ then follow all the tiles down to the mid-row tiles, and go to 2.
\item If $c^3_i$ or $c^4_i$ then follow all the tiles up to the mid-row tiles, and go to 2.
\end{itemize}
\end{enumerate}

If our $S_e$-tiling $I$ is total, then the resulting $\tau$ from this process is infinite and corresponds to some $p \in [T_e]$. Thus, we have shown that for $h(e) \in TILE$ we can take any $S_e$-tiling and computably recover a path demonstrating that $e \in ILL$.
\end{proof}

\begin{figure}[t]
\begin{center}
\begin{tikzpicture}
	\draw (0,0) rectangle (4,4) node[pos=.5] {$c^4_i$};
	\draw (4,0) rectangle (5,4) node[pos=.5,rotate=-90] {lower copy of $\sigma$};
	\draw (5,0) rectangle (9,4) node[pos=0.5] {$c^3_i$};
	\draw (0,4) rectangle (4,5) node[pos=0.5] {left mid-row $M_{-i}$};
	\draw (4,4) rectangle (5,5) node[pos=0.5] {$\lambda$};
	\draw (5,4) rectangle (9,5) node[pos=0.5] {right mid-row $M_i$};
	\draw (0,5) rectangle (4,9) node[pos=0.5] {$c^1_i$};
    \draw (4,5) rectangle (5,9) node[pos=.5,rotate=-90] {upper copy of $\sigma$};
    \draw (5,5) rectangle (9,9) node[pos=0.5] {$c^3_i$};
\end{tikzpicture}
\caption{Overall shape of our tiling construction in the proof of \ref{thm:TILE-ILL}.}
\label{fig:ShapeTilingTILE-ILL}
\end{center}
\end{figure}

We show in figure \ref{fig:ShapeTilingTILE-ILL} the overall shape of our tiling proposed in the proof of theorem \ref{thm:TILE-ILL}. The $c^j_i$'s occupy the upper left/right and lower left/right quarter planes of $\mathbb{Z}^2$, with the middle rows joining the upper/lower left quarter planes and upper/lower right quarter planes. Thus, our root tile connects the two planes with the paths from a tree coded in the upper and lower columns.

\begin{corollary}[C. 2019]
$TILE$ is $\Sigma^1_1$-Complete.
\end{corollary}

\begin{proof}
This follows immediately from the combination of facts that $TILE$ is $m$-equivalent to a $\Sigma^1_1$-complete set, namely $ILL$, which we obtain by the opposite argument shown in corollary \ref{cor:CompTreesPi11}. As such, everything expressible in $ILL$ is also expressible in $TILE$, so every $a \in \Sigma^1_1$ has some representation in $TILE$.
\end{proof}

We should point out that a key part of this proof is that we have not restricted ourselves to finite sets of prototiles, which we know from theorem \ref{thm:TMTilings} is $\Sigma^0_1$ complete. By allowing ourselves to consider infinite sets of prototiles, we have found a way to get $\Sigma^1_1$ completeness by a proof that gives an equivalence between familiar objects, namely the ill-foundedness of trees. In a sense, this result could be entirely expected.

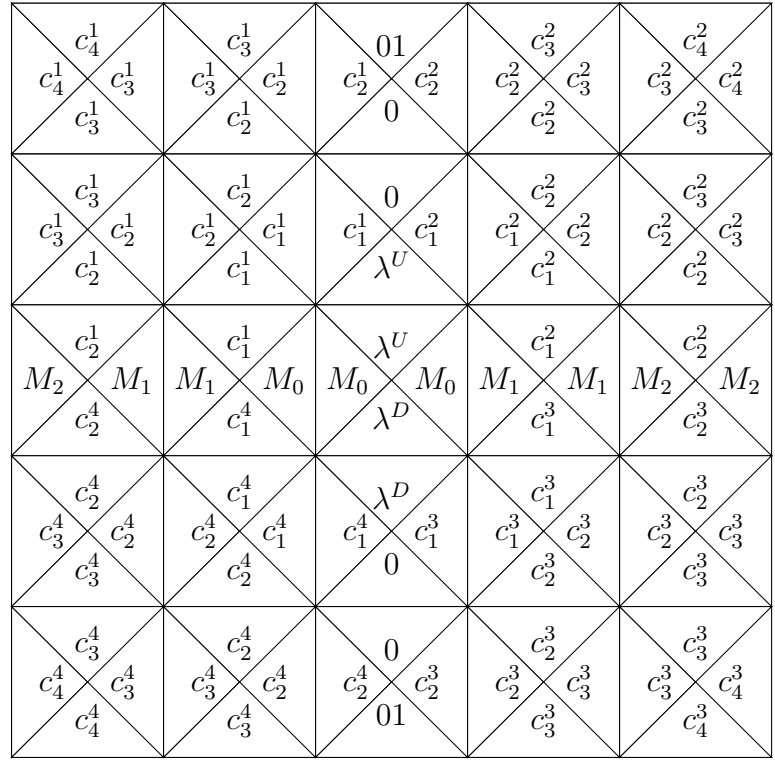
\begin{figure}[t]
  \centering
\begin{tikzpicture}
\begin{scope}[scale=2]
\foreach \x/\y/\l/\u/\r/\b in {0/0/$c^1_4$/$c^1_4$/$c^1_3$/$c^1_3$,1/0/$c^1_3$/$c^1_3$/$c^1_2$/$c^1_2$,2/0/$c^1_2$/$01$/$c^2_2$/$0$,3/0/$c^2_2$/$c^2_3$/$c^2_3$/$c^2_2$,4/0/$c^2_3$/$c^2_4$/$c^2_4$/$c^2_3$,
                               0/1/$c^1_3$/$c^1_3$/$c^1_2$/$c^1_2$,1/1/$c^1_2$/$c^1_2$/$c^1_1$/$c^1_1$,2/1/$c^1_1$/$0$/$c^2_1$/$\lambda^U$,3/1/$c^2_1$/$c^2_2$/$c^2_2$/$c^2_1$,4/1/$c^2_2$/$c^2_3$/$c^2_3$/$c^2_2$,
                               0/2/$M_2$/$c^1_2$/$M_1$/$c^4_2$,1/2/$M_1$/$c^1_1$/$M_0$/$c^4_1$,2/2/$M_0$/$\lambda^U$/$M_0$/$\lambda^D$,3/2/$M_1$/$c^2_1$/$M_1$/$c^3_1$,4/2/$M_2$/$c^2_2$/$M_2$/$c^3_2$,
                               0/3/$c^4_3$/$c^4_2$/$c^4_2$/$c^4_3$,1/3/$c^4_2$/$c^4_1$/$c^4_1$/$c^4_2$,2/3/$c^4_1$/$\lambda^D$/$c^3_1$/$0$,3/3/$c^3_1$/$c^3_1$/$c^3_2$/$c^3_2$,4/3/$c^3_2$/$c^3_2$/$c^3_3$/$c^3_3$,
                               0/4/$c^4_4$/$c^4_3$/$c^4_3$/$c^4_4$,1/4/$c^4_3$/$c^4_2$/$c^4_2$/$c^4_3$,2/4/$c^4_2$/$0$/$c^3_2$/$01$,3/4/$c^3_2$/$c^3_2$/$c^3_3$/$c^3_3$,4/4/$c^3_3$/$c^3_3$/$c^3_4$/$c^3_4$}
{
\draw[fill=white] (1+\x,1-\y) rectangle (0+\x,0-\y);
\filldraw[fill=white] (0+\x,0-\y) -- (0.5+\x,0.5-\y) -- (0+\x,1-\y) -- cycle;
\filldraw[fill=white] (0+\x,0-\y) -- (0.5+\x,0.5-\y) -- (1+\x,0-\y) -- cycle;
\filldraw[fill=white] (0+\x,1-\y) -- (0.5+\x,0.5-\y) -- (1+\x,1-\y) -- cycle;
\node at (0.5+\x,0.5-\y) [label=above:{\u},label=left:{\l},label=right:{\r},label=below:{\b}] {};
}
\end{scope}
\end{tikzpicture}
  \caption{Tile Path Construction}
  \label{fig:TT1}
\end{figure}

Figure \ref{fig:TT1} shows an example of a patch around the root tile for some $S_e$-tiling generated by the above algorithm. The first two bits of a path $\sigma$, with $\sigma \upharpoonright 2 = $`01'. Note that we can see in this diagram that if $| \sigma | < \omega$ then there will be gaps at some point going up/down from the root tile, there by such an $e$ will not be total, and so $e \notin TILE$.

\begin{definition}\label{def:WELL}
We define the set of well-founded computable trees: $$WELL = \{ e : \varphi_e \text{ is the characteristic function of a well-founded tree } T \subseteq \omega^{< \omega} \}$$
\end{definition}
Recall that by proposition \ref{prop:wellfddR} it follows that $WELL$ is $\Pi^1_1$-complete, which is an important fact we will use.

We let $\neg TILE$ be the set of computable characteristic functions of infinite sets of prototiles that do not have total tilings of plane. It is interesting that, by the same construction above, we can get that $WELL \equiv_m \neg TILE$, despite unequal complements and totality issues.

\begin{theorem}[C. 2019]\label{thm:nTILE-WELL}\index{$\neg TILE$ equivalence to $WELL$}
\[ (\neg TILE) \equiv_m WELL \]
\end{theorem}

\begin{proof}
We proceed as for the proof of \ref{thm:TILE-ILL} - it suffices to show $WELL \leq_m \neg TILE$ as $(\neg TILE) \leq_m WELL$ will follow then by $\Pi^1_1$-completeness of $WELL$ and lemma \ref{lemma:wellfddTreesPi11}. Given this, we want computable $h$ such that $$ e \in WELL \iff h(e) \in \neg TILE $$

We derive the same $S_e \subset \mathcal{S}$ as we derive in the previous proof. Thus we have an $h$ such that $\Phi^p : \mathbb{Z}^2 \rightarrow S_e$ is given for any path $p \in [T_e]$.

If we have some $e \in WELL$, then by our construction, it must be the case that $\varphi_{h(e)}$ would not give a total tiling of the plane as the well-foundedness of $T_e$ would give that there is no infinite $p \in [T_e]$. Thus, there is no set of column tiles in $S_e$ that will tile the central column of our tilings. Thus it follows that $h(e) \in \neg TILE$.

Now suppose that we have some $h(e) \in \neg TILE$, and let $\mathcal{I}$ be the class of all $S_e$-tilings of the plane. For any given $I \in \mathcal{I}$ we know that $I$ is not a total tiling of the plane, but we know that by our construction both halves of the plane about the central column will be computably tiled. Thus, the gaps in our tiling that make it non-total must be along this central column for all $I \in \mathcal{I}$.

Given this central column is composed of tiles that code paths in $[T_e]$, it must be the case that there is no output of $\varphi_e$ that is an infinite path $p \in [T_e]$. Thus it follows that if $h(e) \in \neg TILE$ then $e \in WELL$.
\end{proof}

As we shall see in the next section, this construction gives rise to some interesting implications when it comes to equivalences of free Domino Problems and infinite sets of prototiles. 

\subsection{Further Equivalences for $WELL$ and $ILL$}

It was found that the equivalences in the previous section were not the only ones  we could construct when we consider infinite sets of prototiles. Indeed, when we consider other free Domino Problems, we can prove further equivalences using a similar framework. In order to do this analysis, we need the following definitions. 

\begin{definition}\label{def:WTILE}
\begin{align*}
WTILE = \{ e : \, & \varphi_e \text{ is the char. func. of a Wang prototile set that has tilings} \\
	       & \text{that are infinite, connected, but not necessarily total} \}
\end{align*}
\end{definition}

$WTILE$ is short for \emph{`weakly-tile'}, and intuitively stands for infinitely connected, but not total tilings. This notion of \emi{weakly tiling} the plane gives us a natural notion of \emi{strongly not tiling} the plane, which we define as follows:

\begin{definition}\label{def:SNT}
\begin{align*}
SNT = \{ e : \, & \varphi_e \text{ is the char. func. of a Wang prototile set whose} \\
		& \text{connected tilings are finite} \}
\end{align*}
\end{definition}

Intuitively we can think of $WTILE$ tilings as being everything in $TILE$ but plus other tilings up to infinite connected `snakes' of tiles that are connected. Though we are now considering tilings that are no longer necessarily total, the fact that they are infinite and connected is the key property we wish to analyse.

On the other hand, $SNT$ denotes tilings that form (potentially infinitely many) disconnected patches of tiles. We can picture disconnected colonies of mould, for example, as an intuition for what these tilings can look like.

As such, prototile sets that are in $SNT$ are necessarily disconnected, whereas tilings in $WTILE$ are necessarily connected, in a graph theoretic sense. We can use the following construction to analyse tilings of infinite sets of prototiles for these properties. Again, we will use $WELL$ and $ILL$ from previous proofs as fundamental tools.

\begin{theorem}[C. 2019]\label{thm:SNT-WELL}\index{$SNT$ equivalence to $WELL$}
\[ SNT \equiv_m WELL \]
\end{theorem}

\begin{proof}
As before, we denote Wang prototiles through the 4-tuple $\langle l,u,r,b \rangle$, and for $\sigma \in \omega^{\omega}$, let $\sigma(n)$ denote the $n^{th}$ symbol of $\sigma$.

We prove these equivalences sequentially. Similarly to the previous proof, it follows from the $\Pi^1_1$-completeness of $WELL$ that for a $\Pi^1_1$ set $A$, $$(WELL \leq_m A) \rightarrow (A \equiv_m WELL)$$ As such, it suffices to show that $WELL \leq_m SNT$, as $SNT \leq_m WELL$ will follow from this, giving our $m$-equivalence.

We want some computable $g$ such that $$ e \in WELL \iff g(e) \in SNT $$ which will give us our $m$-reduction.

In order to carry out this proof, we will need to fix the following colours/symbols:
\begin{itemize}
\item Let $\lambda$ denote the empty string as before, and fix unique colours $\lambda^U, \lambda^D, \lambda^L$, and $\lambda^R$.
\item For $\sigma \in \omega^{< \omega}$ let $| \sigma |$ denote the length of $\sigma$, 
\item Let $\sigma^\frown n$ denote the concatenation of $\sigma$ with some $n \in \omega$.
\item For $j \in \{ 1,2,3,4 \}$ and $n \in \omega$ fix colours $\sigma^j_n$ for every $\sigma$.
\item Let $\sigma \in \omega^\omega$, and for all $i \in \omega$ let $\sigma_i \in \omega^{< \omega}$ denote successive initial segments of $\sigma$ of length $i$ such that $\sigma_0 \prec \sigma_1 \prec \ldots \prec \sigma$.
\item Let $\sigma_0 = \lambda$ by our notation above. 
\end{itemize}

With these colours and symbols fixed, let $e \in WELL$ be given. We construct the following schema tiles:

We start with the \textbf{root tile}:

\begin{center}
\sampletile{$\lambda^L$}{$\lambda^U$}{$\lambda^R$}{$\lambda^D$}
\end{center}

We will also need \textbf{middle column and row tiles}:

\begin{center}
\sampletile{$s^1$}{$s$}{$s^2$}{$\sigma$}
\sampletile{$\sigma$}{$s^2$}{$s$}{$s^3$}
\sampletile{$s^4$}{$\sigma$}{$s^3$}{$s$}
\sampletile{$s$}{$s^1$}{$\sigma$}{$s^4$}
\end{center}
Where $s = \sigma^\frown n$ for $\sigma \in \omega^{< \omega}$ and $n \in \omega$.

Lastly, we will also require \textbf{quadrant filling tiles}:

\begin{center}
\sampletile{$s^1_{i+1}$}{$s^1_{i+1}$}{$s^1_{i}$}{$s^1_{i}$}
\sampletile{$s^2_{i}$}{$s^2_{i+1}$}{$s^2_{i+1}$}{$s^2_{i}$}
\sampletile{$s^3_{i}$}{$s^3_{i}$}{$s^3_{i+1}$}{$s^3_{i+1}$}
\sampletile{$s^4_{i+1}$}{$s^4_{i}$}{$s^4_{i}$}{$s^4_{i+1}$}
\end{center}
Where for $j \in \{ 1,2,3,4 \}$ we have that $s^j_i = \sigma \in \omega^{< \omega}$ of length $i$, and $s^j_{i+1} = \sigma^\frown n$ for some $n \in \omega$ as before. Each colour $s^j_i$ thereby encodes some string in $\omega^{< \omega}$, and $s^j_{i+1}$ is the extension of this by 1 character, and both are initial segments of some infinite path.

We can now construct a library $\mathcal{U}$ of tiles from which we will select the relevant ones we need. To generate $\mathcal{U}$ we will take all of the colours we fixed earlier and apply them to the prototile schema above as follows:
\begin{itemize}
\item We colour the root tile with the colours we fixed to get the prototile $\langle \lambda^L, \lambda^U, \lambda^R, \lambda^D \rangle$ and put this tile into $\mathcal{U}$.
\begin{itemize}
\item As before, our root tile has unique colours for each direction to prevent trivial tilings of the plane from the root tile alone. 
\end{itemize}
\item With $j \in \{ 1,2,3,4 \}$, for each initial segment colour $s^j_n$ we fixed earlier, colour all of the possible quadrant tiles and put these into $\mathcal{U}$.
\begin{itemize}
\item For each $\sigma,\tau \in \omega^{< \omega}$, where $\tau = \sigma^\frown i$ is an ancestor for some $\sigma$ with $i \in \omega$, we fix 8 colours:
\begin{itemize}
\item[--] $\sigma^1, \sigma^2, \sigma^3, \sigma^4$
\item[--] $\tau^1, \tau^2, \tau^3, \tau^4$
\end{itemize}
\item We then proceed to create 4 prototiles:
\begin{enumerate}
\item $\langle \tau^1, \tau^1, \sigma^1, \sigma^1 \rangle$
\item $\langle \sigma^2, \tau^2, \tau^2, \sigma^2 \rangle$
\item $\langle \sigma^3, \sigma^3, \tau^3, \tau^3 \rangle$
\item $\langle \tau^4, \sigma^4, \sigma^4, \tau^4 \rangle$
\end{enumerate}
\end{itemize}
\item We also colour for every $\sigma_n \in \omega^{< \omega}$ of length $n$, and every $i \in \omega$ the following column tiles:
\begin{enumerate}
\item $\langle (\sigma_n^\frown i)^1, \sigma_n^\frown i, (\sigma_n^\frown i)^2, \sigma_n \rangle$
\item $\langle \sigma_n, (\sigma_n^\frown i)^2, \sigma_n^\frown i, (\sigma_n^\frown i)^3 \rangle$
\item $\langle (\sigma_n^\frown i)^4, \sigma_n, (\sigma_n^\frown i)^3, \sigma_n^\frown i \rangle$
\item $\langle \sigma_n^\frown i, (\sigma_n^\frown i)^1, \sigma_n, (\sigma_n^\frown i)^4 \rangle$
\end{enumerate}
\end{itemize}

We are now left with a requirement to colour the middle-row and middle-column prototiles. We again ensure that our $e \in WELL$ has been through the pre-processing lemma \ref{lemma-pproc}, which ensures there is a tree $T_e$ computed by $\varphi_e$.

Our $g$ will then construct $U_e \subset \mathcal{U}$ as follows:
\begin{enumerate}
\item Select the root tile, and add this into $U_e$.
\item Select all of the middle column and middle row tiles that correspond to each path $p \in [T_e]$ and add these also into $U_e$.
\item Select from the quadrant filling tiles with the relevant $\tau$ such that for any $\sigma \prec p \in [T_e]$, $\tau$ is the immediate ancestor $\sigma^\frown i$ for $i \in \omega$ such that $\sigma \prec \tau \prec \ldots \prec p$. We then add to this the quadrant tiles we need into $U_e$.
\end{enumerate}

The construction of the prototile set $U_e$ embeds some path $\sigma_n$, where $\varphi_e(\sigma_n) = 1$, 4 times from the root tile - each copy going one of the 4 directions up, down, left, or right, forming `spokes' from the root tile that represent a path through $T_e$. The quadrant tiles are then used to fill in the gaps between these spokes with the intention that we \emph{could} get a total $U_e$-tiling of the plane if $e \notin WELL$. As before, the root tile's empty strings are equivalent to $\sigma_0 = \lambda^L = \lambda^U = \lambda^R = \lambda^D$.

We first want to verify that $$e \in WILL \rightarrow g(e) \in SNT$$ This can be done by analysing the behaviour of the tiling function $\Psi^p : \mathbb{Z}^2 \rightarrow U_e$.

To do this, let $e \in WELL$ be given. Then we can construct $U_e$ as above, and then observe what will happen in a $U_e$-tiling of the plane. Given the well-foundedness of $\varphi_e$ means that there is no infinite $p \in [T_e]$ such that our $U_e$-tilings would have infinitely long spokes. Given this, each $U_e$-tiling will have a bound on the width and height of the tiling, and as such our tiling function $\Psi^p$ will not be total over $\mathbb{Z}^2$. 

Given this fact, $\varphi_{g(e)}$ will only generate a finite patch tiling that is connected. Thus we can say that $g(e)$ must only have connected tilings that are patches, and so we get $g(e) \in SNT$.

For the converse direction it suffices to show that $$ e \notin WELL \rightarrow g(e) \notin SNT$$ Given $e \notin WELL$ there exists an infinite $p \in [T_e]$ which we will use as our oracle. This follows by construction of $\Psi^p : \mathbb{Z}^2 \rightarrow U_e$ as a total TM as follows - let $\sigma_n = p \upharpoonright n$:
\begin{itemize}
\item For $\Psi^p(0,0)$ we return the root tile $\langle \lambda^L, \lambda^U, \lambda^R, \lambda^D \rangle$
\item For $\Psi^p(x,0)$ we return one of two tiles:
\begin{itemize}
\item If $x$ is negative: $\langle \sigma_n, \sigma^1_n, \sigma_{n-1}, \sigma^4_n \rangle$
\item If $x$ is positive: $\langle \sigma_{n-1}, \sigma^2_n, \sigma_n, \sigma^3_n \rangle$
\end{itemize}
\item For $\Psi^p(0,y)$ we return one of two tiles:
\begin{itemize}
\item If $y$ is negative: $\langle \sigma_n^4, \sigma_{n-1}, \sigma^3_n, \sigma_n \rangle$
\item If $y$ is positive: $\langle \sigma_n^1, \sigma_n, \sigma^2_n, \sigma_{n-1} \rangle$
\end{itemize}
\item For $\Psi^p(x,y)$ such that $x,y \neq 0$, we return tile for the correct quadrant such that $\sigma_{n-1}$ and $\sigma_n$ are present for $n = |x| + |y|$.
\end{itemize}

\textbf{NB} - we substitute $\lambda^L, \lambda^U, \lambda^R, \text{ and } \lambda^D$ as required for $\sigma_0$ to ensure that all of our tiles align in the plane.

With $p \in [T_e]$ infinite, given $e \notin WELL$, then $\Psi^p$ is total, which gives us immediately that there are total planar $U_e$-tilings. Thus, our connected tilings for $U_e$ are not patches, and so $g(e) \notin SNT$. 
\end{proof}

\begin{figure}[t]
  \centering
\begin{tikzpicture}
\begin{scope}[scale=2]
\foreach \x/\y/\l/\u/\r/\b in {2/0/$01^1$/$01$/$01^2$/$0$,
                               1/1/$01^1$/$01^1$/$0^1$/$0^1$,2/1/$0^1$/$0$/$0^2$/$\lambda^U$,3/1/$0^2$/$01^2$/$01^2$/$0^2_1$,
                               0/2/$01$/$01^1$/$0$/$01^4$,1/2/$0$/$0^1$/$\lambda^L$/$0^4$,2/2/$\lambda^L$/$\lambda^U$/$\lambda^R$/$\lambda^D$,3/2/$\lambda^R$/$0^2$/$0$/$0^3$,4/2/$0$/$01^2$/$01$/$01^3$,
                               1/3/$01^4$/$0^4$/$0^4$/$01^4$,2/3/$0^4$/$\lambda^D$/$0^3$/$0$,3/3/$0^3$/$0^3$/$01^3$/$01^3$,
                               2/4/$01^4$/$0$/$01^3$/$01$}
{
\draw[fill=white] (1+\x,1-\y) rectangle (0+\x,0-\y);
\filldraw[fill=white] (0+\x,0-\y) -- (0.5+\x,0.5-\y) -- (0+\x,1-\y) -- cycle;
\filldraw[fill=white] (0+\x,0-\y) -- (0.5+\x,0.5-\y) -- (1+\x,0-\y) -- cycle;
\filldraw[fill=white] (0+\x,1-\y) -- (0.5+\x,0.5-\y) -- (1+\x,1-\y) -- cycle;
\node at (0.5+\x,0.5-\y) [label=above:{\u},label=left:{\l},label=right:{\r},label=below:{\b}] {};
}
\end{scope}
\end{tikzpicture}
  \caption{Weakly Tiling Path Construction}
  \label{fig:WT1}
\end{figure}
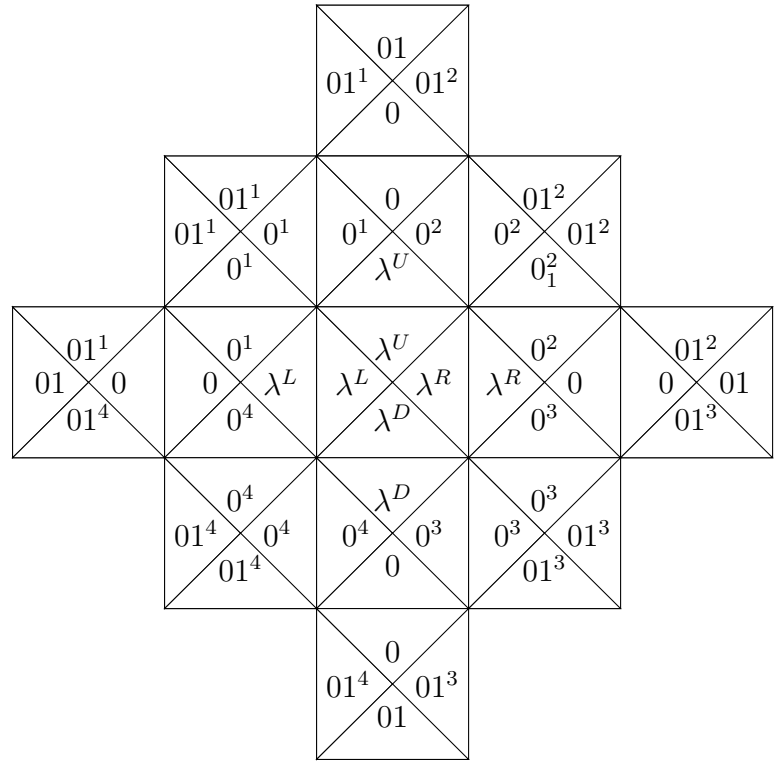

In Figure \ref{fig:WT1} we find the proposed construction, showing the four copies of some path $\sigma$ emanating from the central root tile. The absence of any tiles to complete the edges means that these can never join together to form a complete tiling of the plane - the only way for there to be a total tiling of the plane is for $e \notin WELL$, from which our result follows.

\begin{corollary}[C. 2019]
$SNT \equiv_m WELL$ implies that $SNT$ is $\Pi^1_1$-Complete.
\end{corollary}

\begin{proof}
This follows as a consequence that $SNT$ is equivalent to a $\Pi^1_1$-complete set, $WELL$. As such, every $b \in \Pi^1_1$ will have some representation in $SNT$, and as such, $SNT$ is also $\Pi^1_1$-complete.
\end{proof}

It should be noted that the proofs of theorem \ref{thm:SNT-WELL} could have been shortened to just the tiles that enumerate the paths that we are interested in - however, we will use the construction with filler tiles later in this thesis. We shall also utilise the quadrant filling tiles in this construction in the next theorem.

Following on from theorem \ref{thm:SNT-WELL}, we asked what the relationship to $WTILE$ was, and found that we can state the following theorem:

\begin{theorem}[C. 2019]\label{thm:WTILE-ILL}\index{$WTILE$ equivalence to $ILL$}
\[ WTILE \equiv_m ILL \]
\end{theorem}

\begin{proof}
We get $WTILE \leq_m ILL$ by the $\Sigma^1_1$-completeness of $ILL$. It is sufficient to then show that $ILL \leq_m WTILE$ by our construction for the proof of \ref{thm:SNT-WELL}. As such, we want computable $g$ such that $$ e \in ILL \iff g(e) \in WTILE $$

We derive the same $U_e \subset \mathcal{U}$ as given in the proof of theorem \ref{thm:SNT-WELL}, so we have a $g$ such that $\varphi_{g(e)} : \mathbb{Z}^2 \rightarrow U_e$ computable.

By our construction, we have that if $e \in ILL$ then $\varphi_{g(e)}$ will be a total function from the $\mathbb{Z}^2$ lattice into $U_e$. The resulting tiling will have 4 infinite spokes coming from the root tile, and these are infinite connected tilings that satisfy $g(e) \in WTILE$, even without knowing that $\varphi_{g(e)}$ is total.

For the converse direction it would suffice to show that $$e \notin ILL \rightarrow g(e) \notin WTILE$$ which can be seen through the following argument. Given $e \notin ILL$ then there exists no path $p \in [T_e]$ that is infinite. Thus, when we create $U_e$ by means of $g(e)$, we must create a tile set that has only connected patch of the plane, violating the requirements for $WTILE$, thus $g(e) \notin WTILE$.
\end{proof}

\subsection{Discussion of these Results}

These results differ from previous work by \cite{Harel1986} insofar as they do not rely on any knowledge of the properties of recurrent patterns within a tiling, but rather manage to specifically equate several forms of Domino Problems on infinite sets of prototiles.

It is worth noting some of the following facts about these results:
\begin{enumerate}
\item At no point to we restrict ourselves to requiring to use every tile in a generated prototile set.
\item We do not require any special conditions on how/where our tilings start. 
\end{enumerate}

For 1, it is of interest that we do not require every tile $t \in S_e$ or $t \in U_e$ to be used at all. To this end, we have specifically added extra colours the make specific alignments and prevent trivial planar tilings - specifically from the root tiles we defined. 

For 2, these tilings can be essentially tiled without stating specific initial criteria as we have been very careful to include design elements that essentially force the hand of the tiling function into only admitting certain tilings that code the precise behaviour we want.

Note also that the classes of tilings from either of these prototile sets effectively encode the paths down the trees computed by some $e$, once $e$ has been passed through our tree-filtering lemma \ref{lemma-pproc}.

It is also worth noting that our $m$-equivalences are such that they work despite the mismatch in complements for the sets we concern ourselves with - \emph{e.g.} the complements of $TILE$ and $WTILE$ are quite different, and yet they are both $m$-equivalent to $ILL$. This shows us that infinite computable sets of Wang prototiles are not rich enough to discern the differences that we are mathematically aware of.

\chapter{Aperiodicity, Tilings, and Logical Complexity}
\label{chap5}
\setcounter{equation}{0}
\renewcommand{\theequation}{\thechapter.\arabic{equation}}


\epigraph{Everything is simpler than you think and at the same time more complex than you imagine.}{\textit{Goethe (attrib.)}}

In this chapter we will explore and present results relating to tiling problems that ask about properties of total planar tilings - specifically whether they are periodic or aperiodic. We present first an overview of past results, and then provide new results inspired by our work in Chapter 3, culminating in a completeness result between periodicity/aperiodicity in infinite prototile sets and the class of problems of the form $(\Pi^1_1 \wedge \Sigma^1_1)$.

\section{Aperiodic Tilings and $\Sigma^1_1$/$\Pi^1_1$ Sets}

We will now look at aperiodicity in tilings and uncover some interesting facts about the $m$-reducibility of previously defined sets $WELL$ and $ILL$ to periodic and aperiodic tiling problems.

\subsection{Definitions of Periodic and Aperiodic Tilings}

We will use the following definitions in our analysis of aperiodic prototile sets derived from our definitions in Chapter 3.

\begin{definition}[Periodic Tilings]\label{def:PTile}
A tiling $T$ of the plane is a \emi{periodic tiling} iff there exists some non-zero vector $\mathbf{v}$ such that $\mathbf{v}$ defines a shift of $T$ such that $$ T = \mathbf{v} T $$

A set of prototiles $\mathcal{S}$ is \emph{periodic} iff it admits only periodic tilings of the plane. For computable $e$, let $PTile$ be as follows 
\begin{align*}
PTile = \{ e: & \, \varphi_e \text{ is the characteristic function for a set of prototiles} \\
	& \, \text{whose tilings are all periodic total tilings.} \} 
\end{align*}
\end{definition}

Our requirement that a periodic, set of prototile has \emph{only total} tilings that meet these criteria is how we avoid trivial periodic tilings by means of tilings that only tile some finite portion of the plane. 

Analogously we have the following definition for aperiodic tilings:

\begin{definition}\label{def:ATile}
A tiling $T$ of the plane is an \emi{aperiodic tiling} iff for any vector $\mathbf{v}$ necessary that $T \neq \mathbf{v} T$. Similarly, a set of prototiles $\mathcal{S}$ is \emph{aperiodic} iff it admits no periodic tilings of the plane.

For computable $e$, let $ATile$ be as follows:
\begin{align*}
ATile = \{ e: \varphi_e & \, \text{ is the characteristic function for a set of prototiles} \\
	&  \,\text{whose tilings are only aperiodic total tilings.}  \}
\end{align*}
\end{definition}

It is worth recalling that Simpson's equivalence of tiling problems on Wang prototiles with 2-dimensional subshifts of finite type in \cite{Simpson2007} is prophetic with respect to extending our gaze beyond domino problems and into questions of the existence of shifts of total tilings themselves.

\subsection{Overview of Aperiodicity} 

Whilst periodic tilings have been around since ancient times - of which a plethora of examples mathematical significance can be found in \cite{GrunbaumTP} - aperiodicity is relatively new. We will first discuss the origins of aperiodic tilings sets, and then set the scene and context in which some famous aperiodicity results find themselves.

\subsubsection{Origins of Aperiodic Prototile Sets}

As documented in \cite[P.520-600]{GrunbaumTP}, the study of aperiodicity in tilings did not occur until Robinson proved that such tilings must necessarily exist in 1968. Conway, Amman, and Penrose all made headways in the study of aperiodicity in tilings. One such result can be found in the following definitions and proposition - for which we shall use the presentation in \cite{Shen2010}:

\begin{definition}\label{def:macro-tile}
Let $S$ be a finite set of prototiles. Then a \emi{macro tile} is a square of size $n \times n$ filled with matching tiles from $S$.
\end{definition}

\begin{definition}
Let set of prototiles $S$ and a set of macro tiles $M$ be given. We say that $S$ \emph{implements} $M$ if any $S$-tiling can be split by horizontal and vertical cuts into macro-tiles $m \in M$.
\end{definition}

\begin{definition}
A set of prototiles $S$ is a \emi{self-similar prototile set} if it implements some macro-tile set $M$, with $M$ isomorphic to $S$, which we shall write $M \cong S$.
\end{definition}

Here, `isomorphic' means that we can find a one to one correspondence between the sets of $M$ and $S$ prototiles - that is, for some $m \in M$, we can find a corresponding $s \in S$ such that under a chosen mapping of the edge conditions of $m$, $s$ has the same edge conditions.

Note, that if $n$ exists and $S$ is self-similar, then $S$ will have total tilings of the plane, as for any  patch tiling, we can inflate the tilings with the substituted macro tilings to obtain arbitrarily large tilings of the plane by compactness. Though, we shall lose this compactness argument when we graduate from finite to infinite prototile sets.

\begin{proposition}[\protect{\cite[Sec. 4]{Shen2010}}]\label{prop:SHEN-Aperiodic}
A self-similar prototile set $S$ has only aperiodic tilings.
\end{proposition}

\begin{proof}
Proof from \cite{Shen2010}. Suppose for contradiction that a self-similar prototile set $S$ is periodic. We let $p \in \omega$ be the period of some $S$-tiling $T$. By definition, $T$ can be split uniquely into macro-tiles from $M \cong S$ by $n \times n$ cuts, for some unique $n \in \omega$. A shift by $p$ should respect this splitting, else we get a different splitting, so $p$ must be some multiple of $n$.

`Zooming out' from our tiling, by which we mean rescaling our tiling by some fixed factor, we can proceed in replacing each $M$ macro-tile by its corresponding $S$ tile, we get a $\frac{p}{n}$ shift of $T$. However, by the same reasoning $\frac{p}{n}$ must also be a multiple of n, so we can zoom out again, and continue this construction.

We must therefore conclude that $p$ is a multiple of $n^k$ for any $k$, meaning that $p$ is a zero vector. $\rightarrow \leftarrow$
\end{proof}

The classic instance of such results can be found in Penrose Tilings, specifically the presentation from \cite{Gummelt1996}, and the original article by Penrose in \cite{Penrose1979} - wherein Penrose shows how you can acquire aperiodic tilings of the plane from as few as two prototiles. Indeed, two distinct but related prototile sets are given: the Penrose Rhombi and Penrose Kite and Dart prototile sets.

Interesting tangents of study that have derived from the study of aperiodic tile sets has been found in the study of quasicrystals - crystalline lattice structures that are ordered but not periodic. Penrose tilings have been found to have given some insight into the icosahedral phases of quasicrystals - see \cite{Mackay1987}.

Their proofs of aperiodicity follow as analogous arguments to the above - by showing that the Penrose constructions `deflate' and `inflate' to copies of the tiling, we show that we can tile every arbitrary finite portion of the plane. Thus, by a basic compactness argument, we find that Penrose prototiles tile the plane. However, if they do so, then the inflation/deflation processes give the same bi-simulation argument as given by proposition \ref{prop:SHEN-Aperiodic}. As such, any Penrose tiling must then also be invariant under any linear shift, else they would fail to be self-similar in the way that there are, and so Penrose tilings are aperiodic.

There is a fantastic treatment of the underlying algebraic theory by de Bruijn in two papers: \cite{DEBRUIJN198139}, followed by \cite{DEBRUIJN198153} - both are dedicated to P\'{o}lya. The theory is quite exceptionally beautiful, but beyond the scope of this thesis to include. The essential idea that was given in this work is called the `cut and project' method, where a five-dimensional lattice is projected through a `window' onto the plane in order to acquire the corner points of a Penrose tiling. The original results can be found in \cite{DEBRUIJN198139}, with an excellent overview of this work and its relationship to actual physical phenomena can be found in the work in Au-Yang \etal \cite{AuYang2013}.

The existence of precisely 8 corner configurations in any Penrose tiling is also given in \cite{DEBRUIJN198153}, which is again work that is worthy of study but beyond the scope of this thesis. 

\begin{figure}[t]
\includegraphics[scale=0.75]{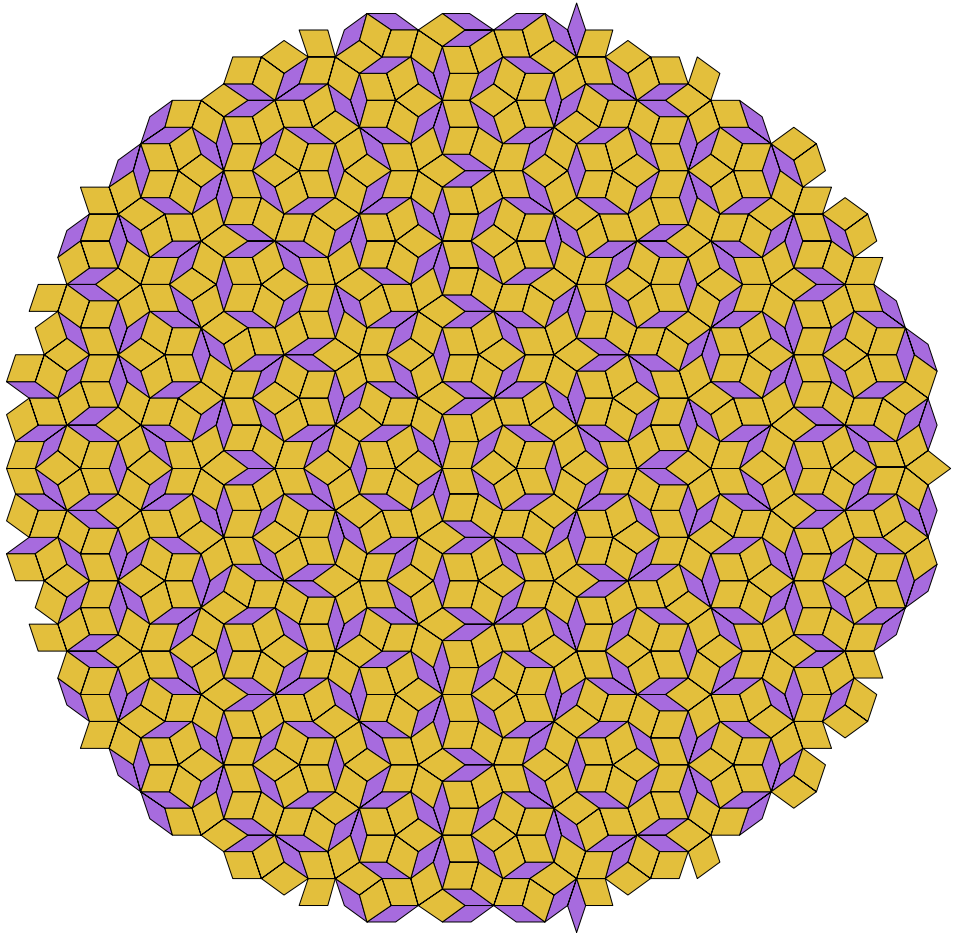}
\caption{A Penrose Tiling - generated online at \url{https://misc.0o0o.org/penrose/}}
\label{fig:PPT}
\end{figure}

In the continuation of their work we outlined above, Shen \etal in \cite{Shen2010} produced some very novel conditions under which aperiodic tilings could be found by means of fixed points - they show that it is possible to have some predicate $S$ that is isomorphic to the set of tiles $T$ that is used to implement it. This is analogous to the challenge of creating Quines in computer science - that is, computer programs whose output upon being run is to print their own source code. Just as Quines are necessarily existing, so are these Shen fixed-point tilings.

\subsubsection{Aperiodic Wang Prototiles}

As we quoted in the introduction, Simpson in \cite{Simpson2007} draws the equivalence between tiling problems in Wang prototile sets and 2-dimensional subshifts of finite type. Utilising this as our base intuition, we present now an overview of aperiodicity in Wang tiles, for which there have been some very interesting and recent developments.

Building on from this basis, the question was asked about what the \emph{smallest} aperiodic Wang prototile sets might be. The survey in \cite{Rao2015} gives a fascinating timeline: Berger originally came up with a set of 20,426 Wang prototiles that was aperiodic for his thesis. By \cite{GrunbaumTP}, a set of 24 aperiodic Wang prototiles was presented, with improvements by Robinson and Amman along the way.

\begin{figure}[t]
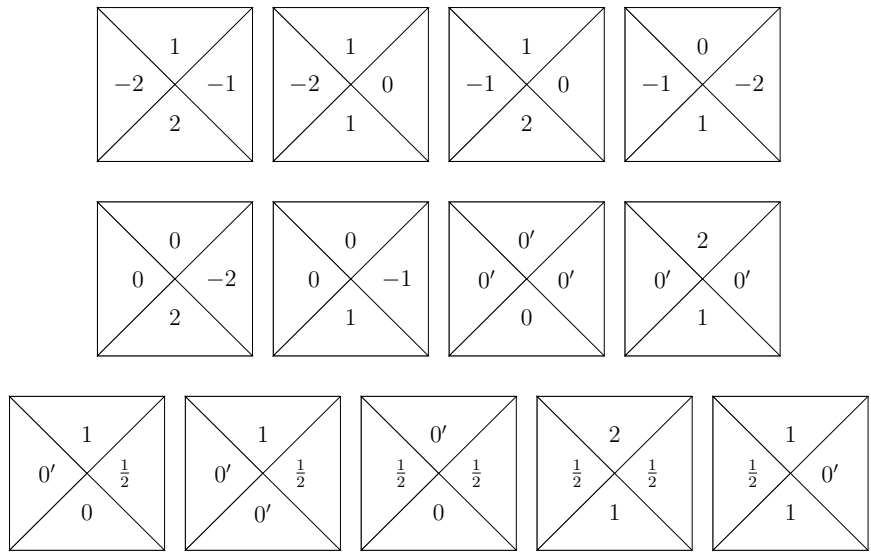

\begin{center}
\sampletile{$-2$}{$1$}{$-1$}{$2$}
\sampletile{$-2$}{$1$}{$0$}{$1$}
\sampletile{$-1$}{$1$}{$0$}{$2$}
\sampletile{$-1$}{$0$}{$-2$}{$1$}

\hspace{0.5cm}

\sampletile{$0$}{$0$}{$-2$}{$2$}
\sampletile{$0$}{$0$}{$-1$}{$1$}
\sampletile{$0'$}{$0'$}{$0'$}{$0$}
\sampletile{$0'$}{$2$}{$0'$}{$1$}

\hspace{0.5cm}

\sampletile{$0'$}{$1$}{$\frac{1}{2}$}{$0$}
\sampletile{$0'$}{$1$}{$\frac{1}{2}$}{$0'$}
\sampletile{$\frac{1}{2}$}{$0'$}{$\frac{1}{2}$}{$0$}
\sampletile{$\frac{1}{2}$}{$2$}{$\frac{1}{2}$}{$1$}
\sampletile{$\frac{1}{2}$}{$1$}{$0'$}{$1$}
\end{center}
	\caption{A set of 13 aperiodic Wang prototiles due to Culik \cite{CULIK1996245}.}
\label{fig:13-ATile}
\end{figure}

After a result by Kari \cite{Kari1996}, it was Culik who set a record in \cite{CULIK1996245} - an aperiodic set of 13 Wang prototiles, which we have included in figure \ref{fig:13-ATile}. These were derived from the states of automata transducers which can compute non-repeating reals. As such, any prototile set coding this behaviour will likewise be non-repeating, thereby aperiodic. 

The most significant breakthrough in this area has been a recent publication from Jeandel and Rao in \cite{Rao2015} where they proved the following two important results:

\begin{theorem}[\protect{\cite[Thm. 5]{Rao2015}}]
There exists an aperiodic set of 11 Wang prototiles.
\end{theorem}

\begin{figure}[b]
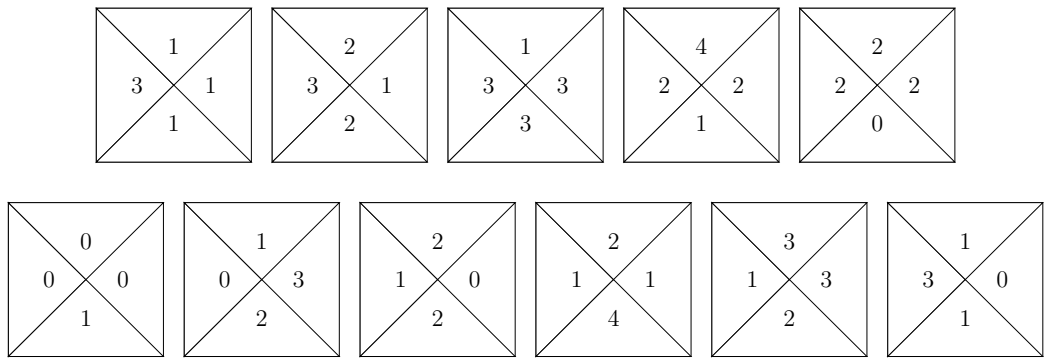

\begin{center}
\sampletile{$3$}{$1$}{$1$}{$1$}
\sampletile{$3$}{$2$}{$1$}{$2$}
\sampletile{$3$}{$1$}{$3$}{$3$}
\sampletile{$2$}{$4$}{$2$}{$1$}
\sampletile{$2$}{$2$}{$2$}{$0$}

\hspace{0.5cm}

\sampletile{$0$}{$0$}{$0$}{$1$}
\sampletile{$0$}{$1$}{$3$}{$2$}
\sampletile{$1$}{$2$}{$0$}{$2$}
\sampletile{$1$}{$2$}{$1$}{$4$}
\sampletile{$1$}{$3$}{$3$}{$2$}
\sampletile{$3$}{$1$}{$0$}{$1$}

\end{center}
	\caption{A set of 11 aperiodic Wang prototiles due to Jeandel and Rao \cite{Rao2015}.}
\label{fig:11-ATile}
\end{figure}

\begin{theorem}[\protect{\cite[Thm. 1]{Rao2015}}]
There is no aperiodic Wang prototile set with 10 tiles or fewer.
\end{theorem}

The proof of both of these theorems are computer assisted, and they used a series of innovative techniques to check the tilings they generated - from the simple cases of repeating patterns, through to the complicated cases that were in fact subsets of the Kari and Culik constructions above. These more advanced cases - of which there were 4 - were not computer-checkable, so the proofs and checks were carried out by hand. It transpired that each of these aperiodic tilings were coding transducers in some way, and as such were given by similar reasoning to the aperiodicity results due to Kari and Culik. We have included the 11-prototile set in figure \ref{fig:11-ATile}.

It has been postulated, and subsequently answered to a lesser degree than expected in \cite{Socolar2011}, the question ``Does there exist a single-prototile that tiles the plane aperiodically?'' The Taylor-Socolar tile detailed in \cite{Socolar2011} achieves this, but by the use of a tile that is defined with gaps between its various pieces - though tilings of the plane utilising this tile cover every point.

In general, the literature has not, however, given any consideration to infinite sets of prototiles and their periodicity or aperiodicity. However, as seen in \cite{Harel1986}, the aperiodic properties of some finite prototile sets - specifically that if a specified tile appears only finitely often in a planar tiling, then this must be an aperiodic tiling - were found to code $\Pi^1_1$ statements, indicating that perhaps this would be some interesting candidate for further analysis and study. 

\subsubsection{Quasi-periodicity of tilings}\label{subsec:quasiperiodic}

When observing the properties of Penrose tilings, it is immediate that certain patterns recur regularly, even though the overall tiling is aperiodic. Such tilings are in the class of \emi{quasi-periodic} tilings, which we define as follows, from \cite{Delvenne2004}:

\begin{definition}\label{def:quasiperiodic}
For a given prototile set $S$, $S$ is \emph{quasi-periodic} iff each $S$-tiling of the plane is of the form such that for every pattern $u$ of the tiling there is an integer $k$ such that $u$ appears in every $(k \times k)$ patch of tiles. 
\end{definition}

Where here a \emi{pattern} is any valid, finite patch of tiles that occurs in our tiling. Intuitively, something is quasiperiodic if any finite patch can be found occurring infinitely often and within a bound in any tiling. As a reference, consider a star-like pattern in a Penrose tiling. 

There is a lot of interesting work found in papers such as Delvenne and Blondel \cite{Delvenne2004}, and survey papers connecting quasicrystals to quasi-periodic tilings like Schechtman \cite{Schechtman2013}. The most interesting parts of these are the way in which Penrose tilings mimic and indeed accurately code actual physical surfaces found in Shi \etal in \cite{Shi2017} - where we can note that their 7 diagrams of the ``angles and islands around each vertex'' line up with de Bruijn's derived unique vertex configurations for Penrose tilings found in \cite{DEBRUIJN198139} and \cite{DEBRUIJN198153}. We note that, although these 7 configurations are not the 8 identified by de Bruijn, we suspect that given two of the configurations in the mathematics are identical with edge-conditions removed, they look to be identical under the microscope in \cite{Shi2017}.

Such connections are found in other quasicrystals which we alluded to previously - \eg Subramanian \etal in \cite{Subramanian2016}, Shi \etal \cite{Shi2017} and Au-Yang \etal \cite{AuYang2013} are all readily accessible physics papers that make extensive use of the developed mathematics behind Penrose tilings as quasicrystals. This is, however, a digression from the main content of this thesis.

Indeed, the work of Socolar \etal in \cite{Socolar2011} is a very interesting way of determining the dynamics of this aperiodic tiling system. We will consider more the dynamics of tilings in Chapter 6 - but it is worth noting that it is an open problem as to whether the tile-by-tile tilings of the plane due to the method in \cite{Socolar2011} does indeed lead to planar tilings.

\section{Periodicity and Aperiodicity of $ILL$}

\begin{theorem}[C. 2019]\label{thm:ILL-ATile}
$$ ILL \leq_m ATile $$
\end{theorem}

\begin{proof}
To see this fact, we note that the construction of our function $h$ in the proof of theorem \ref{thm:TILE-ILL} gives an infinite set of prototiles $\mathcal{S}$ that tiles the plane in such a way that the root tile will only occur once, and every point $(x,y)$ in the plane has some unique tile in $\mathcal{S}_e$ that covers it. As such, any ill-founded tree $e \in ILL$ coded into a $\mathcal{S}_e$ by $h$ in our given construction is necessarily aperiodic. Thus it follows that for any $e \in ILL$, our given $h(e) \in ATile$.

Conversely, any $h(e) \in ATile$ must tile the plane, and as such our $e$ must be in $ILL$ otherwise it would be a well-founded tree, and so not tile the plane as outlined in our previous proof.
\end{proof}

It was, however, found that the following additional result could also be obtained:

\begin{theorem}[C. 2019]\label{thm:ILL-PTile}
$$ ILL \leq_m PTile $$
\end{theorem}

\begin{proof}
We can obtain the result by an adapting the procedure in the proof from \ref{thm:TILE-ILL} in the following way. We require a computable $f$ such that $$ \forall e (e \in ILL \iff f(e) \in PTile ) $$

We start by defining our colours as the following:

\begin{itemize}
\item Let $\lambda$ denote the empty string, and let $\lambda^U, \lambda^D$ be unique colours.
\item Fix $M$ unique, and $U_i, D_i$ unique for all $i \in \omega$.
\item Let $\alpha \in \omega^\omega$, and for all $i \in \omega$, let $\sigma_i \in \omega^{< \omega}$ denote successive initial segments of $\sigma$ of length $i$ such that $ \sigma_0 \prec \sigma_1 \prec \ldots \prec \sigma_i \ldots \prec \alpha$. 
\item We fix for each $\sigma_i$ an `up' $\sigma^U_i$ and `down' $\sigma^D_i$ colour that will be used in the prototile set construction.
\item Let $\sigma_0 = \lambda$ as before.
\end{itemize}

With these fixed, let $e \in ILL$ be given. We will construct our prototile set from the following schema tiles:

We start with a modified \textbf{root tile}:
\begin{center}
\sampletile{$M$}{$\lambda^U$}{$M$}{$\lambda^D$}
\end{center}

Next, we require \textbf{column tiles} of the following form:
\begin{center}
\sampletile{$U_i$}{$\sigma_i^{U \frown} n$}{$U_i$}{$\sigma_i^U$} \sampletile{$D_i$}{$\sigma^D$}{$D_i$}{$\sigma_i^{D\frown} n$}
\end{center}

We then construct our prototile set $\mathcal{S}_e$ similarly to the previous proof, by colouring the above schema tiles as follows:

\begin{itemize}
\item Colour the root tile with the tuple $\langle M, \lambda^U, M, \lambda^D \rangle$ and put this into $\mathcal{S}_e$.
\begin{itemize}
\item \textbf{NB} - we still maintain the difference between the `up' and `down' variants of our empty string symbol in order to prevent trivial root-tile only tilings of the plane, though they would be undoubtedly periodic.
\end{itemize}
\item We fix some path $p \in \varphi_e$ such that $\sigma_n \prec p$ for $\sigma_n \in \omega^{< \omega}$, and add a column tile where it holds that $\varphi_e(p \upharpoonright n) = 1$.
\begin{itemize}
\item For $\sigma_0$ we use the appropriate placement of $\lambda^U$ and $\lambda^D$ as before.
\item We also select distinct colours for $\sigma^U_i$ and $\sigma^D_i$ in order that we fail to tile the plane if $e \notin ILL$.
\end{itemize}
\end{itemize}

We can now verify that for each $e \in ILL$ we get $f(e) \in PTile$. The core idea in this construction is to have infinitely many copies of our central column tilings from our previous proof, laid out in such as way that for left or right shift of our tiling, we get the same tiling back, thus $f(e)$ would be periodic. 

As before, we can define our tiling function $\Phi^p : \mathbb{Z}^2 \rightarrow S_e$ as follows:
\begin{itemize}
\item For $\Phi^p(x,0)$ return the root tile $\langle M, \lambda^U, M , \lambda^D \rangle$.
\item For $\Phi^p(x,y)$, with $\sigma = p \upharpoonright y$, 
\begin{itemize}
\item If $y > 0$ return the tile $ \langle U_y, \sigma^{U \frown} n , U_y, \sigma^U \rangle$
\item If $y < 0$ return the tile $ \langle D_y, \sigma^D , D_y ,\sigma^{D \frown} n \rangle$
\end{itemize}
\end{itemize}

To see that our tilings are periodic, note that all of our root tiles will form an infinite middle-row of tiles that can be left or right shifted. We then build up our tilings, noting that each successive column will have prototiles selected that code specifically some copy of our path $p$ upwards or downwards. Thus, every $\mathcal{S}_e$-tiling will have infinitely many leftwards or rightwards shifts. 

Thus, if $\mathbf{v}$ is a `shift right one' vector, then we have that an $\mathcal{S}_e$-tiling $T_e$ has the property $$T_e = \mathbf{v} T_e$$ meaning that $f(e) \in PTile$.

Suppose we have some $f(e) \in PTile$, then it follows that from any root tile we can extract some infinite path moving upwards that gives us that $e \in ILL$. We can also locate a root tile from any tile we select in our $\mathcal{S}_e$-tilings by moving appropriately down our $UM_i$'s or up our $DM_i$'s until a root tile is reached.

From this position we can then follow our tiling upwards in order to extract an infinite path that was given by $e$. As such, if our tiling is total and total, $e \in ILL$.
\end{proof}

\begin{figure}[t]
  \centering
\begin{tikzpicture}
\begin{scope}[scale=2]
\foreach \x/\y/\l/\u/\r/\b in {0/0/$U_2$/$01^U$/$U_2$/$0^U$,1/0/$U_2$/$01^U$/$U_2$/$0^U$,2/0/$U_2$/$01^U$/$U_2$/$0^U$,3/0/$U_2$/$01^U$/$U_2$/$0^U$,4/0/$U_2$/$01^U$/$U_2$/$0^U$,
                               0/1/$U_1$/$0^U$/$U_1$/$\lambda^U$,1/1/$U_1$/$0^U$/$U_1$/$\lambda^U$,2/1/$U_1$/$0^U$/$U_1$/$\lambda^U$,3/1/$U_1$/$0^U$/$U_1$/$\lambda^U$,4/1/$U_1$/$0^U$/$U_1$/$\lambda^U$,
                               0/2/$M$/$\lambda^U$/$M$/$\lambda^D$,1/2/$M$/$\lambda^U$/$M$/$\lambda^D$,2/2/$M$/$\lambda^U$/$M$/$\lambda^D$,3/2/$M$/$\lambda^U$/$M$/$\lambda^D$,4/2/$M$/$\lambda^U$/$M$/$\lambda^D$,
                               0/3/$D_1$/$\lambda^D$/$D_1$/$0^D$,1/3/$D_1$/$\lambda^D$/$D_1$/$0^D$,2/3/$D_1$/$\lambda^D$/$D_1$/$0^D$,3/3/$D_1$/$\lambda^D$/$D_1$/$0^D$,4/3/$D_1$/$\lambda^D$/$D_1$/$0^D$,
                               0/4/$D_2$/$0^D$/$D_2$/$01^D$,1/4/$D_2$/$0^D$/$D_2$/$01^D$,2/4/$D_2$/$0^D$/$D_2$/$01^D$,3/4/$D_2$/$0^D$/$D_2$/$01^D$,4/4/$D_2$/$0^D$/$D_2$/$01^D$}
{
\draw[fill=white] (1+\x,1-\y) rectangle (0+\x,0-\y);
\filldraw[fill=white] (0+\x,0-\y) -- (0.5+\x,0.5-\y) -- (0+\x,1-\y) -- cycle;
\filldraw[fill=white] (0+\x,0-\y) -- (0.5+\x,0.5-\y) -- (1+\x,0-\y) -- cycle;
\filldraw[fill=white] (0+\x,1-\y) -- (0.5+\x,0.5-\y) -- (1+\x,1-\y) -- cycle;
\node at (0.5+\x,0.5-\y) [label=above:{\u},label=left:{\l},label=right:{\r},label=below:{\b}] {};
}
\end{scope}
\end{tikzpicture}
  \caption{$PTile$ for $e \in ILL$ Construction}
  \label{fig:PTileILL}
\end{figure}

In figure \ref{fig:PTileILL} we give an example of the tiling construction for theorem \ref{thm:ILL-PTile} for the initial segment $\sigma = 01$. This illustrates the way in which we create vertical dual copies of the given path from our ill-founded tree in such a way that any left shift vector $\mathbf{l}$, or right shift vector $\mathbf{r}$ and a given $T_e$, we have that $$ \mathbf{l}T = T = \mathbf{r}T$$

\begin{figure}[t]
\begin{center}
\begin{tikzpicture}
	\draw (0,0) rectangle (1,4) node[pos=.5,rotate=-90] {$\cdots$};
	\draw (0,4) rectangle (1,5) node[pos=0.5] {$\cdots$};
    \draw (0,5) rectangle (1,9) node[pos=.5,rotate=-90] {$\cdots$};
	\draw (1,0) rectangle (2,4) node[pos=.5,rotate=-90] {lower copy of $\sigma$};
	\draw (1,4) rectangle (2,5) node[pos=0.5] {$\lambda$};
    \draw (1,5) rectangle (2,9) node[pos=.5,rotate=-90] {upper copy of $\sigma$};	
    \draw (2,0) rectangle (3,4) node[pos=.5,rotate=-90] {lower copy of $\sigma$};
	\draw (2,4) rectangle (3,5) node[pos=0.5] {$\lambda$};
    \draw (2,5) rectangle (3,9) node[pos=.5,rotate=-90] {upper copy of $\sigma$};
	\draw (3,0) rectangle (4,4) node[pos=.5,rotate=-90] {lower copy of $\sigma$};
	\draw (3,4) rectangle (4,5) node[pos=0.5] {$\lambda$};
    \draw (3,5) rectangle (4,9) node[pos=.5,rotate=-90] {upper copy of $\sigma$};
    \draw (4,0) rectangle (5,4) node[pos=.5,rotate=-90] {lower copy of $\sigma$};
	\draw (4,4) rectangle (5,5) node[pos=0.5] {$\lambda$};
    \draw (4,5) rectangle (5,9) node[pos=.5,rotate=-90] {upper copy of $\sigma$};
	\draw (5,0) rectangle (6,4) node[pos=.5,rotate=-90] {$\cdots$};
	\draw (5,4) rectangle (6,5) node[pos=0.5] {$\cdots$};
    \draw (5,5) rectangle (6,9) node[pos=.5,rotate=-90] {$\cdots$};
\end{tikzpicture}
\caption{Overall shape of our tiling construction in the proof of \ref{thm:TILE-ILL}.}
\label{fig:ShapeTilingILL-PTile}
\end{center}
\end{figure}

Figure \ref{fig:ShapeTilingILL-PTile} shows the overall shape of this tiling construction used in the proof of theorem \ref{thm:ILL-PTile}. This diagram is complimentary to the previous figure \ref{fig:PTileILL}.

Note that we were required to preserve the up vs. down directions of our paths, which we were not required to do before. The reason being is that we wanted to preserve that the existence of a tiling derived with $f(e)$ implies that our original $e \in ILL$. We could very well have constructed periodic tilings of $e$'s that are either in $WELL$ or $ILL$. This realisation drove the results in the next section \ref{sec:PAWell}.

\subsection{Periodicity and Aperiodicity of $WELL$}\label{sec:PAWell}

Before we carry on with the proofs in this section we will need the following tool - the ability to take disjoint unions of prototile sets. Our requirement for this construction can be outlined in the following definition and subsequent proposition:

\begin{definition}
We say that two prototile sets $S_1$ and $S_2$ have \emph{common edge meets} iff for some tile $t_i \in S_1$, with $t_i = \langle l_i, u_i, r_i, b_i \rangle $, there exists a tile $s_i \in S_2$ such that one of the following hold:
\begin{itemize}
\item $s_i = \langle r_i, \cdot, \cdot, \cdot \rangle$
\item $s_i = \langle \cdot, b_i, \cdot, \cdot \rangle$
\item $s_i = \langle \cdot, \cdot, l_i, \cdot \rangle$
\item $s_i = \langle \cdot, \cdot, \cdot, u_i \rangle$
\end{itemize}
where $\cdot$ denotes a `wildcard placeholder' for any other possible colour.
\end{definition}

We say that two prototiles $S_1$ and $S_2$ have no common edge meets if the above definition does not hold - intuitively, you cannot place any tile from $S_1$ next to any tile from $S_2$, and vice versa. The following proposition demonstrates an important consequence of two prototile sets being edge-meet disjoint.

\begin{proposition}[C. 2019]\label{prop:DisjointEdgeMeets}
If two periodic (aperiodic) prototile sets $S_1, S_2$ share no common edge meets, then their union $S_1 \cup S_2$ is also periodic (aperiodic).
\end{proposition}

\begin{proof}
Let periodic prototile sets $S_1, S_2$ be given. If $S_1$ and $S_2$ share no common edge meets, then for any selection of a tile $t \in S_1 \cup S_2$, the resultant tiling must be formed from only tiles from $S_1$ if $t \in S_1$ or $S_2$ otherwise, as the edge-meet criteria from each prototile set is incompatible. Thus any tiling from such a $S_1 \cup S_2$ is periodic.

We note that the same argument holds for $S_1$ and $S_2$ being aperiodic. 
\end{proof}

To illustrate an example of where this fails - which is essentially the canonical case that we wish to avoid - we provide the following:

\begin{example}\label{example:UnionTileSets}
Let it be given that a periodic tiling consisting of squares can be made aperiodic by the bisection of a single randomly chosen square into two rectangles. Thus we give the following example to illustrate how this can be done in Wang prototile sets, and thereby show the importance of the lack of edge-meets between prototile sets. 

Let $S_1$ be given by the prototile 
\begin{center}
\sampletilefill{red}{blue}{red}{blue}
\end{center}
and let $S_2$ be given by the prototiles
\begin{center}
\sampletilefill{red}{blue}{green}{green} \sampletilefill{green}{blue}{red}{green} \\
\vspace{2mm}
\sampletilefill{red}{green}{green}{blue} \sampletilefill{green}{green}{red}{blue}
\end{center}

Clearly both $S_1$ and $S_2$ are periodic by themselves. However, $S_1 \cup S_2$ will have tilings that, say, feature only finitely many of the patch tilings given by the prototiles in $S_2$, and would therefore be aperiodic. The same could be done by a single column of tiles from the prototile in $S_1$ being inserted into an $S_2$-tiling, which would also make it aperiodic.
\end{example}

As such, given the example above, we present a construction that provides a way of combining prototile sets, yet preserving the periodicity and aperiodicity conditions we wish to. 

\begin{definition}[Disjoint Union of Tile Sets]\index{disjoint union of prototile sets}
Let the \emph{disjoint union of prototile sets} $A$ and $B$, denoted $A \sqcup B$, be given as follows:
\begin{itemize}
\item For each prototile $t \in A$, let $t = \langle a,b,c,d \rangle$ then this gets mapped to $$ \langle a,b,c,d \rangle \mapsto \langle (1,a), (1,b), (1,c), (1,d) \rangle $$
\item For each prototile $s \in B$, let $s =\langle e,f,g,h \rangle$ then we map this similarly: $$ \langle e,f,g,h \rangle \mapsto \langle (2,e), (2,f), (2,g), (2,h) \rangle $$
\end{itemize}

Likewise, for any arbitrary number of prototile sets $S_i$ for $i \in \omega$ the disjoint union $ \bigsqcup_{i \in \omega} S_i $ is given by mapping each $t_j \in S_i$, with $t_j = \langle l_j, u_j, r_j, b_j \rangle$ by $$  \langle l_j, u_j, r_j, b_j \rangle \mapsto  \langle (i,l_j), (i,u_j), (i,r_j), (i,b_j) \rangle $$
\end{definition}

The intuition behind this disjoint union is the ability to take two sets of (potentially infinite) prototile sets and `apply a tint' to each prototile in each prototile set, thereby placing us in the position given in proposition \ref{prop:DisjointEdgeMeets}. Thus, we can talk about the tiling properties of the resultant disjoint union, but each subset will be incompatible for tiling with any others. 

Our intention is to be able to talk about the disjoint union of two prototile sets $A$ and $B$ in the following way, after proposition \ref{prop:DisjointEdgeMeets}:
\begin{itemize}
\item If both $A$ and $B$ are periodic (aperiodic) then the disjoint union $A \sqcup B$ will be periodic (aperiodic), and so will likewise belong to $PTile$ ($ATile$).
\item If $A$ is periodic and $B$ is aperiodic, or vice versa, then $A \sqcup B$ will have both periodic and aperiodic tilings and so will belong to neither $PTile$ nor $ATile$.
\end{itemize}

In our previous example \ref{example:UnionTileSets}, were we to take $S_1 \sqcup S_2$, then we would only have periodic tilings, given both $S_1$ and $S_2$ are periodic, total planar tilings, and would fail to share edge-meet conditions in $S_1 \sqcup S_2$.

Prototile sets that are not in either $PTile$ nor $ATile$ are relatively easy to find. A straightforward example is the set consisting of the following sixteen prototiles:
\begin{center}
\sampletile{$0$}{$0$}{$0$}{$0$} 
\sampletile{$0$}{$0$}{$0$}{$1$}
\sampletile{$0$}{$1$}{$0$}{$0$}
\sampletile{$0$}{$1$}{$0$}{$1$} \\
\smallskip
\sampletile{$0$}{$0$}{$1$}{$0$}
\sampletile{$0$}{$0$}{$1$}{$1$}
\sampletile{$0$}{$1$}{$1$}{$0$}
\sampletile{$0$}{$1$}{$1$}{$1$} \\
\smallskip
\sampletile{$1$}{$0$}{$0$}{$0$}
\sampletile{$1$}{$0$}{$0$}{$1$}
\sampletile{$1$}{$1$}{$0$}{$0$}
\sampletile{$1$}{$1$}{$0$}{$1$} \\
\smallskip
\sampletile{$1$}{$0$}{$1$}{$0$}
\sampletile{$1$}{$0$}{$1$}{$1$}
\sampletile{$1$}{$1$}{$1$}{$0$}
\sampletile{$1$}{$1$}{$1$}{$1$}
\end{center}

These prototiles allow us to encode two binary strings - one going vertically, and another horizontally. Thus, if we place tiles such that they encode periodic repeating strings, such as ``$0101010101\ldots$" using these prototiles in our tiling of the plane, then our tiling will clearly be periodic. 

However if we use non-repeating, aperiodic strings - such as using a Martin-L\"{o}f random string vertically and the binary version of Champernowne's constant\footnote{This is constructed by concatenating every binary number: $0110111001011101111000\ldots$} horizontally - then our tiling will be clearly aperiodic. 

Essentially, in this tiling we code two binary strings - $\sigma$ going left to right and $\tau$ going up and down. If either $\sigma$ or $\tau$ (or both) are periodic, then the tiling is periodic. Else, the tiling is aperiodic.

We will use our previous constructions, and fix the following construction names.

\begin{definition}\label{def:AITPIT}
Let the following short hand definitions be given:
\begin{itemize}
\item \textbf{AIT} (Aperiodic Ill-founded Tilings) - the construction found in the proof of theorem \ref{thm:TILE-ILL}.
\item \textbf{PIT} (Periodic Ill-founded Tilings) - the construction found in the proof of theorem \ref{thm:ILL-PTile}.
\end{itemize}
\end{definition}

Recall, our constructions here take any ill-founded tree and generate either periodic or aperiodic prototile sets as required. We shall use these constructions in the following sections in conjunction with our notion of disjoint union of prototile sets (`prototile set tinting') in order to obtain the following results.

\begin{theorem}[C. 2019]\label{thm:WELL-PTile}
$$ WELL \leq_m PTile $$
\end{theorem}

\begin{proof}
As before, we want some recursive function $k$ such that $$ e \in WELL \iff k(e) \in PTile $$

We begin by fixing some recursive ill-founded tree $R$ and feeding this through the \textbf{PIT} construction to obtain a set of prototiles $\mathcal{R}$ that has only periodic tilings of the plane for any infinite path in $R$.

We next take our $e$ and pass this through the \textbf{AIT} construction to get a prototile set $U_e$ that tiles the plane only if $e \notin WELL$. We then let our desired prototile set $S_e$ generated by this recursive method be $$ S_e = \mathcal{R} \sqcup U_e$$

If $e \in WELL$ then the only tilings of the plane will be given by $\mathcal{R}$, and as such, $k(e) \in PTile$. 

If $e \notin WELL$ then both $\mathcal{R}$ and $U_e$ will give tilings of the plane, meaning that $k(e) \notin PTile$, as it would have both periodic \emph{and} aperiodic tilings.
\end{proof}

By a nearly identical argument we shall obtain the following result:

\begin{theorem}[C. 2019]\label{thm:WELL-ATile}
$$ WELL \leq_m ATile $$
\end{theorem}

\begin{proof}
We proceed exactly as above, to construct a recursive $l$ such that $$e \in WELL \iff l(e) \in ATile$$ but with our argument switching the periodic and aperiodic constructions from our previous proof.

We fix a recursive ill-founded tree $R$ and now feed this through the \textbf{AIT} construction, giving us a new $\mathcal{R}$ we shall use. Likewise, we will take our $e$ and pass this through the \textbf{PIT} construction to get $V_e$. Our prototile set $S_e$ is now given by $$ S_e = \mathcal{R} \sqcup V_e $$

If $e \in WELL$ then as above, the only tilings of the plane will come from $\mathcal{R}$, except that this time they shall be aperiodic, and so $l(e) \in ATile$. 

Similarly, if $e \notin WELL$ then both $\mathcal{R}$ and $V_e$ will give tilings of the plane, and given $V_e$ gives periodic tilings, we have that $l(e) \notin ATile$.
\end{proof}

\subsubsection{An Alternative Proof}

We note that there exist alternative and more intuitive ways that we can prove both \ref{thm:WELL-PTile} and \ref{thm:WELL-ATile} that we shall provide here.

\begin{proof}[Alternative Proof for \ref{thm:WELL-PTile}, C. 2019]
We begin by using the construction in \ref{thm:SNT-WELL} - the finite diamond-shaped patches of tiles that will not tile the plane iff the tree whose paths it tiles is well-founded. To this tiling set, we add the following prototile schemes:

\textbf{Corner tiles:}
\begin{center}
\sampletile{$\sigma_n$}{$\sigma_n$}{$\sigma_n$}{$\sigma_n$}
\end{center}
for each $\sigma \in \omega^{<\omega}$, with $|\sigma| = n$ and $\sigma \in \varphi_e$.

\textbf{Edge Connecting tiles:}

\begin{center}
\sampletile{$\sigma^2_n$}{$\sigma^4_n$}{$\sigma^4_n$}{$\sigma^2_n$}
\sampletile{$\sigma^3_n$}{$\sigma^3_n$}{$\sigma^1_n$}{$\sigma^1_n$}
\end{center}
for each $\sigma_n$ as above.

The idea of these tiles are, as we shall see, to fill in the gaps between fragments of our original prototile set construction, and provide total and periodic tilings of the plane.

We construct our library $\mathcal{U}$ as before, and extract $U_e$ as before, adding in the requisite Corner tiles and Edge Connecting tiles, being careful to remove the quadrant filling tiles we had included so far for paths $\sigma^j_n$. We then note that we only require two pairs of quadrant tile types that will meet in the total planar tiling - $\sigma^2_n$ tiles will meet with $\sigma^4_n$ tiles, and $\sigma^3_n$ tiles will meet with $\sigma^1_n$ tiles.

The resulting $U_e$ then takes each of our previous patch tilings and allows us to join them together by the addition of the connective tiles. Thus, we are effectively tiling with our `meta-tiles' formed from the patch tilings we constructed above. 

So, we can let this above procedure be a computable function $p$. If $e \in WELL$ then $p(e)$ will construct a $U_e$, all of whose tilings are periodic total tilings of the plane. Thus $p(e) \in PTile$.

Likewise, if $e \notin WELL$ then only one path will be tiled, and will be infinite and total. However, as it will only use the root tile once in any tiling, it follows that there are no linear shifts of our tiling that can be performed. Thus, $p(e) \notin PTile$.

As such, we have \[ e \in WELL \iff p(e) \in PTile  \] which gives us our $m$-reduction \[ WELL \leq_m PTile \]
\end{proof}

\section{Completeness of $PTile$ and $ATile$}

Given we have assessed the relationship of $WELL$ and $ILL$ to tiling problems regarding periodicity and aperiodicity, it is natural to next seek some completeness for this general class of problems. In this spirit, we present the following theorem:

\begin{theorem}[C. 2019]\label{thm:X-ATile}
Let $X \subset \omega$ be in $( \Pi^1_1 \wedge \Sigma^1_1 )$, that is $$ X = \{ n : \chi(n) \wedge \psi(n) \}$$ such that $\chi \in \Sigma^1_1$ and $\psi \in \Pi^1_1$, then $$X \leq_m ATile$$
\end{theorem}

Intuitively, this proof arises from the fact that our definitions of $PTile$ and $ATile$ are both of the form ``there exists a tiling" followed by some general statement about all of the tilings given by that prototile set. 

In this proof, we will pass each statement through the periodic or aperiodic construction for the ill-founded ($\Pi^1_1$) side of the conjunction as desired. We then take the disjoint union of this with the $\Sigma^1_1$ side of the construction being passed through the opposite (a)periodic construction to obtain the result. The formal proof now follows.

\begin{proof}
To show that $X \leq_m ATile$, we want some computable $h$ such that $$n \in X \iff h(n) \in ATile$$. 

First let us define our two recursive functions $f:X \rightarrow \omega$ and $g : X \rightarrow \omega$ as follows:
\begin{itemize}
\item $f(n)$ be such that $(\varphi_{f(n)}$ is a tree $\wedge f(n) \in ILL) \leftrightarrow \chi(n)$
\item $g(n)$ be such that $(\varphi_{g(n)}$ is a tree $\wedge g(n) \in WELL) \leftrightarrow \psi(n)$
\end{itemize}

Our function $f$ holds only if the $\Sigma^1_1$ side of our formula given by $\chi(n)$ and constructs index that computes the tree $T \subseteq \omega^{<\omega}$ given by this formula, resulting in an index $f(n) \in ILL$.

Likewise the function $g$ holds if the $\Pi^1_1$ side of our formula given by $\psi(n)$ holds, and constructs index that computes the tree $T \subseteq \omega^{<\omega}$ given by this formula, resulting in an index $g(n) \in WELL$.

Now let the $U,V$ be defined as follows:
\begin{itemize}
\item $U$ is the set of prototiles obtained by passing $\varphi_{f(n)}$ through the \textbf{AIT} construction to create an aperiodic prototile set for $\varphi_{f(n)}$ being ill-founded.
\item $V$ is the set of prototiles obtained by passing $\varphi_{g(n)}$ through the \textbf{PIT} construction to create a periodic prototile set for $\varphi_{g(n)}$ not being well-founded.
\end{itemize}

Both of these constructions are given by the previous results, and so are known computable reductions. $h(n)$ be then the function that produces the prototile set that is the disjoint union $S_n = U \sqcup V$. 

These two infinite sets of prototiles have both been passed through constructions designed for total planer tilings intended for ill-founded trees. Thus, the prototile set corresponding to our well-founded prototiles, $V$, will only tile the plane if $\neg \psi(n)$ holds. Given this, we now utilise our disjoint union in obtaining $S_n$ in order to restrict the behaviour of our combined prototile sets to obtain the result we want. 

We thus have the following 4 cases: 
\begin{enumerate}
\item $\chi(n) \wedge \psi(n)$ - In this case, everything is as we would like it to be, as the only planar $S_n$-tilings will be given by $U$, which are aperiodic.
\item $\neg \chi(n) \wedge \psi(n)$ - In this case we will get no total $S_n$-tilings of the plane.
\item $\chi(n) \wedge \neg \psi(n)$ - In this case we will get both periodic and aperiodic $S_n$-tilings of the plane.
\item $\neg \chi(n) \wedge \neg \psi(n)$ - In this case we will only get periodic $S_n$-tilings of the plane.
\end{enumerate}

Given by our construction of $h$ we only get aperiodic tilings of the plane for $n$ precisely when $(\chi(n) \wedge \psi(n))$, it follows that $n \in X \rightarrow h(n) \in ATile$.

For the converse argument, take that $h(n) \in ATile$ is given. For the class of $S_e$-tilings $\mathcal{T}$ given by $h(e)$ we take some $T \in \mathcal{T}$ and ask if $T$ is total. If $T$ is a total tiling, then we can extract (as described in \ref{thm:TILE-ILL}) an infinite path corresponding to the ``$\varphi_{f(n)} \leftrightarrow \chi(n)$" part of the definition of $n \in X$. 

If $T$ is not a total tiling, then we know that we have infinitely many copies of the path given by $\varphi_{g(n)}$ corresponding to the ``$\varphi_{g(n)} \leftrightarrow \psi(n)$" part of the definition of $n \in X$. 

Thus, by examining the class of $S_n$-tilings given by $h(n) \in ATile$ we can get that $n \in X$, for any $X$ of the desired form in the theorem.
\end{proof}

\begin{theorem}[C. 2019]\label{thm:X-PTile}
For $X = \{ n : \chi(n) \wedge \psi(n) \}$, with $\chi(n) \in \Sigma^1_1$ and $\psi(n) \in \Pi^1_1$, then $$ X \leq_m PTile $$
\end{theorem}

\begin{proof}
Our proof proceeds precisely as for \ref{thm:X-ATile} in order to give a recursive $k$ such that $$ n \in X \iff k(n) \in PTile $$ except that we differ in constructing $U$ and $V$ as follows:
\begin{itemize}
\item $U$ is the set of prototiles obtained by passing $\varphi_{f(n)}$ through the \textbf{PIT} construction to create a periodic prototile set for $\varphi_{f(n)}$ being ill-founded.
\item $V$ is the set of prototiles obtained by passing $\varphi_{g(n)}$ through the \textbf{AIT} construction to create an aperiodic prototile set for $\varphi_{g(n)}$ not being well-founded.
\end{itemize}
Wherein we have essentially swapped the roles of \textbf{PIT} and \textbf{AIT} in order to achieve our result. We can then re-analyse the outcomes as follows:
\begin{enumerate}
\item $\chi(n) \wedge \psi(n)$ - In this case, we only get periodic $S_n$-tilings of the plane.
\item $\neg \chi(n) \wedge \psi(n)$ - In this case we will get no total $S_n$-tilings of the plane.
\item $\chi(n) \wedge \neg \psi(n)$ - In this case we will get both periodic and aperiodic $S_n$-tilings of the plane.
\item $\neg \chi(n) \wedge \neg \psi(n)$ - In this case we will only get aperiodic $S_n$-tilings of the plane.
\end{enumerate}

Thus, our $k$ has precisely the same properties as our previous $h$, with the periodicity properties reversed. As such, the forwards and reverse directions of our implication are precisely the same, giving our result.
\end{proof}

Once we define our constructions in these results, the entire proofs are essentially captured in the four cases. The fact that both $ATile$ and $PTile$ have interchangeably periodic and aperiodic $\Sigma^1_1$ and $\Pi^1_1$ parts was unexpected, but actually quite natural.

The background intuition for these results was the observation that the existence of a tiling, and the fact that all tilings either have exclusively or no periodic/aperiodic parts. If we allow ourselves to use quantification of sets in the analytic hierarchy as above, we obtain the following corollary:

\begin{corollary}[C. 2019]\label{cor:ATilePTileComplete}
Aperiodicity and periodicity for infinite prototile sets is $(\Sigma^1_1 \wedge \Pi^1_1)$-complete
\end{corollary}

\begin{proof}
This follows from our previous theorem \ref{thm:X-ATile} and theorem \ref{thm:X-PTile} working in tandem. Any problem given in the form $$\zeta(n) \leftrightarrow (\chi(n) \wedge \psi(n))$$ for $\chi(n) \in \Sigma^1_1$ and $\psi(n) \in \Pi^1_1$ has a representation as a tiling problem on infinite prototile sets by our constructions above, thereby having both periodic and aperiodic total tilings being given.
\end{proof}

In fact, we can choose which of aperiodic or periodic tilings we would like for our infinite prototile sets.

As an aside, the author did attempt to find other problems that share this same or similar syntactical form or structure. The closest that we could find was a definition and corollary in Bagaria \etal \cite[def. on p.6, Cor. 6.8]{Bagaria2015} wherein they show that Vop\v{e}nka's Principle for $\Sigma_{n+2}$ classes is equivalent for $(\Sigma_{n+1} \wedge \Pi_{n+1})$ classes, which naively seems to be a weaker form. However, these only work for $n \geq 1$, so are not an exact match, and indeed were superseded by the work by Bagaria \etal in \cite[Cor 4.13]{Bagaria2012}, where the result was weakened further to $\Pi_{n+1}$.\footnote{We would like to thank Dr. Andrew Brooke-Taylor for these references.} Aside from these references, it does indeed seem to be the case that very little in logic has $(\Sigma^1_1 \wedge \Pi^1_1)$ as the natural syntactic shape.

\section{Aperiodicity and Periodicity for Finite Prototile Sets}

\begin{definition}\label{def:PTileFIN}
Let the set of periodic finite prototile sets be
\begin{align*}
PTile_{FIN} = \{ e : & \, e \text{ tiles the plane from a finite set of prototiles}  \\
		     & \,  \text{ all of whose tilings are periodic} \}
\end{align*}
\end{definition}

\begin{definition}\label{def:ATileFIN}
Let the set of aperiodic finite prototile sets be
\begin{align*}
ATile_{FIN} = \{ e : & \, e \text{ tiles the plane from a finite set of prototiles}  \\
                     & \,  \text{ all of whose tilings are aperiodic} \}
\end{align*}
\end{definition}

\begin{definition}
Let a \emi{megatile} $M$ be a finite patch of tiles such that $M$ can be considered to be a tile at scale.
\end{definition}

Note, we differentiate this from a \emph{macrotile} we used earlier, as we are not interested specifically in simulating the original prototile set in our megatiles. We wish to be able to treat blocks of tiles as individual units.

\begin{proposition}[C. 2019][Rectangularisation of Megatiles]
For any non-rectangular megatile $M$ made up of Wang tiles in a periodic tiling $T$, there is a rectangular megatile $M^*$ that tiles $T$ precisely the same as $M$.
\end{proposition}

\begin{proof}
Let $\mathbf{v}$ be the periodicity vector for $T$ such that $[ \mathbf{v} T = T]$ for every non-zero $\mathbf{v}$-shift. Clearly we can rewrite $\mathbf{v}$ in the normal Cartesian orthogonal left-right, up-down basis - let $\mathbf{xy} = \mathbf{v}$.

We first select a tile $t \in T$, our tiling, and begin with the rectangle formed by one application on $t$. This rectangle will have sides of length $|\mathbf{x}| \text{ and } |\mathbf{y}|$, and will capture the translation of this one tile $t$. For each $t_i \in M$, a megatile in our periodic tiling, we can get a sequence $r_1, r_2, \ldots$ of rectangles tracking the motion of each rectangle.

We take either a column (row) of each $r_i$'s such that they overlap at the boundary. We keep appending $r_i$'s under (to the right of) each other until we get the bottom row (right-most column) matches the top row (left-most column). Once we have this, which is guaranteed by the periodicity of our tilings, we can trim the duplicated column (row) and we obtain a single rectangle that has captured all of the translations of each $t_i \in M$ under $\mathbf{v}$.

\end{proof}

The resultant rectangle in the proof has at least two opposite edges that are some permutation of an integer multiple of the $t_i \in M$. Thus, our theorem is guaranteed by the finiteness of our prototile set.

We will now explore the logical complexity of whether finite prototile sets are periodic or aperiodic. Our first result in this endeavour is somewhat unexpected:

\begin{theorem}[C. 2019]\label{thm:ATileFIN-Pi01}
$$ ATile_{FIN} \in \Pi^0_1 $$
\end{theorem}

\begin{proof}
Let $S$ be a finite prototile set, and define the following set: $$ EPTile_{FIN} = \{ e : \text{ there exist periodic tilings given by } \varphi_e \}$$ Given it is equivalent to the halting state of a TM that finds the period of some $S$-tiling $T$, specifically $$ \psi(S) = \exists s ( s \text{ is the period of an $S$-tiling } T ) $$ it naturally follows that $$ EPTile_{FIN} \in \Sigma^0_1$$

Note that this computable search across all possible tilings can proceed iteratively along a sequence of $S$-tilings, which are enumerable given $S$ is finite, given by $$ T_0, T_1, T_2, \ldots $$ We only require that our search stops once for $S$ to be in $EPTile_{FIN}$.

We now note that $\neg \psi(S)$ is equivalent to saying that our periodicity finding machine will not halt for any $S$-tiling, noting that this does not require set comprehension. Thus, $$\neg \psi(S) \in \Pi^0_1$$ and given this is equivalent to saying every $S$-tiling is aperiodic, the theorem follows by: $$ ATile_{FIN} = \overline{EPTile_{FIN}} $$

\end{proof}

\begin{theorem}[C. 2019]\label{thm:PTileFINPi11}
$$ PTile_{FIN} \in \Pi^1_1 $$
\end{theorem}

\begin{proof}
 For a any prototile $S$ and any $S$-tiling $T_S$ we have $$ S \in PTile_{FIN} \iff (\forall T_S) (\exists \mathbf{v}) [T_S = \mathbf{v} T_S] $$ We also notice that for any finite prototile set $S$, the maximal shift is given by every tile of $S$ in a line, thus a periodicity vector $\mathbf{v}$ has a maximal length determined by $| S |$. Given that $\mathbf{v}$ is bounded by the size of $S$, we get that $$ PTile_{FIN} \in \Pi^1_1 $$
\end{proof}

However, given our previous result in theorem \ref{thm:ATileFIN-Pi01}, we may consider that there is some arithmetical representation of $PTile_{FIN}$. But after some searching, we pose the following conjecture:

\begin{conjecture}[C. 2019]
$PTile_{FIN}$ has no arithmetical representation.
\end{conjecture}

The intuition for this follows from the fact that we are required to quantify over every possible $S$-tiling for some prototile set $S$, and thereby guarantee that there is no such $S$-tiling where there is no periodicity vector. As such, this would appear to consistently give $PTile_{FIN} \in \Pi^1_1$ as given above. A concrete proof that there is no arithmetical representation of $PTile_{FIN}$ has not been found, so the possibility remains open.

\chapter{Weihrauch Reducibility and Tiling Problems}
\label{chap5}
\setcounter{equation}{0}
\renewcommand{\theequation}{\thechapter.\arabic{equation}}


\epigraph{An algorithm must be seen to be believed, and the best way to learn what an algorithm is all about is to try it.}{\textit{Donald Knuth, \\ The Art of Computer Programming, Vol. 1, 1999}}

In this chapter we will show how our constructions in the previous section can be utilised as tiling principles on represented spaces of Wang prototile sets and tilings. We present several Weihrauch reductions between these tiling problems for Wang tiles and closed choice problems.

\section{Weihrauch Reducibility}

For this section, we use \cite{Brattka2011} and \cite{Marcone2008} as our primary source material. We give a brief background overview of the theory surrounding Weihrauch reductions and their recent uses, primarily from the viewpoint of computable analysis.

\subsection{Core Concepts in Weihrauch Reducibility}

Computable analysis lends notions of computability and incomputability to computable separable metric spaces by means of notions of effective approximation. The aim is to study multi-valued functions between these spaces and to deal with their non-unique solutions. Indeed, in papers such as \cite{Weihrauch2001}, techniques from computability and reverse mathematics were combined in order to tackle a problem in computable analysis.

As Weihrauch points out in \cite{Weihrauch1995}, a core technique in computable analysis is to take notions of topological continuity and replace them with notions of computability - indeed, the explicit definition of `topologically reducible' is precisely the notion of (computably) reducible in that paper, with `computable' substituted for `continuous'. 

As such we give the following definition of reducibility for multi-valued functions (from \cite{Marcone2008}). Let $f : \subseteq X \rightrightarrows Y$ denote that $f$ is a multi-valued function with $dom(f) \subseteq X \wedge ran(f) \subseteq Y$. The idea is to take $\Pi_2$ theorems of the form $$ (\forall x \in X) \, (\exists y \in Y) \, [ (x,y) \in A ] $$ as operations $f : \subseteq X \rightrightarrows Y$ such that $$x \mapsto \{ y \in Y : (x,y) \in A \}$$ Note that the `$: \subseteq$' here indicates the (potential) partiality of our functions.

\textbf{Core Idea}: As given by \cite{Brattka2011}, the core idea for Weihrauch reducibility in relation to the choice and boundedness conditions we will study here is that, rather than defining our problems directly, we ask instead what can be understood by means of negative information. That is - if we obtain a set $X$ by negative information, say by enumeration of the complement of $X$, then how difficult is it to actually find a member of $X$? Can we define $\chi_X$ this way?

We shall put these ideas more formally:

\begin{definition}
A \emi{represented space} \textbf{X} is a pair $(X, d_X)$ where $X$ is a set and $d_X :\subseteq \omega^\omega \rightarrow \textbf{X}$ is a partial surjective function.
\end{definition}

An intuitive definition is given by Weihrauch in \cite{Weihrauch1995}:

\begin{definition}[Notations and Representations]
Using the notation for surjective partial functions above, and with $\Sigma$ denoting a finite alphabet, with $\Sigma^{<\omega}$ and $\Sigma^\omega$ denoting finite and infinite strings from $\Sigma$ respectively.
\begin{enumerate}
\item A \emi{naming system} of a set, $M$, is a surjective function $\nu : \subseteq \Sigma^{<\omega} \rightarrow M$, essentially naming every element of $M$ with finite strings.
\item A \emi{representation} is a surjective function $\delta : \subseteq \Sigma^\omega \rightarrow M$, essentially naming by infinite sequences.
\end{enumerate}
\end{definition}

Weihrauch then gives the following definition of \emph{reducibility}:
\begin{definition}
For $Y, Y^\prime \in \{ \Sigma^{<\omega}, \Sigma^\omega \}$, and for functions $\gamma : \subseteq Y \rightarrow M$ and $\gamma^\prime : \subseteq Y^\prime \rightarrow M$, we say that $ \gamma \leq \gamma^\prime $ if and only if $$ \forall y \in dom(\gamma) \, [ \gamma(y) = \gamma^\prime(f(y)) ] $$ for some computable function $f : \subseteq Y \rightarrow Y^\prime$.
\end{definition}

Likewise, $(\gamma \equiv \gamma^\prime)$ if and only if $(\gamma \leq \gamma^\prime \wedge \gamma^\prime \leq \gamma)$. However, Brattka \etal in \cite{Brattka2011} give some more general, and arguably applicable, definitions. These notions of Weihrauch reducibility will require the following notion of a realizer:

\begin{definition}
For represented spaces \textbf{X} and \textbf{Y},	
\begin{itemize}
\item For some function $f : \subseteq \textbf{X} \rightrightarrows \textbf{Y}$, a function $F:\subseteq \omega^\omega \rightarrow \omega^\omega$ is a \emi{realizer} of $f$, written $F \vdash f$, if and only if $$\forall p \in d_X^{-1}(dom(f)) \, [ d_Y(F(p)) \in f(d_X(p))]$$
\item $f$ is computable if and only if it has a computable realizer.
\item $f$ is continuous if and only if it has a continuous realizer.
\end{itemize}
\end{definition}

This is more easily summarised in the following commutative diagram:

\begin{center}
\begin{tikzpicture}[every node/.style={midway}]
  \matrix[column sep={4em,between origins}, row sep={2em}] at (0,0) {
    \node(R) {$\omega^\omega$}  ; & \node(S) {$\omega^\omega$}; \\
	\node(R/I) {$\textbf{X}$}; & \node (T) {$\textbf{Y}$};\\
  };
  \draw[<-] (R/I) -- (R) node[anchor=east]  {$d_X$};
  \draw[->] (R) -- (S) node[anchor=south] {$F$};
  \draw[->] (S) -- (T) node[anchor=west] {$d_Y$};
  \draw[->] (R/I) -- (T) node[anchor=north] {$f$};
\end{tikzpicture}
\end{center}

\begin{definition}[Weihrauch Reducibility]\label{def:Weihrauch}\index{Weihrauch reducibility}
Let $f:\subseteq \textbf{X} \rightrightarrows \textbf{Y}$ and $g :\subseteq \textbf{U} \rightrightarrows \textbf{V}$. We say that $f$ is \emph{Weihrauch reducible} to $g$, written $$ f \leq_W g$$ if there exist computable $H,K : \subseteq \omega^\omega \rightarrow \omega^\omega$, such that $$ F = K \langle id_{\omega^\omega}, GH \rangle$$ is a realizer of $f$ for every realizer $G$ of $g$.

We say that $f$ is \emi{strongly Weihrauch reducible} to $g$, written $f \leq_{sW} g$, if $$F = K(GH)$$ is a realizer for $f$.
\end{definition}

Here $\langle \cdot \rangle$ is the pairing function, as before, and $id_{\omega^\omega}$ is the identity function on Baire space. We can also say that the single-valued function $F$ is \emph{Weihrauch reducible} to $G$, also written $F \leq_W G$ if there exist single-valued computable functions $H$ and $K$ such that $$ F = K \langle id, GH \rangle $$

In \cite{Brattka2011}, these functions $H$ and $K$ are described as `functions of adaption' - $H$ being an `input adaption' and $K$ being an `output adaption'. The key idea here is to note that $H$ is the input adjustment into problems that $G$ understands, and likewise, $K$ is the transformation of the output of $G$ into an equivalent output of $F$. Thus, if $K$ does not need to know what the original input to $H$ was, represented by $id$ in Weihrauch reducibility, then the reducibility is thus defined to be stronger with respect to not needing to be `reminded' about the input that was originally fed into $H$.

Given these definitions, the following commutative diagram summarises the Weihrauch reducibility of some $f \leq_{W} g$:

\begin{center}
\begin{tikzpicture}[scale=2.5]
\node (1) at (0,1) {$\omega^\omega$};
\node (2) at (1.5,1) {$\textbf{U}$};
\node (3) at (0,0) {$\omega^\omega \rangle$};
\node (4) at (1.5,0) {$\textbf{V}$};
\node (A) at (-1.5,1) {$\omega^\omega$};
\node (B) at (-3,1) {$\textbf{X}$};
\node (C) at (-1.5,0) {$\omega^\omega$};
\node (D) at (-3,0) {$\textbf{Y}$};

\node (id) at (-0.3,0) {$\langle id ,$};

\path[->,font=\scriptsize]
(A) edge node[above]{$d_{\textbf{X}}$} (B)
(A) edge node[left]{$F$} (C)
(B) edge node[left]{$f$} (D)
(C) edge node[below]{$d_{\textbf{Y}}$} (D);

\path[->,font=\scriptsize]
(1) edge node[above]{$d_{\textbf{U}}$} (2)
(1) edge node[right]{$G$} (3)
(2) edge node[right]{$g$} (4)
(3) edge node[below]{$d_\textbf{V}$} (4);

\path[->,font=\scriptsize]
(A) edge node[above]{$H$} (1)
(-0.2,-0.1) edge [bend left] node[below]{$K$} (C);

\draw [->,font=\scriptsize] (A) edge [bend left] node[below]{$id$} (-0.3,0.15);

\end{tikzpicture}
\end{center}

Note that the input it the arrows for $H$ and $id$ must be identical in order for the reducibility to work. Recall that for Weihrauch reducibility to be strong, we can do without this $id$ arrow and requirement, giving us the following commutative diagram which illustrates strong Weihrauch reducibility for some $f \leq_{sW} g$:

\begin{center}
\begin{tikzpicture}[scale=2.5]
\node (1) at (0,1) {$\omega^\omega$};
\node (2) at (1.5,1) {$\textbf{U}$};
\node (3) at (0,0) {$\omega^\omega$};
\node (4) at (1.5,0) {$\textbf{V}$};
\node (A) at (-1.5,1) {$\omega^\omega$};
\node (B) at (-3,1) {$\textbf{X}$};
\node (C) at (-1.5,0) {$\omega^\omega$};
\node (D) at (-3,0) {$\textbf{Y}$};


\path[->,font=\scriptsize]
(A) edge node[above]{$d_{\textbf{X}}$} (B)
(A) edge node[left]{$F$} (C)
(B) edge node[left]{$f$} (D)
(C) edge node[below]{$d_{\textbf{Y}}$} (D);

\path[->,font=\scriptsize]
(1) edge node[above]{$d_{\textbf{U}}$} (2)
(1) edge node[right]{$G$} (3)
(2) edge node[right]{$g$} (4)
(3) edge node[below]{$d_\textbf{V}$} (4);

\path[->,font=\scriptsize]
(A) edge node[above]{$H$} (1)
(3) edge node[below]{$K$} (C);


\end{tikzpicture}
\end{center}

We state the following notion of \emi{realizer reducibility} from \cite{Brattka2011}:

\begin{definition}[realizer reducibility]
For $F : \subseteq \omega^\omega \rightarrow \omega^\omega$, a realizer for $f : \subseteq \textbf{X} \rightrightarrows \textbf{Y}$ ($F \vdash f$ in our notation). Let $f,g$ be multi-valued functions on represented spaces. Then $f$ is \emph{Weihrauch reducible} to $g$, $f \leq_W g$ as before, if and only if $$ \{ F : F \vdash f \} \leq_W \{ G : G \vdash g \} $$
\end{definition}

This single-valued function $F$ can be parallelized, written $\widehat{F}$, by letting $$\widehat{F} (x_0, x_1, x_2, \ldots) := F(x_0) \times F(x_1) \times F(x_2) \times \ldots$$ for some $F : \omega^\omega \rightarrow \omega^\omega$. It is shown in \cite{Brattka2011} that such parallelization is a closure operator for Weihrauch reducibility, as well as the fact that a resulting parallelized partial order forms a lattice into which the Turing and Medvedev degrees can be embedded.

Indeed, we can obtain the following proposition from \cite{Brattka2011}:
\begin{proposition}[\protect{\cite[Prop. 2.5]{Brattka2011}}]
Let $f$ and $g$ be multi-valued functions on represented spaces. Then
\begin{itemize}
\item $f \leq_W \widehat{f}$.
\item If $f \leq_W g$ then $\widehat{f} \leq_W \widehat{g}$.
\item $\widehat{f} \equiv_W \widehat{\widehat{f}}$.
\end{itemize}
\end{proposition}

Much is also made of the study of various kinds of choice in this setting, which is the subject of the next section.

\section{Weihrauch Reducibility and Choice Principles}

We now look to the Weihrauch reducibility of specific choice principles, objects that have much relevance in computable analysis. Let $C$ denote the choice principle given by 
\begin{center}
``For any set $A \subseteq \mathbb{N}$ has a characteristic function $\chi_A : \mathbb{N} \rightarrow \{0,1\}$."
\end{center}

A \emi{choice principle} or \emi{choice function} is given by this definition, and a significant amount of study is given as to the Weihrauch degrees of these functions, \eg in \cite{BrattkaPauly2010} and \cite{Brattka2011}. We shall state some of these results presently.

\begin{definition}[Compact Choice]\index{compact choice}
Let $X$ be a computable metric space, and $\mathcal{K}(X)$ be the set of compact subsets of $X$. The multivalued operation $$ CC_{\mathcal{K}(X)} : \subseteq \mathcal{K}(X) \rightrightarrows X, A \mapsto A $$ with $$dom(CC_{\mathcal{K}(X)}) := \{ A \subseteq X : A \neq \emptyset \text{ compact} \} $$ is called the \emph{compact choice} of $X$.
\end{definition}

Note, in this definition we have the inclusion of the notation ``$A \mapsto A$'', which might give the incorrect impression that we are simply mapping members in $A$ to members in $A$, but our aim is in fact rather to convey that we are mapping a given closed set $A$ to the set of its members in a multi-valued way. Our $\mathcal{K}(X)$ denotes the set of compact subsets of $X$, which are represented by enumerations of finite rational open covers which are not necessarily minimal.

\begin{definition}[Omniscience Principles]
We introduce the following principles: 
\begin{itemize}
\item Limited Principle of Omniscience (LPO) - For any sequence $\sigma \in \omega^\omega$ there exists $n \in \omega$ such that $\sigma(n) = 0 \text{ or } \sigma(n) \neq 0$ for all $n \in \omega$.
\item Lesser Limited Principle of Omniscience (LLPO) - For any sequence $\sigma \in \omega^\omega$ such that $\sigma(k) \neq 0$ for at most one $k \in \omega$, it follows that 
\begin{itemize}
\item $\sigma(2n) = 0$ for all $n \in \omega$, or 
\item $\sigma(2n+1) = 0$ for all $n \in \omega$.
\end{itemize}
\end{itemize}
\end{definition}

These notions may seem unusual, but their motivation is firmly rooted in constructive mathematics - LPO and LLPO translate the usually `forbidden' principle of excluded middle and de Morgan's laws, respectively. Though intuitionistic reasoning rejects such ideas, their representations as LPO/LLPO have realizers that correspond to discontinuous operations of varying degree of discontinuity - see \cite{Brattka2011} for details on how these and other principles, such as Markov's Principle, become somewhat unproblematic owing to their continuous, and thereby computable, realizers in this setting.

To illustrate how such a notion of choice is handled in the literature, we state the following theorem, for which the proof can be found in \cite{Brattka2011}:

\begin{theorem}[\protect{\cite[Thm. 2.10]{Brattka2011}}]\label{thm:compmetspaceLLPO}
Let $X$ be a computable metric space. Then $CC_{\mathcal{K}(X)} \leq_W \widehat{LLPO}$. If there is a computable embedding $\iota : 2^\omega \hookrightarrow X$, then $CC_{\mathcal{K}(X)} \equiv_W \widehat{LLPO}$.
\end{theorem}

We omit the proof of this, but it can be found in \cite{Brattka2011}. The following definitions are taken from \cite{Brattka2009}.

\begin{definition}[Weakly Computable]\index{weakly computable}
A function $F : \subseteq X \rightrightarrows Y$ on represented spaces $X$ and $Y$ is called \emph{weakly computable} if $F \leq_W \widehat{LLPO}$. Similarly, we also call functions like $F$ \emph{weakly continuous} given $F \leq_W \widehat{LLPO}$ holds with respect to some oracle.
\end{definition}

Based off the previous theorem \ref{thm:compmetspaceLLPO} and definition we can get the following corollary:

\begin{corollary}[\protect{\cite[Cor. 2.11]{Brattka2011}}]
Let $X$ be a represented space and let $Y$ be a computable metric space. Any weakly computable single-valued operation $F :\subseteq X \rightarrow Y$ is computable.
\end{corollary}

For our tiling problem equivalences, we will need the following definition, taken from \cite{Brattka2012} and \cite{BrattkaPauly2010}:

\begin{definition}[Closed Choice]\index{closed choice}\label{def:ClosedChoice}
Let $(\textbf{X}, d_X)$ be a represented space. Then the \emi{closed choice} operation of this space is defined by \[ C_\textbf{X} : \subseteq \mathcal{A}(\textbf{X}) \rightrightarrows \textbf{X}, A \mapsto A \] where $\mathcal{A}(\textbf{X})$ are the closed subsets of \textbf{X}, and our choice function takes some non-empty closed subset $A \in \mathcal{A}(\textbf{X})$ and outputs some point $x \in A$. We therefore have $dom(C_X) := \{ A \in \mathcal{A}(X) : A \neq \emptyset\}$
\end{definition}

We will be specifically interested in closed choice for Baire space - as this is where the trees we have been considering so far are found. \cite{Brattka2011} demonstrates how this is, in a sense, the `hardest' kind of choice, by the following definitions and theorem below. 

\begin{definition}
We define the following choice maps as follows:
\begin{enumerate}
\item \textbf{Discrete choice} - $$C_{\omega} : \subseteq \mathcal{A}(\omega) \rightrightarrows \omega,dom(C_{\omega}) = \{ A \subset \omega :  A \neq \emptyset\} $$
\item \textbf{Interval choice} -  $$C_I : \subseteq \mathcal{A}([0,1]) \rightrightarrows [0,1],dom(C_I) = \{ [a,b]:  0 \leq a \leq b \leq 1\} $$
\item \textbf{Proper interval choice} - $$C_I^- : \subseteq \mathcal{A}([0,1]) \rightrightarrows [0,1] ,dom(C_I^-) = \{ [a,b]: 0 \leq a \leq b \leq 1 \} $$ 
\item \textbf{Compact choice} -  $$C_K : \subseteq \mathcal{A}([0,1]) \rightrightarrows [0,1],dom(C_K) = \{ K \subseteq [0,1] : K \neq \emptyset, K \text{ compact} \} $$
\end{enumerate}

However, Brattka and Gherardi also present choice principles as boundedness principles instead of principles of choice over intervals. The intuition here stems from a similar question asked in \cite{BrattkaPauly2010}:
\begin{center}
``Given information about what does not constitute a solution, find a solution."
\end{center}
\begin{flushright}
\emph{-- Brattka, Brecht, Pauly in \cite{BrattkaPauly2010}}
\end{flushright}

So, by seeing choice principles as boundedness principles, we shift our view to the defined negative information about the represented set $A$, which is then given explicitly in the form of a finite number of bounds. It turns out that this is very useful in reducing problems in analysis - they often turn out to have a `boundedness representation'.

We find that the boundedness principle analogues of the above choice principles, given in \cite{Brattka2011}, are as follows:

\begin{enumerate}
\item $B : \mathbb{R}_< \rightarrow \mathbb{R}, x \mapsto x$
\item $B_I : \mathbb{R}_< \times \mathbb{R}_> \rightrightarrows \mathbb{R}, (x,y) \mapsto [x,y], dom(B_I) = \{ (x,y):x \leq y\}$
\item $B_I^- : \mathbb{R}_< \times \mathbb{R}_> \rightrightarrows \mathbb{R}, (x,y) \mapsto [x,y], dom(B_I^-) = \{ (x,y) : x < y \}$
\item $B_I^+ : \mathbb{R}_< \times \overline{\mathbb{R}_>} \rightrightarrows \mathbb{R}, (x,y) \mapsto [x,y], dom(B_I^+) = \{ (x,y) : x \leq y \}$
\end{enumerate}

Where $\mathbb{R}, \mathbb{R}_<, \mathbb{R}_>$ are equipped with ordinary Cauchy representations $\rho$ of the real numbers, the left $\rho_<$, and right $\rho_>$ respectively.

\end{definition}

These various notions of choice illustrate the degree of detail we can command in this theory. Brattka \etal in \cite{Brattka2011} illustrate the relationships between these choice operators in the following theorem. They denote these \emi{choice chains}, indicating the relationships between choice principles, boundedness principles, and our omniscience principles.

\begin{theorem}[\protect{\cite[Thm. 3.10]{Brattka2011}}][Choice Chains]
It is obtained in \cite{Brattka2011} that:
\begin{enumerate}
\item $LLPO \leq_W C_I^- \leq_W C_I \leq_W C_K \equiv_W \widehat{LLPO} \leq_W C_A$
\item $LPO \leq_W C_\omega \leq_W B_I^+ \leq_W C_A \leq_W C \equiv_W \widehat{LPO}$
\item $LLPO \leq_W LPO$, $C_I^- \leq_W C_\omega$, $C_I \leq_W B_I^+$
\end{enumerate}
\end{theorem}

As a finale, Brattka proves the following theorem:

\begin{proposition}[\protect{\cite[Prop. 3.7]{Brattka2011}}]
$$ B \equiv_W C \equiv_W \widehat{LPO} $$
\end{proposition}

This is somewhat surprising when read out loud - all of our boundedness principles are Weihrauch equivalent to all of our (closed) choice principles, both of which are equivalent to the Limited Principle of Omniscience. This result and the background theory and definitions in Weihrauch reducibility provide the backdrop for our result we present in the next subsection.

\section{Weihrauch Reducibility and Tiling Problems}

We will look specifically at Closed Choice on Baire Space, denoted $C_{\omega^{\omega}}$, defined above.

We require a proper intuition for $C_{\omega^\omega}$ - namely that any realizer for this principle in Baire space takes a tree $T \subset \omega^{< \omega}$ as input, and returns a path through it, in keeping with our definitions above.

\begin{definition}
Let the following notations be given:
\begin{itemize}
\item Let $\mathbb{W}$ denote the set of all possible Wang tiles, represented as 4-tuples.
\item Let $\mathscr{T}_{\mathbb{W}}$ denote the class of all possible tilings of all possible Wang tiles.
\end{itemize}
\end{definition}

In the spirit of our previous definitions, we define the following class.

\begin{definition}\label{def:ChooseTiling}\index{$ChooseTiling$}
Let $ChooseTiling$ or $CT$ be a multivalued operator such that $$ CT : \subseteq \mathcal{P}(\mathbb{W}) \rightrightarrows \mathscr{T}_{\mathbb{W}}, S \mapsto \mathcal{T}_S$$
where $S$ is a set of prototiles, and $\mathcal{T}_S$ is an $S$-tiling. $ChooseTiling$ as an operator/principle takes some subset of all possible Wang prototiles $S \subset \mathbb{W}$ and returns a total planar $S$-tiling $\mathcal{T}_S$, as a tiling function $f: \mathbb{Z}^2 \rightarrow S$.
\end{definition}

Note that we do not use definition of $TILE$ from \ref{def:TILE} in this definition, as we defined $TILE$ to be a set of indices for Turing Machines. However, our realizer for $CT$ will not be computable.

Intuitively this operator takes some set of prototiles and returns a total planar tiling. Thus, a realizer for $CT$ is a function $$F :\subseteq \omega^\omega \rightarrow \omega^\omega$$ which takes some set of prototiles $S$, and outputs some infinite sequence corresponding to a total planar $S$-tiling given by the tiling function $f: \mathbb{Z}^2 \rightarrow S$

We present the following result:

\begin{theorem}[C. 2019]\label{thm:CT-SW-ClosedC}
$$ CT \equiv_{sW} C_{\omega^{\omega}} $$
\end{theorem}

\begin{proof}
As per the definition of strong Weihrauch reducibility, we require to show the following reductions hold in order to get that $CT$ and $C_{\omega^\omega}$ are strong Weihrauch equivalent:
\begin{align*}
(CT \leq_{sW} C_{\omega^{\omega}}) \, \wedge \, & (C_{\omega^{\omega}} \leq_{sW} CT) \\
\end{align*}

Denote realizers $C \vdash C_{\omega^\omega}$ and $T \vdash CT$, we thus require computable $H,K :\subseteq \omega^\omega \rightarrow \omega^\omega$ such that $$ C = K(TH)$$ as well as computable $I,J :\subseteq \omega^\omega \rightarrow \omega^\omega$ such that $$ T = J(CI)$$ which we can represent with the following commutative diagram:

\begin{center}
\begin{tikzpicture}[scale=2.5]
\node (1) at (0,1) {$\omega^\omega$};
\node (2) at (1.5,1) {$\mathbb{W}$};
\node (3) at (0,0) {$\omega^\omega$};
\node (4) at (1.5,0) {$\mathscr{T}_{\mathbb{W}}$};
\node (A) at (-1.5,1) {$\omega^\omega$};
\node (B) at (-3,1) {$\mathcal{A}(X)$};
\node (C) at (-1.5,0) {$\omega^\omega$};
\node (D) at (-3,0) {$X$};


\path[->,font=\scriptsize]
(A) edge node[above]{$d_{\mathcal{A}(X)}$} (B)
(A) edge node[left]{$C$} (C)
(B) edge node[left]{$C_{\omega^\omega}$} (D)
(C) edge node[below]{$d_{X}$} (D);

\path[->,font=\scriptsize]
(1) edge node[above]{$d_{\mathbb{W}}$} (2)
(1) edge node[right]{$T$} (3)
(2) edge node[right]{$CT$} (4)
(3) edge node[below]{$d_{\mathscr{T}_{\mathbb{W}}}$} (4);

\path[->,font=\scriptsize]
(A) edge [bend left] node[above]{$H$} (1)
(3) edge [bend left] node[below]{$K$} (C)
(1) edge [bend left] node[above]{$I$} (A)
(C) edge [bend left] node[below]{$J$} (3);


\end{tikzpicture}
\end{center}

We will utilise our constructions in the proof of \ref{thm:TILE-ILL}, and notice that in that construction of tilings from trees, all the parts of our construction used in the proof are computable. This important detail will inform much of the proof of this theorem. 

Note that underlying actions of $C$ and $T$: 
\begin{itemize}
\item $C$ takes some closed subset of Baire space, a tree, and finds some infinite path through it, and returns this as its output.
\item $T$ takes some finite or infinite set of prototiles $S$ and finds some infinite sequence of tiles that corresponds to a total tiling of the plane using tiles from $S$ respecting all edge meet conditions.
\end{itemize}

We will first show that there are computable $I,J : \subseteq \omega^\omega \rightarrow \omega^\omega$ such that $T = J(CI)$ in order to prove $CT \leq_{sW} C_{\omega^\omega}$. We first notice that $T$ is a function that takes a set of prototiles and produces and infinite sequence corresponding to some infinite planer tiling. Thus, $I$ will encode information about the possible $S$-tilings in a way that we can ask $C$ to process this and give us an answer that $J$ will translate back into some tiling.

Given some prototile set $S$ as input to $T$, our computable $I$ will construct the tree of possible tilings as we saw constructed in the proof of Wang's Extension Theorem, theorem \ref{thm:WangExtension} in this thesis. Specifically, $I$ will code uniquely each tile in $S$, and then proceed to code each successively larger sequence of possible tiles in square rings of tiles, joining them into the tiling tree $T_S$ based on the required edge-match criteria. With this done, we have a full tree of valid tilings given by successively larger rings that properly extend the previous finite square patch of tiles. This encoded tiling tree $T_S$ is a subset of Baire space, and so it is this tiling tree that we supply to $C$.

Given trees are closed subsets of Baire space, this is a problem that $C$ will be able to provide an answer for. Thus for our tiling tree $T_S$, $C(T_S) = p$, with $p$ some infinite path through $T_S$. This path will represent a planar tiling, by the construction of $T_S$ by $I$.

We can take the infinite path $p$ through our tiling tree generated by $I$ and then decode this as a tiling by assigning all of the tiles for initial segment of $p$ by decoding the specific arrangement of tiles coded by $I$ into the tiling tree for our $S$. Given we computably generated $T_S$, we can match up each successive initial segment of $p$ by decoding each of the encoded tiles in $p$ without knowledge of the input to $I$, with the tiles that should be placed around the previous patch of tiles being coded in each successive segment of $p$. This is our computable function $J$ that will complete our reduction, and is essentially a computable inverse of the operation of $I$, taking a coded sequence corresponding to edge-matched finite patches of $S$-tilings and recovering a planar $S$-tiling from this.

With this now done, we have successfully shown that $T$ can be computed by means of translation of tree sets into input for $C$ by $I$, and the output of $C$ can then be reinterpreted by $J$, such that we have satisfied the requirement and shown that $T = J(CI)$ is a realizer for $CT$.

Next we prove $C_{\omega^\omega} \leq_{sW} CT$, by finding computable $H,K : \subseteq \omega^\omega \rightarrow \omega^\omega $ such that $C = K(TH)$. By our intuition above we require a computable function $H$ to convert some tree into sets of prototiles, which will then allow our realizer $T$ to construct a total planar tiling, and return this to us as output. We then require a computable $K$ to take this tiling and recover from it an infinite path, which will then be returned as one of the possible paths from our original tree.

We take the two constructions in our proof of theorem \ref{thm:TILE-ILL} to be the computable functions that we need. We will explicitly show which parts relate to this reduction for this part of the proof. 

First, note that for a given tree, converting each path into the library $\mathcal{S}$ is a computable task. Although in the previous proof, we require a path to then choose the $S_e \subset \mathcal{S}$ for our original tiling, here we can pass this prototile set to our realizer $T$ and it will give us a sequence corresponding to a tiling of the plane. Explicitly, we construct this library as follows from the proof of theorem \ref{thm:TILE-ILL}:

\begin{itemize}
\item Fix a root tile with the tuple $\langle M^L_0, \lambda^U, M^R_0, \lambda^D \rangle$ and put this tile into $\mathcal{S}$.
\item For all the $c^j_i$ and $M_i$ colour the mid-row tiles.
\item For all $c^j_i$ colour all of the quadrant tiles, and put these into $\mathcal{S}$.
\item For each point (a string) $p$ in our input tree we add column tiles for each initial segment $\sigma$ and $\sigma^\frown n$ in $p$ - note, we still take two copies of each and construct two tiles for each successive symbol in each path, as one is required to go up and the other in the mirror position downwards.
\end{itemize}

With our full library $\mathcal{S}$ constructed, we now have an infinite set of prototiles which we can pass as the input to our realizer $T$. The output from this will be a planar tiling about which we already know the useful properties, namely that from this we can recover the path coded in each of the $\mathcal{S}$-tilings.

We can extract the path from an $\mathcal{S}$-tiling in the following manner. Our following computable method will be the same as the method to extract the path $p$ from our $S_e$-tiling in the proof of theorem \ref{thm:TILE-ILL}:
\begin{enumerate}
\item If we choose the root tile, read upwards along the column of tiles, from which we can recover a path $p$.
\item If we choose a mid-row tile, then we follow the descending chain of $M_i$ colours to the root tile, and then go to step 1.
\item If we choose a quadrant tile, then for our given $i \in \omega$ from our chosen tile:
\begin{itemize}
\item If $c^1_i$ or $c^2_i$ then follow all the tiles down to the mid-row tiles, and go to step 2.
\item If $c^3_i$ or $c^4_i$ then follow all the tiles up to the mid-row tiles, and go to step 2.
\end{itemize}
\end{enumerate}

Thus we have computable functions $H$, that creates from a tree a valid input for $T$, and a computable $K$, that takes the output from $T$ and extracts an infinite path $p$ for our original tree. This is satisfying the same function as the realizer $C$, thus $C = K(TH)$ is satisfied and is a realizer for $C_{\omega^\omega}$, completing our theorem.

\end{proof}

\section{Weihrauch Reductions for Weak Planar Tilings}

We can also prove a similar result for the following tiling principle, based around the definition for $WTILE$ we originally gave in definition \ref{def:WTILE}. 

\subsection{Weihrauch Equivalence for $CWPT$}

We first need the following definition of a `wild card':
\begin{definition}
Let $*$ denote the \emi{wild card} that satisfies the edge meet conditions of any Wang prototile in $\mathbb{W}$ in a tiling function $f : \mathbb{Z}^2 \rightarrow S \cup \{ * \}$, for a given set of prototiles $S \subset \mathbb{W}$.
\end{definition}

The wild card tile is intended to give us a way of handling the `blank', or `no tile', possibility that we first encountered in our definition of $WTILE$. Thus, an infinite region that is not tiled will be mapped by infinitely many wild cards. We can now continue on and define non-total tilings of the plane by adding this wild card to our prototile sets.

\begin{definition}\label{ref:CWPT}\index{$ChooseWeakPatchTiling$}
Let $ChooseWeakPatchTiling$, shortened to $CWPT$, be such that $$ CWPT :\subseteq \mathcal{P}(\mathbb{W}) \rightrightarrows \mathscr{T}_{\mathbb{W}}, S \mapsto \mathcal{T}_S $$ where $S$ is a set of prototiles, and $\mathcal{T}_S$ is an $S$-tiling. Similar to $CT$ defined in \ref{def:ChooseTiling}, $ChooseWeakPatchTiling$, is an operator/principle that takes some subset of all possible Wang prototiles $S \subset \mathbb{W}$ such that $S$-tilings returns a connected planar, but not necessarily total, $S$-tiling $\mathcal{T}_S$ given by $$f: \mathbb{Z}^2 \rightarrow S \cup \{ * \}$$ where $*$ is the `tiling wild card' defined above. $CWPT$ also returns an infinite connected region $R \subseteq \mathbb{Z}^2$ which is covered by this infinite connected patch of tiles.
\end{definition}

The following result can now be demonstrated:

\begin{theorem}[C. 2019]\label{thm:CWPT-SW-ClosedC}
$$ C_{\omega^\omega} \equiv_{sW} CWPT $$
\end{theorem}

\begin{proof}
We first reiterate that we are explicitly after two reductions to obtain our equivalence, explicitly:
\begin{align*}
(C_{\omega^\omega} \leq_{sW} CWPT ) \wedge ( CWPT \leq_{sW} C_{\omega^\omega} )
\end{align*}

Let our realizers be $C \vdash C_{\omega^\omega}$ and $W \vdash CWPT$, with $C,T : \subseteq \omega^\omega \rightarrow \omega^\omega$. As before, we want computable $H,K,I,J : \subseteq \omega^\omega \rightarrow \omega^\omega$ such that the following diagram commutes:

\begin{center}
\begin{tikzpicture}[scale=2.5]
\node (1) at (0,1) {$\omega^\omega$};
\node (2) at (1.5,1) {$\mathbb{W}$};
\node (3) at (0,0) {$\omega^\omega$};
\node (4) at (1.5,0) {$\mathscr{T}_{\mathbb{W}}$};
\node (A) at (-1.5,1) {$\omega^\omega$};
\node (B) at (-3,1) {$\mathcal{A}(X)$};
\node (C) at (-1.5,0) {$\omega^\omega$};
\node (D) at (-3,0) {$X$};


\path[->,font=\scriptsize]
(A) edge node[above]{$d_{\mathcal{A}(X)}$} (B)
(A) edge node[left]{$C$} (C)
(B) edge node[left]{$C_{\omega^\omega}$} (D)
(C) edge node[below]{$d_{X}$} (D);

\path[->,font=\scriptsize]
(1) edge node[above]{$d_{\mathbb{W}}$} (2)
(1) edge node[right]{$W$} (3)
(2) edge node[right]{$CWPT$} (4)
(3) edge node[below]{$d_{\mathscr{T}_{\mathbb{W}}}$} (4);

\path[->,font=\scriptsize]
(A) edge [bend left] node[above]{$H$} (1)
(3) edge [bend left] node[below]{$K$} (C)
(1) edge [bend left] node[above]{$I$} (A)
(C) edge [bend left] node[below]{$J$} (3);


\end{tikzpicture}
\end{center}

We will prove the more straightforward of the two first, namely that $CWPT \leq_{sW} C_{\omega^\omega}$. To do this, we will require our two computable functions $I,J$ to be such that $$ W = J(CI) $$ This will be achieved in the same way as for the proof of theorem \ref{thm:CT-SW-ClosedC}.

We begin by using our intuition from the proof of theorem \ref{thm:SNT-WELL}, where we can think of our prototile sets as coding paths through trees. As for the proof there, we let $I$ be the function that codes the tree of all possible tilings from our given prototile set, but this time we allow for each boundary enumerated into this tree to be incomplete - as we only care that our tilings are connected, not that they are total. 

Despite this, we still arrive at a tree that is some subset of $\omega^\omega$. This follows from noticing that for a given prototile set $S$, our tiling functions $f : \mathbb{Z}^2 \rightarrow S$ can be extended in the following way $$ f: \mathbb{Z}^2 \rightarrow (S \cup \{ * \}) $$ where $*$ stands for the ``no tile here" option we have now allowed for $f$ to be a weak tiling of the plane. 

For this $f$ there exists an infinite patch $P \subseteq \mathbb{Z}^2$ such that
\begin{itemize}
\item $f$ follows the tiling rules.
\item $f$ is not $*$ on $P$.
\item $P$ is connected.
\item $|P| = \infty$.
\end{itemize}

This gives us some $\xi \in \Sigma^1_1$ such that $$\exists f \, \exists P \, (\xi(f,P))$$ is true if and only if $f$ weakly tiles the plane according to our definition of $CWPT$. By \cite[p.4]{sacks_2017} we have a $\Sigma^1_1$-normal form given which allows us to rewrite this formula as $$ \exists f \, \exists P \, \exists X (\psi(f,P,X)) $$ is true for $X$ a sequence of Skolem functions and $\psi \in \Pi^0_1$. This gives us a tree with which $C$ can find a path for. Thus, our $I$ is defined and computable.

Once $C$ returns a path, this path will correspond to an infinite sequence of tiles in the plane, and so our $J$ will take this and reconstruct our tiling from the selected paths through the generated tiling tree from $I$ that $C$ has provided a path from. Thus, the input and output adaption functions $I,J$ are both computable, and by utilising $C$ we have that $W  = J(CI)$, giving us $CWPT \leq_{sW} C_{\omega^\omega}$.

Next, we will prove that $C_{\omega^\omega} \leq_{sW} CWPT$. To begin, we will once again require that our computable $H,K$ be such that $$ C = K(WH) $$ This naturally comes about given that our definition of $CWPT$ includes not just the tiling function, but the knowledge of which region is an infinite patch of the $\mathbb{Z}^2$ lattice that is tiled by tiles from $S$.

Our computable $H$ will be given by the prototile set construction similar to that given in \ref{thm:SNT-WELL} - we take the input that is some tree $T \subseteq \omega^{< \omega}$, and then generate the tile set $S$ as follows. Fix $R$ $B$, and $P$ to be `red', `blue', and `purple' respectively - effectively making certain quadrants of Wang prototiles fixed colours. The prototiles we need to create for $S$ are:

\begin{itemize}
\item Add a unique root tile $\langle R, \lambda^U, B, P \rangle $ into $S$:
\begin{center}
\sampletile{$R$}{$\lambda^U$}{$B$}{$P$}
\end{center}
\item For each path $\sigma \in [ T ]$ add the tile: $\langle R, \sigma^\frown n, B, \sigma \rangle$:
\begin{itemize}
\item \textbf{NB}: We identify the empty string $\lambda^U$ with $\sigma(0)$
\end{itemize}
\begin{center}
\sampletile{$R$}{$\sigma^\frown n$}{$B$}{$\sigma$}
\end{center}
\end{itemize}

With this, we then give this $S$ as input to $W$, which will return two things:
\begin{enumerate}
\item A tiling function $f : \mathbb{Z}^2 \rightarrow S$.
\item A region $R \subseteq \mathbb{Z}^2$ containing an infinite patch of tiles.
\end{enumerate}

Intuitively, our tilings given by the coding above are long snakes of tiles where an infinite path is coded going up from the root tile. Given we have all this information available to $K$, we can make $K$ the computable function that first chooses the minimum point in $R$ - i.e. the point $(x,y)$ that has the lowest values for $x$ and $y$ - which is the point closest to $(0,0)$. 

With this point given, we can then follow the tiles from this point down until we reach the root tile $\langle R, \lambda^U, B, P \rangle $. This is done by fixing the $x$ co-ordinate from this point, and then subtracting one from $y$ until we find the $m$ such that $ (x,y-m) = \langle R, \lambda^U, B, P \rangle $.

With this found, we can then read each initial segment of an infinite path $\sigma \in [T]$. With this recovered, we can return this as an infinite path through the original tree $T$ that has been obtained by our realizer $W$. Thus we have satisfied $C = K(WH)$ as required. 

Finally, we note that both directions give our result, $C_{\omega^\omega} \equiv_{sW} CWPT$.
\end{proof}

\subsection{Weihrauch Reducibility for Other Weak Tiling Principles}

We will first state a neat notion of compositional product used in Weihrauch reducibility - Brattka and Pauly give the following theorem in \cite{Brattka2016}:

\begin{theorem}[\protect{\cite[Prop. 3.5 \& Thm. 4.1]{Brattka2016}}]\label{thm:supCompProd}
For every $f$ and $g$, the following supremum exists:
$$ sup\{ f_0 \circ g_0 : f_0 \leq_W f \wedge g_0 \leq_W g \} $$
\end{theorem}

As such, we will define the following compositional product:

\begin{definition}[Compositional Product]
The \emi{compositional product} of $f$ and $g$, written $f \star g$, is precisely this supremum from theorem \ref{thm:supCompProd}.
\end{definition}

The core idea in this compositional product is that we can find two sub-principles $f_0$ and $g_0$ that are each Weihrauch reducible to the principles we are interested in, and then use the composition of these new principles to achieve a Weihrauch reducibility of the two original principles composed. This enables us to sequentially apply different principles in order to obtain new Weihrauch reductions - a technique that we will now utilise.

%

Note that the proof of theorem \ref{thm:CWPT-SW-ClosedC} requires that we provide the exact location of the infinite patch containing our infinite patch tiling is returned in addition to our tiling function. However, given we expected our weakly tiling prototile set $S$ to be non-total we knew how to `read' an $S$-tiling when given a known-infinite region $R$ which was tiled by $S$.

We now explore what happens if we change these requirements to take some prototile set $S$ that is total, and return an infinite region $R \subset \mathbb{Z}^2$ and a tiling function $f : R \rightarrow S $ - thereby reducing a total tiling to a weaker tiling of the plane.

\begin{definition}[$CIPT$]\index{$CIPT$}\label{def:CIPT}
Let $ChooseInfinitePatchTiling$, or $CIPT$ be defined similarly as before $$ CIPT : \subseteq \mathcal{P}(\mathbb{W}) \rightrightarrows \mathscr{T}_{\mathbb{W}} $$ with $CIPT$ taking a set of prototiles $S$ that gives total tilings of the plane, and returning the pair $(R,t)$, composed of an infinite connected region $R \subset \mathbb{Z}^2$ with a tiling function $t : R \rightarrow S$ such that we have an infinite $S$-tiling on $R$.
\end{definition}

We now have the machinery we need to state the following theorem:

\begin{theorem}[C. 2019]\label{thm:CC-WKLstarCIPT}
$$ C_{\omega^\omega} \leq_W C_{2^\omega} \star CIPT $$
\end{theorem}

Here, $C_{2^\omega}$ denotes the Closed Choice principle on Cantor Space which is equivalent to Weak K\"onig's Lemma (WKL) which we defined in section \ref{sec:KLemmaWKL}. As such, we can pass $C_{2^\omega}$ a finitely branching infinite tree, and it will return a path through it. Our use of this in the compositional product is due to the fact that we cannot always guarantee in a weak tiling of the plane that we can easily find our infinite path in a computable way given an input prototile set that is total.

\begin{proof}
We require two principles $f$ and $g$ such that $f \leq_W C_{2^\omega}$ and $g \leq_W CIPT$ and aligned in such a way that $f \circ g \geq_W C_{\omega^\omega}$. Let our $g$ and $f$ be defined as follows:
\begin{itemize}
\item $g$ will be the principle of taking some tree $\mathcal{T} \subseteq \omega^{< \omega}$ and returning some pair $(R,t)$ with $R \subset \mathbb{Z}^2$ an infinite connected region, tiled by $t : R \rightarrow S$, and $S$ is a prototile set with total tilings of the plane for any ill-founded tree $\mathcal{T}$.
\item $f$ will be the principle that will take a pair $(R,t)$ as above, and return an infinite sequence of tiles through the infinite connected region $R$ based on the tiling $t$.
\end{itemize}
For realizers $G \vdash g$ we will make use of the construction from the proof of theorem \ref{thm:SNT-WELL}, and in the final part of the proof, we will decode an infinite path through $\mathcal{T}$ from an infinite sequence of tiles from this construction.

Our proof will come in three main parts:
\begin{enumerate}
\item We first require computable $H,K$ such that $G = K(\langle id, TH \rangle)$ is a realizer for $g$, given $T \vdash CIPT$.
\item Next, we require computable $I,J$ such that $F = J(\langle id, WI \rangle)$ is a realizer for $f$, given $W \vdash C_{2^\omega}$
\item Finally we then require computable $X,Y$ such that $C = Y(\langle id, AX \rangle)$ as a realizer for $C_{\omega^\omega}$ given $A \vdash f \circ g$.
\end{enumerate}
With $H,K,J,I,X,Y,A : \subseteq \omega^\omega \rightarrow \omega^\omega$.

By $A = FG$ from the above, we will aim to arrive at the final form $$ C = Y(\langle id, FGX \rangle) $$ is a realizer for $C_{\omega^\omega}$. We will later prove in theorem \ref{thm:DPW-CC} that $CIPT \leq_{sW} C_{\omega^\omega}$, hence we only focus on this particular direction for our theorem.

Here follows the general plan for our proof. Recall that a realizer $C$ for $C_{\omega^\omega}$ takes some Baire space tree $\mathcal{T} \subseteq \omega^{< \omega}$ and returns a path through it. Thus, we need to align our realizers such that we take this tree $\mathcal{T}$, construct some prototile library $S$ that gives total planar tilings, and then show that if we restrict our planar $S$-tilings to some infinite region $R \subset \mathbb{Z}^2$, we can still recover an infinite path through our original $\mathcal{T}$ by means of $C_{2^\omega}$, which we recall is Weak K\"onig's Lemma. We do this last step by finding some infinite sequence of tiles through $R$, and set up the construction of our $S$ such that we can recover the path through $\mathcal{T}$ by means of a path through the spanning tree of $R$.

Intuitively we want to show that even if we remove much of the structural information of a total tiling of the plane but retain some infinite part, we can still find some reduction for $C_{\omega^\omega}$ by utilising a weaker closed choice principle to `fix' the damage we did to our original tiling.

\begin{proof}[Proof of (1)] - Let $g$ be the principle that takes some tree $\mathcal{T} \subseteq \omega^{< \omega}$ as input, and returns $(R,t)$, composed of an infinite connected region $R \subset \mathbb{Z}^2$, and a tiling function $t : R \rightarrow S$, where $S$ is the prototile set that has a total planar tiling for some path $p \in [\mathcal{T}]$ given $\mathcal{T}$ is ill-founded.

Let our $H$ be the computable function that takes as input some tree in Baire space, $\mathcal{T} \subseteq \omega^{<\omega}$ and produces a prototile set $S$ given by the construction we used in the proof of theorem \ref{thm:SNT-WELL}. The resulting prototile set gives a set with a total tiling for a path in $\mathcal{T}$, given that $\mathcal{T}$ is ill-founded.

$H$ passes this prototile set $S$ to our realizer $T \vdash CIPT$ which returns our $(R,t)$ as desired for our output. As such, our computable $K$ does nothing to this, and we have that $g \leq_W CIPT$.
\end{proof}

\begin{proof}[Proof of (2)] - For this, we want $f$ to be the principle of taking some pair $(R,t)$, comprised as above of a tiling for an infinite connected region $R \subset \mathbb{Z}^2$ given by a $t : R \rightarrow S$, and we wish to return some infinite sequence of tiles through this infinite connected tiling over $R$.

To obtain our reduction, we let our computable $I$ be the function that takes some tiling on an infinite region $R$ and computably constructs a spanning tree in the graph theoretic sense by starting at some point closest to $(0,0)$ and enumerating each tile based on the von Neumann neighbourhood of the edge meets for each successive tile that has not already been enumerated. This algorithm is generally a breadth-first search along the following lines:
\begin{enumerate}
\item Choose some tile in the $S$-tiling of $R$, and set the root node of $\mathcal{T}_R$ as the empty string $\lambda$.
\item Enumerate the tiles to the upper, lower, right, and left sides if they are available as successors in the resulting tree and have not yet been enumerated into the tree.
\item Group each of the successors by whether they are upper/lower or left/right in order to obtain binary branching.
\end{enumerate}
By the end of this process we have some $\mathcal{T}_R$, a finitely branching tree, which is bounded given the finite bound on the neighbourhood around each tile. We can pass $\mathcal{T}_R$ to a realizer $W \vdash C_{2^\omega}$. This will take our bounded branching tree and give us some infinite path through it. 

With this returned, we pass this to a computable function $J$ which takes the path returned by $W$ and decodes the infinite sequence of tiles through the tiled region $R$ that $W$ has found. $J$ can computably recover this by the fact that we can program it to decode each 4-tuple as a Wang tile in our tiling, and so obtain the full infinite sequence of tiles. $J$ finally outputs this infinite sequence of tiles through $R$.
\end{proof}

Now that we have our two subordinate principles defined and shown to be Weihrauch reducible to our desired components in our compositional product, we can now complete the proof by showing how these two principles work to give our desired reduction of $C_{\omega^\omega} \leq_W C_{2^\omega} \star CIPT$.

\begin{proof}[Proof of (3)] - The final stage of this proof will show that using a realizer $A \vdash f \circ g$ will be such that $C = Y(\langle id, AX \rangle)$ is a realizer for $C_{\omega^\omega}$ for computable $X,Y : \omega^\omega \rightarrow \omega^\omega$.

Our input adaption $X$ is a `do nothing' function, passing the input tree $\mathcal{T} \subseteq \omega^{< \omega}$ to a realizer for $G$.

$G$ returns a tiling $t : R \rightarrow S$ for an infinite region $R \subset \mathbb{Z}^2$ that contains some infinite path through $\mathcal{T}$ by the process described above. However, we cannot predict enough about the structure of the tiling of $R$, and so pass this to our realizer $F \vdash f$ that can take such a tiling on an infinite region $R$ and locate an infinite sequence of tiles through this region. Given our construction of $F$ we know that we will always locate the path by means of the bounded branching on the spanning tree across $R$ which is procured by means of a breadth-first search, and given $R$ is infinite we will get our tile sequence accordingly in this composition.

Our output adaption $Y$ works as follows: Recall that the prototile set $S$ generated inside $G$, taken from the proof of theorem \ref{thm:SNT-WELL}, has some coding of an initial segment $\sigma$ of the path $p$ we desire in every prototile, thus we can take the infinite sequence of tiles given by the realizer $F$ and computably decode each initial segment of $p$ in turn.

Our tile sequence may begin on any tile, but this will give us some initial segment $\sigma \prec p$, and although its immediate neighbours may not give us additional bits, it is certain that some tile at some point will give us some additional bit $i \in \omega$ that such that $\sigma^\frown i \prec p$. This is guaranteed by the fact that the construction of our prototile set $S$ has longer initial segments of $p$ found in any direction you care to look, so as long as $R$ is infinite and we find some infinite path through it, we will certainly recover an infinite path in $p \in [\mathcal{T}]$ by means of this process.

As such, the composition of $f \circ g$ and the corresponding composition of the various realizers and input and output adaption functions $X,Y$, we can conclude that that input is a tree $\mathcal{T}$ in Baire space, and the output is a path through this tree $\mathcal{T}$, which is precisely the function of $C_{\omega^\omega}$, completing our reduction.
\end{proof}

Given this, we have our final result by the combination of $f \circ g$ as the compositional product equivalent to $C_{\omega^\omega}$ giving our desired result $$ C_{\omega^\omega} \leq_W C_{2^\omega} \star CIPT $$
\end{proof}

It should be noted that the \textbf{AIT} and \textbf{PIT} constructions given by definition \ref{def:AITPIT} could not be utilised in this proof, at least not without significant rework. The most immediate construction was that given for the proof of \ref{thm:SNT-WELL}.

\section{General Weihrauch Reducibility for Wang Domino Problems}

Let the following definition for the general ``Domino Problem for Wang Tiles'' principle be given as follows.

\begin{definition}\label{def:DPW}\emi{Domino Problem for Wang Tiles Principle}
Let $\mathbb{W}$ denote the set of all possible Wang prototiles, and $\mathscr{T}_\mathbb{W}$ be the set of all possible tilings given by all possible Wang prototiles. Let the general principle of ``Domino Problems for Wang Prototile Sets'', $DPW$, be given by $$ DPW : \subseteq \mathcal{P}(\mathbb{W}) \rightrightarrows \mathscr{T}_\mathbb{W}, S \mapsto \mathcal{T}_S $$ where $S \subset \mathbb{W}$ and $\mathcal{T}_S$ is the class of all $S$-tilings. 

Our input for $DPW$ is some prototile set $S \subset ( \mathbb{W} \cup \{ * \} )$, and our output is a planar tiling, given as a tiling function $f : \mathbb{Z}^2 \rightarrow (S \cup \{ * \} )$ that meets our edge requirements and has some infinite connected patch.
\end{definition}

Note that we intend $DPW$ to be the universal multivalued function from any set of prototiles $S$ to any possible $S$-tiling that has an infinite connected region. Our aim is to show that any additional requirements on Wang tilings are essentially captured by the closed choice principle on Baire space.

Given this is the general principle that governs any domino problem for sets of Wang prototiles, the following reducibility will apply to any given Domino Problem as a general case of sections of the proofs in this chapter.

\begin{theorem}[C. 2019]\label{thm:DPW-CC}
$$ DPW \leq_{sW} C_{\omega^\omega} $$
\end{theorem}

This would appear to be intuitively true, given that our method for capturing all possible tilings of any Wang prototile set on a tiling tree, that this tree is always constructable. 

\begin{proof}
We first note that every tiling problem is generally of the following form
\begin{align}\label{eq:DPW-Pi12}
	\forall X \, \exists Y \, (\varphi(X) \rightarrow \psi(X,Y)) 
\end{align}
where $$X \subset ( \mathbb{W} \cup \{ * \} ) $$ is a set of Wang prototiles that also permits the wild card $*$, which we used in the proof of theorem \ref{thm:CWPT-SW-ClosedC}, and $Y$ is a set that encodes a tiling of the $\mathbb{Z}^2$ lattice, and $\varphi$ and $\psi$ are arithmetical functions such that:
\begin{itemize}
\item $\varphi(X)$ holds if $X$ is a valid set of Wang prototiles.
\item $\phi(X,Y)$ holds if $Y$ is a valid $X$-tiling.
\end{itemize}

By this formulation, we see that the formula \ref{eq:DPW-Pi12} is in $\Pi^1_2$, and we can thus obtain the following normal form for this (see \cite[p.6]{sacks_2017} for how this is done) given by: $$ \forall X \, \exists Y \, \theta(X,Y) $$ where $\theta \in \Pi^0_1$, $X$ is our prototile set as before, and $Y$ captures all Skolem functions that give our $X$-tilings.

By Lemma \ref{lemma:trees-rec-rels}, it follows that any domino problem can be defined in this way, and is thereby representable by a path $p = [T]$ for some $\Pi^0_1$ tree $T \subset \omega^{< \omega}$. 

Given $C_{\omega^{\omega}}$, by definition, takes a tree that is a subset of Baire space and returns a path, our Weihrauch reduction follows.
\end{proof}

Intuitively, this theorem shows that our tiling trees that we have made use of are always of the correct kind for $C_{\omega^\omega}$ to process and return a path that encodes a planar tiling.

\subsection{Further Weak Tiling Problems}\label{sec:FurtherWeakTilingProblems}

There are other weak tiling problems we can consider, although they are currently just outside the scope of this thesis. Take the following example, $WeakInfinitePatchTiling$:

\begin{definition}\label{def:WIPT}[$WIPT$]
Let $WeakInfinitePatchTilings$, shortened to $WIPT$, be defined similarly as before $$ WIPT : \subseteq \mathcal{P}(\mathbb{W}) \rightrightarrows \mathscr{T}_{\mathbb{W}} $$ with $WIPT$ taking a set of prototiles $S$, and returning a tiling function $f : \mathbb{Z}^2 \rightarrow S \cup \{ * \}$ that we know has an infinite patch, but not where that patch is.
\end{definition}

Because of the lack of any knowledge of the resultant tiling, we do not have enough structure to gain enough information in order to extract an infinite tree without a lot of help. An initial estimate is that we would need the following in order to have a Weihrauch reducibility: $$C_{\omega^\omega} \leq_W C_{2^\omega} \star C_\omega \star WIPT$$ where $C_{\omega^\omega}$ and $C_{2^\omega}$ are closed choice on Baire space and Cantor space respectively, as before, and $C_\omega$ is the principle that takes some function $f: \omega \rightarrow \omega$ with $ran(f) \neq \omega$ as input, and outputs some $n \notin ran(f)$.

\chapter{Small ECA Tilings}
\label{chap6}
\setcounter{equation}{0}
\renewcommand{\theequation}{\thechapter.\arabic{equation}}


\epigraph{In mathematics you don’t understand things. You just get used to them.}{\textit{John von Neumann (attrib.)}}

This chapter presents a small tiling that encodes any Elementary Cellular Automaton in 15 prototiles. We also present some results about this class of automata that show that these prototile sets have interesting properties, specifically that they can be chaotic or Turing complete.

\section{Elementary Cellular Automata}

In this section, we will give formal definitions for Elementary Cellular Automata (ECAs) in preparation for coding them into small tiling sets. Our motivation for this originally was work that was ultimately carried out to its full completion in \cite{Rao2015} - aiming to find small, aperiodic tiling sets by means of coding small chaotic Elementary Cellular Automata into prototile sets. 

However, as we shall show in theorem \ref{thm:ECA-15-Hex}, we found a different way of encoding 3-ary functions as dynamical systems into prototile sets that represent their behaviour in the plane. We maintain the usual structure from previous work on coding Turing Machines into the plane - the 1-dimensional state is given left to right, with subsequent iterations going vertically.

We first give the background theory on ECAs, as well as a basic primer on the relevant pieces of chaos theory, and then proceed to detail results from Cook and Cattaneo \etal about Turing completeness and chaos in ECAs, respectively. Finally we give our representations of any ECA in prototile sets of only 15 tiles using our new construction, replete with diagrams and relevant corollaries. 

\subsection{Elementary Cellular Automata}

We will define a cellular automaton, and elementary cellular automaton (ECA) as per \cite{ECAMathworld}. They have appeared in a considerable amount of research, in areas as varied as computer science, symbolic dynamics, and as we shall see, chaos theory.

\begin{definition}
A \emi{cellular automaton} is pair $( X, R )$ where $X$ is a grid of some specific boundary topology\footnote{These grids can have joined boundaries, fixed boundaries, be bi-infinite, \etc \etc.} and $R$ is the `rule' that is applied successively to the grid. Each row is coloured based on the state of the colours on the previous row.
\end{definition}

We will specifically look at the subclass of cellular automata known as Elementary Cellular Automata, or ECAs. These were first introduced and studied by Wolfram in \cite{NKS}. 

\begin{definition}
\sloppy An \emi{elementary cellular automaton}, or ECA, is a cellular automaton $( X, R_n )$ where the rules in $R$ are derived from the binary representation of $n$. An ECA's rules for a cell at position $i$ on row $j$, written $c_{i,j}$, is determined by the triple $(c_{i-1,j-1} , c_{i,j-1} , c_{i+1,j-1})$. Thus, our rule set $R_n$ is given by a function $r_n : \{ 0,1 \}^3 \rightarrow \{ 0,1 \}$. 
\end{definition}

To acquire our rules for $R_n$ we first take the binary representation of $n$, and then send each of our 8 possible inputs sequentially to each bit of the binary representation of $n$, starting with the least significant bit.

To illustrate how this works, take $R_{30}$. We start with the 8-bit binary representation of 30, 00011110, and then map the inputs to $r_{30}$ as per table \ref{tbl:rule30}.

\begin{center}
\begin{table}[t]
\begin{tabular}{ c | c | c | c | c | c | c | c }
  111 & 110 & 101 & 100 & 011 & 010 & 001 & 000 \\
\hline
  0 & 0 & 0 & 1 & 1 & 1 & 1 & 0 \\  
\end{tabular}
\caption[Rule 30 Automaton Rules]{Rule 30 Automaton Rules}\label{tbl:rule30}
\end{table}
\end{center}

If we let our grid be the full $\mathbb{Z}^2$ lattice, then for each row $x \in \mathbb{Z}^2$, we can define $R_n : \mathbb{Z}^2 \rightarrow \mathbb{Z}^2$ as the successive application of $r_n$ to every triple $(x(i-1), x(i), x(i+1))$, for each cell $x(i) \in x$. 

\section{Some Results about ECAs}

We will define and discuss some background results for our work on ECAs and tilings. With ECA's already being an interesting and fertile area of study, we will give some background theory to the results, and then demonstrate that these results can also be realized as tiling problems by means of coding ECA's into prototile sets.

\subsection{ECAs and Chaos}

We will view ECAs as Discrete Time Dynamical Systems (DTDS) - that is, an iterated system that has discrete time steps. We write these as above, $(X,F)$, where $X$ is the \emi{phase space}, which is equipped with a distance function $d$, and a \emi{next state map} $F:X \mapsto X$, continuous on $X$ according to the topology on $X$ induced by $d$. We also assume that such a metric space $(X,d)$ is perfect - i.e. has no isolated points.

\begin{definition}[Sensitivity]\label{def:sensitivity}\index{topological sensitivity}
A DTDS $(X,F)$ is \emph{sensitive to initial conditions} if and only if there exists $\delta > 0$ such that $$ (\forall x \in X) \, (\forall \epsilon > 0) \, (\exists y \in X) \, (\exists n \in \mathbb{N}) [d(x,y) < \epsilon \wedge d(F^n(x),F^n(y)) \geq \delta] $$
\end{definition}

More intuitively, this definition states that the iterated map has the property that there exist points arbitrarily close to some point $x \in X$ that eventually separate away from $x$ by at least $\delta$. 

We will need, for our definitions of chaos, definitions of the following terms:

\begin{definition}
A dynamical system $(X,F)$ has a \emi{dense orbit} if and only if $$ (\exists x \in X) \, (\forall y \in X) \, (\forall \epsilon > 0) \, (\exists n \in \mathbb{N}) \, [d(F^n(x),y) < \epsilon] $$
\end{definition}

\begin{definition}
A dynamical system $(X,F)$ is \emi{topologically transitive} if and only if for all non-empty open subsets $U,V$ of $X$, $$ (\exists n \in \mathbb{N}) \, [F^n(U) \cap V \neq \emptyset]$$
\end{definition}

For a perfect DTDS $(X,F)$, the existence of a dense orbit necessarily implies topological transitivity. This is an important result in reference to the 1-dimensional dynamical systems we wish to represent in tilings later on - it shows us that the barrier to achieving chaotic behaviour is reassuringly low, which somewhat naturalises our results.

\begin{definition}
A dynamical system $(X,F)$ has \emi{dense periodic points} if and only if the set of all the periodic points given by $$ Per(F) = \{ x \in X : (\exists k \in \omega ) \, F^k(x) = x \} $$ is a dense subset of $X$. Specifically, $$ (\forall x \in X) \, (\forall \epsilon > 0) \, (\exists p \in Per(F)) \, [ d(x,p) < \epsilon ] $$
\end{definition}

Following on from these definitions, Devaney in \cite{Devaney1989} formulated the most well-known definition of chaos as follows:

\begin{definition}[Devaney Chaos]\index{Devaney chaos}
The dynamical system $(X,F)$ is \emph{chaotic} if
\begin{enumerate}
\item $F$ is topologically transitive,
\item $F$ has dense periodic points,
\item $F$ is sensitive to initial conditions.
\end{enumerate}
\end{definition}

Meanwhile, other formulations of chaos came about - the most notable for this work is due to Knudson \cite{Knudson1994}, which is nonperiodicity-free: 

\begin{definition}[Knudson Chaos]\index{Knudson chaos}
The dynamical system $(X,F)$ is chaotic if
\begin{enumerate}
\item $F$ has a dense orbit,
\item $F$ is sensitive to initial conditions.
\end{enumerate}
\end{definition}

This formulation that came about when Knudson proved there existed a dynamical system which is chaotic according to Devaney's definition, but which the restriction of the set to its periodic points was also Devaney Chaotic.

It will be useful later to consider similar restrictions, such as \cite{VellekoopBerglund} that demonstrates the following proposition:

\begin{proposition}[\protect{\cite[Prop. 1, p.353]{VellekoopBerglund}}]
Let $I$ be a (potentially infinite) interval - a 1-dimensional space - and $F: I \mapsto I$ be a continuous, topologically transitive map. Then
\begin{enumerate}
\item The periodic points of $F$ are dense in $I$,
\item $F$ has sensitivity to initial conditions.
\end{enumerate}
\end{proposition}

Thus, for 1-dimensional systems, topological transitivity is `enough' for a dynamical system to be chaotic. Given our ECAs are being considered as essentially 1-dimensional DTDS it becomes clear that our requirements for such a system to be chaotic are quite surprisingly minimal. 

In order to fully describe this, we need notions of `permutivity' for an ECA, which we get from \cite{CattaneoFM00}:

\begin{definition}[Permutivity]\index{permutivity}
A cellular automaton local rule $f$ is \emi{permutive} in $x_i$, for $-k \leq i \leq k$, if and only if for any given sequence $x_{-k}, \ldots , x_{i-1},x_{i+1},\ldots,x_k \in X$ we have $$ \{ f(x_{-k},\ldots,x_{i-1},x_i,x_{i+1},\ldots,x_k) : x_i \in X \} = X $$
\end{definition}

We can refine this idea to leftmost (rightmost) as follows:

\begin{definition}[Leftmost (Rightmost) Permutive]\index{left/rightmost permutive}
A local CA rule $f$ is said to be \emph{leftmost} (\emph{rightmost}) permutive if and only if there is an integer $i$, $-k \leq i \leq 0$ ($0 \leq i \leq k$) such that:
\begin{enumerate}
\item $i \neq 0$,
\item $f$ is permutive in the $i^{th}$ variable, 
\item $f$ does not depend on $x_j$ for $j < i$ ($j > i$).
\end{enumerate}
\end{definition}

As pointed out in \cite{CattaneoFM00}, for ECAs this means that when an ECA is leftmost-permutive, it follows that $$(\forall x_i, x_{i+1}) \, [f(0,x_i, x_{i+1}) \neq f(1, x_i, x_{i+1})] $$ namely, if two strings differ in the $x_{i-1}^{th}$ position, they differ in the $x_i^{th}$ position under $f$. Likewise, when an ECA is rightmost-permutive, the mirror argument follows, specifically $$(\forall x_{i-1}, x_{i}) \, [f(x_{i-1},x_i, 0) \neq f(x_{i-1}, x_i, 1)] $$

We can now use the following result from Cattaneo \etal (Cor. 3.3 in \cite{CattaneoFM00}):

\begin{corollary}[\protect{\cite[Cor. 3.2]{CattaneoFM00}}]
Let $( \mathbb{Z}^2, R_n ) $ be an ECA based on the local rule $r_n$. Then the following are equivalent:
\begin{enumerate}
\item $r_n$ is leftmost or rightmost permutive, or both.
\item $r_n$ is Devaney Chaotic.
\item $r_n$ is Knudson Chaotic
\item $r_n$ is surjective and non-trivial.
\end{enumerate}
\end{corollary}

By Table 1 and the analysis in Section 3.3 in \cite{CattaneoFM00}, it becomes clear that there exist a set of rules that exhibit chaotic behaviour, the most well known of which is $R_{30}$, having been studied in some depth originally by Wolfram in \cite{NKS}.

\subsection{ECAs and Turing Universality}

We now wish to extend results from earlier in this thesis to very small dynamical systems, for which we will need the following definitions:

\begin{definition}
A \emi{cyclic tag system} is a computational system consisting of the following arrangement:
\begin{itemize}
\item A set $P \subset 2^{< \omega}$ of \emph{productions}.
\item A finite binary string $d = d_0, d_1, \ldots d_j$ called the \emph{data string}.
\item A transformation map $$ (i, d) \rightarrow (i+1 (\mathrm{mod}\ n), (d_1, d_2, \ldots, d_k)^\frown P^{d_0}_i) $$ where $i$ is a counter, $n = |P|$, and for all $i$:
\begin{align*}
P^0_i & = \emptyset \\
P^1_i & = P_i
\end{align*}
\end{itemize}
\end{definition}

Intuitively, a cyclic tag system operates as follows:
\begin{enumerate}
\item If $d_0 = 0$, then we delete $d_0$ and do nothing.
\item If $d_0 = 1$, then we delete $d_0$ and append the $i^{th}$ member of $P$, $P_i$.
\item If $d = \emptyset$ then we halt. 
\end{enumerate}

An example computation is as follows. Let $P = \{101, 110, 10\}$ and $d = 11$, our computation is as given in table 2.
\begin{table}[]\label{tbl:cyclicTagSystem}
\begin{tabular}{ll}
$P_i$  & $d$ \\
101  & 11 \\
110  & 1101 \\
11   & 101110 \\
101  & 0111011 \\
110  & 111011 \\
11   & 11011110 \\
101  & $\ldots$
\end{tabular}
\caption{This table shows the development of a cyclic tag system for initial $d$ of $11$ and $P_i$'s in sequence as given in the text. The development of the contents of $d$ is given at each line.}
\end{table}

It is proved in \cite{CookUniversality} that a cyclic tag system is Turing Universal - this was done by showing a Universal Turing Machine can be coded into a 2-tag system, and 2-tag systems can be coded into Cyclic tag systems. The proof is omitted here, but a clear proof can be found in \cite{Neary2006}. 

In 2004, Cook proved in \cite{CookUniversality} the following theorem:

\begin{theorem}[\protect{\cite[Sec 4]{CookUniversality}}]\label{thm:CookECA}
The ECA $R_{110}$ is Turing Universal.
\end{theorem}

This is done by combination of the following theorem and Lemmas:

\begin{lemma}[\protect{\cite[Sec 3]{CookUniversality}}]
A cyclic tag system is Turing Complete.
\end{lemma}

This is a somewhat surprising result, owing to the very minimal nature of cyclic tag systems, but the proof shows that by careful construction of the production sets $P$ it is possible to emulate the tag systems, due to Post, of a small number of states easily. The proof of this coding is fairly straightforward, but is omitted here owing to length.

\begin{lemma}[\protect{\cite[Sec 4]{CookUniversality}}]
A Cyclic Tag system can be implemented in a glider system.
\end{lemma}

\begin{proof}[Sketch of proof of \ref{thm:CookECA}]
Rule 110 has the ability to carry a state of 1's and 0's left and right depending on careful setup of the strings - such patterns that shift iteratively left and right down our ECA state are called `gliders'. A `glider system' is some arrangement of these gliders such that they then propagate left and right. There are 5 glider types documented in \cite{CookUniversality}, and these are crafted into different arrangements of glider systems in order to achieve the result we are interested - specifically, coding the $P$ and $d$ of any cyclic tag system.

\begin{figure}[t]
\includegraphics[scale=0.6]{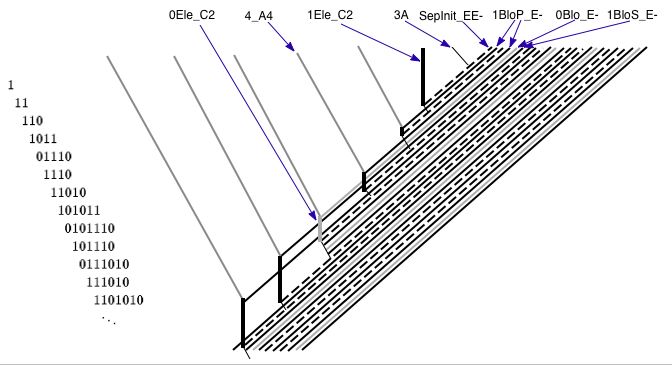}
\label{fig:CookSchema}
\caption{The schematic diagram for Cook's encoding of Cyclic Tag Systems in Rule 110, taken from \cite{GenaroUWE}}
\end{figure}

By carefully implementing a glider system in the input row for an ECA, Cook was able to code Turing Machine computations into the dynamics of $R_{110}$, thereby showing this ECA to be Turing Universal. 

An additional aside, which will be useful in our discussion of ECA tilings, is that the halting state of some TM coded into $R_{110}$ is equivalent to whether the dynamics of the system become aperiodic or remain periodic, equivalent to halting or not halting, respectively.
\end{proof}

An overall schematic diagram can be found in Figure \ref{fig:CookSchema} 

However, we note that there are some cases where simply expecting aperiodicity or continued periodicity is not sufficient. Take a TM that calculates some non-repeating sequence, such as the Champernowne's Constants used earlier in this thesis. The output of this computation will necessarily be aperiodic in any given tiling encoding of this computation.

Thus we have to resolve the issue surrounding this - if our tiling is going to be aperiodic whether we have halted or not, then how can we tell if our computation is running or if it has halted?

Firstly, we note that Rule 110 is not left or right permutive, so any tiling will not naturally be aperiodic by the criteria in the previous section. We next need to note that we can stratify these two notions of `aperiodicity' by means of a straightforward argument on the underlying mechanics of our resultant tilings \emph{in vicem} of the Turing Machines and cyclic tag machines we are representing.

We note that any non-repeating computation will actually be quasi-periodic by our definition \ref{def:quasiperiodic} - a fact that follows when we observe that certain strings, namely those representing states in our Turing Machine via the set of productions $P$ in our cyclic tag system being recurrent in the tiling.

Thus, any aperiodic behaviour will be apparent from the fact that there will be no sign of our Turing computational artefacts in the ECA following entering the halt state. As such, it will either become periodic or aperiodic, but our test for the occurrence of particular words that code these will fail.

The same carries forwards into our tiling by means of looking for particular sequences of tiles - represented as finite tuples - in any resultant tiling. Given this, we can safely work with ECA Rule 110 and not worry about `losing track' of the status of our computation.

\section{Elementary Cellular Automata and Tilings}

In this section we build on work from the author's MSc thesis, \cite{MCarneyMSc}, where we proved the following theorem:

\begin{theorem}[\protect{\cite[Chap. 3]{MCarneyMSc}}]\label{thm:WangECA}
There exists a universal prototile schema consisting of 18 Wang tiles that tiles the plane according to the rules of any given ECA.
\end{theorem}

\begin{proof}

We note that we need to satisfy the following requirements:
\begin{enumerate}
\item Encode each cell in a time-space diagram for a given ECA.
\item Encode the relationships between each cell given by $R_n$.
\item Show how bits can be copied across each other in the tiling in order to emulate the action of $R_n$.
\end{enumerate}

We first construct the prototile scheme that will code the action of our ECA function given by $f_n : \{0,1\}^3 \rightarrow \{0,1\}$, given by our rule $R_n$. This scheme is as follows:

\begin{center}
\sampletilefarlabels{$a$}{$b$}{$c$}{$\substack{}{f_n(a,b,c)}$}
\end{center}

Thus, for each rule we get the following 8 prototiles, where we fill in the specific outputs for each $f_n$ to get our \emph{Rule prototiles}:

\begin{center}
\sampletilefarlabels{$0$}{$0$}{$0$}{$\substack{}{f_n(0,0,0)}$}
\sampletilefarlabels{$0$}{$0$}{$1$}{$\substack{}{f_n(0,0,1)}$}
\sampletilefarlabels{$0$}{$1$}{$0$}{$\substack{}{f_n(0,1,0)}$}
\sampletilefarlabels{$0$}{$1$}{$1$}{$\substack{}{f_n(0,1,1)}$}

\vspace{2.5mm}

\sampletilefarlabels{$1$}{$0$}{$0$}{$\substack{}{f_n(1,0,0)}$}
\sampletilefarlabels{$1$}{$0$}{$1$}{$\substack{}{f_n(1,0,1)}$}
\sampletilefarlabels{$1$}{$1$}{$0$}{$\substack{}{f_n(1,1,0)}$}
\sampletilefarlabels{$1$}{$1$}{$1$}{$\substack{}{f_n(1,1,1)}$}

\end{center}

We add to these \emph{state swapping} tiles that will take an output of $f_n$ and `swap' this bit with the cell's neighbours. We first fix the colour $B$ that will act as `blank', allowing us to line up the tiles above and below each crossover of bits from the distributor tiles (see below):

\begin{center}
\sampletile{$0$}{$B$}{$0$}{$(0,0)$} \sampletile{$0$}{$(0,0)$}{$0$}{$B$}

\vspace{2.5mm}

\sampletile{$0$}{$B$}{$1$}{$(0,1)$} \sampletile{$1$}{$(0,1)$}{$0$}{$B$}

\vspace{2.5mm}

\sampletile{$1$}{$B$}{$0$}{$(1,0)$} \sampletile{$0$}{$(1,0)$}{$1$}{$B$}

\vspace{2.5mm}

\sampletile{$1$}{$B$}{$1$}{$(1,1)$} \sampletile{$1$}{$(1,1)$}{$1$}{$B$}
\end{center}

We now need some \emph{distributor tiles} that will take an output state and distribute this information left, right, and downwards:

\begin{center}
\sampletile{$1$}{$1^{f_n}$}{$1$}{$1$}
\sampletile{$0$}{$0^{f_n}$}{$0$}{$0$}
\end{center}

Note that these tiles differentiate the upper quadrant as being specifically from the output of $f_n$ so as to prevent trivial tilings of the plane using just distributor prototiles. These tiles code exactly the cells from the original time-space diagram.

We then note that each part of the action of some ECA rule $R_n$ is now coded into our tiling:
\begin{itemize}
\item Each cell is represented in any planar tiling due to the above prototile constructions.
\item Each relationship coded by $f_n$ is represented as state swapping tiles creating a space for some rule tile, which then has the output of $f_n$ distributed for this process to repeat.
\item We do not code the upper half-plane owing to our not-knowing the previous rows of computation that took place before our input row.
\end{itemize}

Thus, we have fully represented in 18 prototiles, given by our 8 rule tiles, 8 state swapping, and 2 distributor prototiles.

The tiling process is as follows:
\begin{enumerate}
\item Code the input into a series of distributor tiles.
	\begin{itemize}
	\item We pad the input with infinitely many `0's left and right to achieve a full half-planar tiling.
	\end{itemize}
\item Place the relevant state swapping tiles between each of these.
\item Tile each successive row using the correct tilings, in order to get successive states of the ECA.
\end{enumerate}

\end{proof}

\begin{figure}[t]
\centering
\begin{tikzpicture}[scale=1.5]
\foreach \x/\y/\l/\u/\r/\b/\z in {0/0/white/white/white/white/,1/0/white/gray/white/green/,2/0/white/white/white/white/,3/0/white/gray/black/yellow/,4/0/black/black/black/black/,5/0/black/gray/white/red/,6/0/white/white/white/white/,7/0/white/gray/white/green/,8/0/white/white/white/white/,
                               0/1/white/white/white/white/,1/1/white/green/white/gray/,2/1/white/white/black/black/,3/1/black/yellow/white/gray/,4/1/white/black/white/black/,5/1/white/red/black/gray/,6/1/black/white/white/black/,7/1/white/green/white/gray/,8/1/white/white/white/white/,
                               0/2/white/white/white/white/,1/2/white/gray/black/yellow/,2/2/black/black/black/black/,3/2/black/gray/black/blue/,4/2/black/black/black/black/,5/2/black/gray/black/blue/,6/2/black/black/black/black/,7/2/black/gray/white/red/,8/2/white/white/white/white/,
                               0/3/white/white/black/black/,1/3/black/yellow/white/gray/,2/3/white/black/black/black/,3/3/black/blue/black/gray/,4/3/black/black/black/white/,5/3/black/blue/black/gray/,6/3/black/black/white/white/,7/3/white/red/black/gray/,8/3/black/white/white/black/,
                               0/4/black/black/black/black/,1/4/black/gray/black/blue/,2/4/black/black/black/black/,3/4/black/gray/white/red/,4/4/white/white/white/white/,5/4/white/gray/white/green/,6/4/white/white/white/white/,7/4/white/gray/black/yellow/,8/4/black/black/black/black/,
                               0/5/white/black/black/black/,1/5/black/blue/black/gray/,2/5/black/black/white/white/,3/5/white/red/black/gray/,4/5/black/white/white/black/,5/5/white/green/white/gray/,6/5/white/white/black/black/,7/5/black/yellow/white/gray/,8/5/white/black/white/black/}
{
\draw[fill=\r] (1+\x,1-\y) rectangle (0+\x,0-\y);
\filldraw[fill=\l] (0+\x,0-\y) -- (0.5+\x,0.5-\y) -- (0+\x,1-\y) -- cycle;
\filldraw[fill=\b] (0+\x,0-\y) -- (0.5+\x,0.5-\y) -- (1+\x,0-\y) -- cycle;
\filldraw[fill=\u] (0+\x,1-\y) -- (0.5+\x,0.5-\y) -- (1+\x,1-\y) -- cycle;
\node at (0.5+\x,0.5-\y) [label=below:$\z$] {};
}
\end{tikzpicture}
  \caption{A sample tiling of $S_{30}$. {\em NB:} Indicators $O^f$ and $1^f$ are omitted for clarity.}
  \label{fig:ECA-Wang}
\end{figure}
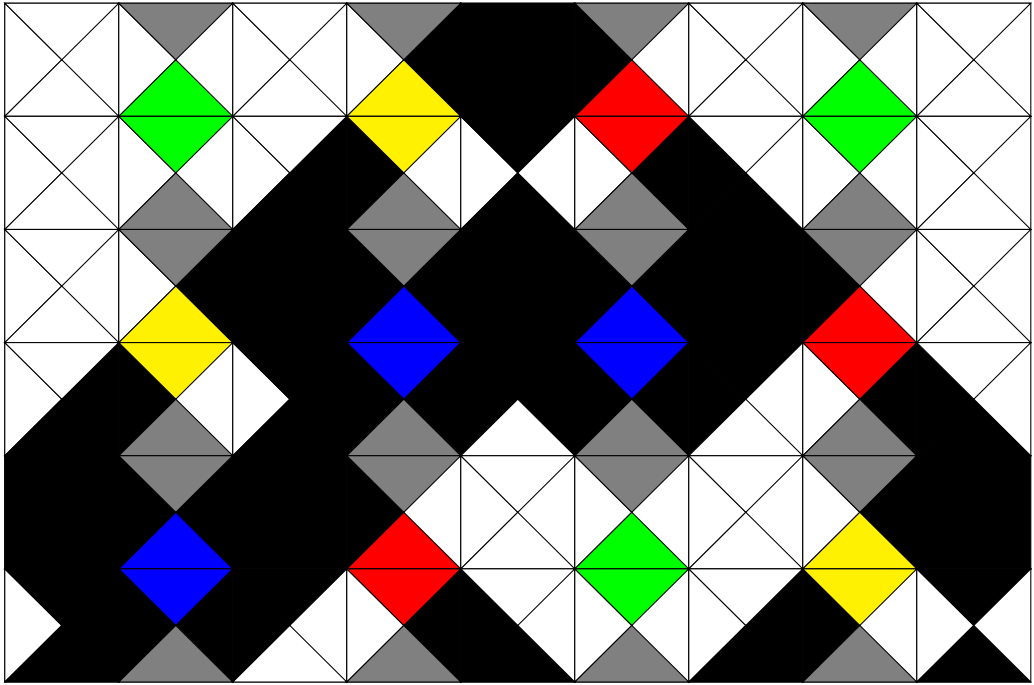

Figure \ref{fig:ECA-Wang} shows the tiling in action, coding the first few rows of ECA rule 30, with $R_{30}$ clearly coded with the connecting tiles showing how the outputs interact with each other.

\section{A 15 Prototile ECA Tiling}

We present a tiling that codes any ECA in only 15 tiles, using an adapted hexagon-based tiling. This particular tiling lends itself to our computable trinary functions that form our $f_n$ ECA functions, and have not yet been found in the literature.

\begin{theorem}[C. 2019]\label{thm:ECA-15-Hex}
For any ECA of Rule $n$ there exists a prototile set $S_n$ of size 15 such that any tiling of the plane $T$ by $S_n$ codes each iteration of the ECA starting from the string coded by the first row. 
\end{theorem}

\begin{proof}

Broadly speaking, we require three things from our tiling of ECA rules - for a given rule $R_n$:
\begin{enumerate}
\item Encoding of each input and output of the $f_n$ for our rule $R_n$.
\item Handling of the `transfer of bits' from one represented cell to the cells lower left, lower centre, and lower right.
\item Fixing of upper half-plane boundary.
\end{enumerate}

For the purposes of this proof, we work on tiling the lower half-plane, with the lower border of the upper half-plane having colour $I$. This means that we do not have to worry about the pre-images of the inverse function $f^{-1}$ which can not be unique or even be a `Garden of Eden', meaning it is a configuration that has no pre-image. Thus simplifying the way in which we tile the plane by omitting these in the upper half plane, essentially fixing it with colour $I$.

We first present the base tiling we are going to use - horizontally aligned hexagons with diamond lozenges filling the gaps between them, as so:

\begin{center}
\begin{tikzpicture}
\begin{scope}[xshift=0cm]\fillshapesimple{(0:1) -- (60:1) -- (120:1) -- (180:1) -- (240:1) -- (300:1) -- cycle}\end{scope}
\begin{scope}[xshift=2cm]\fillshapesimple{(0:1) -- (60:1) -- (120:1) -- (180:1) -- (240:1) -- (300:1) -- cycle} \end{scope}
\begin{scope}[xshift=4cm]\fillshapesimple{(0:1) -- (60:1) -- (120:1) -- (180:1) -- (240:1) -- (300:1) -- cycle} \end{scope}
\begin{scope}[xshift=6cm]\fillshapesimple{(0:1) -- (60:1) -- (120:1) -- (180:1) -- (240:1) -- (300:1) -- cycle} \end{scope}
\begin{scope}[xshift=0cm,yshift=-1.75cm]\fillshapesimple{(0:1) -- (60:1) -- (120:1) -- (180:1) -- (240:1) -- (300:1) -- cycle} \end{scope}
\begin{scope}[xshift=2cm,yshift=-1.75cm]\fillshapesimple{(0:1) -- (60:1) -- (120:1) -- (180:1) -- (240:1) -- (300:1) -- cycle} \end{scope}
\begin{scope}[xshift=4cm,yshift=-1.75cm]\fillshapesimple{(0:1) -- (60:1) -- (120:1) -- (180:1) -- (240:1) -- (300:1) -- cycle} \end{scope}
\begin{scope}[xshift=6cm,yshift=-1.75cm]\fillshapesimple{(0:1) -- (60:1) -- (120:1) -- (180:1) -- (240:1) -- (300:1) -- cycle} \end{scope}
\end{tikzpicture}
\end{center}

We present two tile schemas that we will make use of can be carried out to obtain a tile set $S_n$ for each ECA Rule $R_n$.

Firstly, we give a schema for the hexagon tiles that will code the actual rule action. For $a,b,c, f_n(a,b,c) \in \{0,1\}$ we define our tile schema:

\begin{center}
\begin{tikzpicture} 
\begin{scope}[xshift=0cm]\fillshapesimple{(0:0) -- (180:2) -- (120:2) -- (0:0) -- (120:2) -- (60:2) -- (0:0) -- (60:2) -- (0:2) -- (0:0) -- (0:2) -- (300:2) -- (240:2) -- (180:2) -- cycle}
\draw (150:1) node {$a$};
\draw (90:1) node {$b$};
\draw (30:1) node {$c$};
\draw (270:1) node {$f_n(a,b,c)$};
\end{scope}
\end{tikzpicture}
\end{center}
where $f_n$ is the operation of applying rule $n$ to the three input bits $a,b,c$. Note, if required we can use similar notation to the 4-tuple codes we used for Wang tiles - specifically: $\langle a,b,c,f_n(a,b,c) \rangle$

We can see that for $a,b,c \in \{0,1\}$ there are 8 prototiles that we can define as our basis for each ECA tiling. These are as follows:

\begin{center}
\begin{tikzpicture}[scale=0.6]
\hexagontile{1}{1}{1}{f_n(1,1,1)}
\begin{scope}[xshift=4.5cm]\hexagontile{1}{1}{0}{f_n(1,1,0)}\end{scope}
\begin{scope}[xshift=9cm]\hexagontile{1}{0}{1}{f_n(1,0,1)}\end{scope}
\begin{scope}[xshift=13.5cm]\hexagontile{1}{0}{0}{f_n(1,0,0)}\end{scope}
\begin{scope}[xshift=0cm, yshift=-4cm]\hexagontile{0}{1}{1}{f_n(0,1,1)}\end{scope}
\begin{scope}[xshift=4.5cm, yshift=-4cm]\hexagontile{0}{1}{0}{f_n(0,1,0)}\end{scope}
\begin{scope}[xshift=9cm, yshift=-4cm]\hexagontile{0}{0}{1}{f_n(0,0,1)}\end{scope}
\begin{scope}[xshift=13.5cm, yshift=-4cm]\hexagontile{0}{0}{0}{f_n(0,0,0)}\end{scope}
\end{tikzpicture}
\end{center}

We next define our diamond lozenges as being tiles that are vertically and horizontally quadrisected and use the following tile schema, for $s, t \in \{0,1  \}$:

\begin{center}
\begin{tikzpicture} 
\begin{scope}[xshift=0cm]\fillshapesimple{(0:0) -- (60:2) -- (0:2) -- (300:2) -- (60:2) -- (0:2) -- cycle -- (300:2)}
\draw (30:0.66) node {$s$};
\draw (15:1.33) node {$t$};
\draw (330:0.66) node {$t$};
\draw (345:1.33) node {$s$};
\end{scope}
\end{tikzpicture}
\end{center}

This gives us our 4 connecting lozenges as follows: 

\begin{center}
\begin{tikzpicture}
\lozengetile{(0:0) -- (60:2) -- (0:2) -- (300:2) -- (60:2) -- (0:2) -- cycle -- (300:2)}{0}{0}{0}{0}
\begin{scope}[xshift=2.5cm]\lozengetile{(0:0) -- (60:2) -- (0:2) -- (300:2) -- (60:2) -- (0:2) -- cycle -- (300:2)}{0}{1}{0}{1}\end{scope}
\begin{scope}[xshift=5cm]\lozengetile{(0:0) -- (60:2) -- (0:2) -- (300:2) -- (60:2) -- (0:2) -- cycle -- (300:2)}{1}{0}{1}{0}\end{scope}
\begin{scope}[xshift=7.5cm]\lozengetile{(0:0) -- (60:2) -- (0:2) -- (300:2) -- (60:2) -- (0:2) -- cycle -- (300:2)}{1}{1}{1}{1}\end{scope}
\end{tikzpicture}
\end{center}

These connecting lozenges are required owing to a property of ECAs - namely, for some string $\sigma \in \{0,1\}^{<\omega}$, any bit $b_i \in \sigma$ is needed to calculate the bits $b'_{i-1},b'_i, b'_{i+1} \in \sigma'$. As such, these lozenges achieve the required `crossover' of these bits. These act in principle precisely the same as the `state swapping tile' in our previous theorem \ref{thm:WangECA}.

We will also need the following 3 `I' tiles to make our tiling `neat' and to define the first row of out tiling:

\begin{center}
\begin{tikzpicture}[scale=0.6]
\hexagontile{I}{I}{I}{0}
\begin{scope}[xshift=4.5cm]\hexagontile{I}{I}{I}{1}\end{scope}
\begin{scope}[xshift=7.5cm]\lozengetile{(0:0) -- (0:2) -- (300:2) -- cycle -- (300:2) -- (0:1)}{}{}{I}{I}\end{scope}
\end{tikzpicture}
\end{center}

This will give us a flat edge for the top of the tiling, where we can now see that a tiling of the plane, with no gaps can be achieved, as shown in this diagram:

We can thus define the tiling algorithm for some ECA as follows:
\begin{enumerate}
\item Take the input for our ECA and code this using the `I' tiles.
\begin{itemize}
\item Pad the input with $\langle I, I, I, 0\rangle$ tiles as needed left and right to fill the left and right halves of our lower half-plane.
\item Ensure that the half-lozenge `I'-tiles are placed between the upper gaps between these hexagons.
\end{itemize}
\item Place the correct corresponding lozenge tiles between the hexagon tiles.
\item Place the now-defined hexagon tiles under each hexagon such that the upper 3 sides correspond to the lozenges on the upper left and upper right, and the hexagon immediately above.
\item Go to 2.
\end{enumerate}

Given this algorithm and this tile set, we can code any ECA by choosing the prescribed outputs from $f_n(x,y,z)$ from our rule $n$. Given this setup, we can see that our tiling gives a tiling of the half-plane without any holes, and such that it imitates the behaviour of any ECA.

\end{proof}

As an illustrated example, the full prototile set for Rule 30 can be found in figure \ref{fig:15TileRule30}

\begin{figure}[t]
\begin{center}
\begin{tikzpicture}[scale=0.6]
\hexagontile{1}{1}{1}{0}
\begin{scope}[xshift=4.5cm]\hexagontile{1}{1}{0}{0}\end{scope}
\begin{scope}[xshift=9cm]\hexagontile{1}{0}{1}{0}\end{scope}
\begin{scope}[xshift=13.5cm]\hexagontile{1}{0}{0}{1}\end{scope}
\begin{scope}[xshift=0cm, yshift=-4cm]\hexagontile{0}{1}{1}{1}\end{scope}
\begin{scope}[xshift=4.5cm, yshift=-4cm]\hexagontile{0}{1}{0}{1}\end{scope}
\begin{scope}[xshift=9cm, yshift=-4cm]\hexagontile{0}{0}{1}{1}\end{scope}
\begin{scope}[xshift=13.5cm, yshift=-4cm]\hexagontile{0}{0}{0}{0}\end{scope}
\begin{scope}[xshift=0cm,yshift=-8cm]\lozengetile{(0:0) -- (60:2) -- (0:2) -- (300:2) -- (60:2) -- (0:2) -- cycle -- (300:2)}{0}{0}{0}{0}\end{scope}
\begin{scope}[xshift=2.5cm,yshift=-8cm]\lozengetile{(0:0) -- (60:2) -- (0:2) -- (300:2) -- (60:2) -- (0:2) -- cycle -- (300:2)}{0}{1}{0}{1}\end{scope}
\begin{scope}[xshift=5cm,yshift=-8cm]\lozengetile{(0:0) -- (60:2) -- (0:2) -- (300:2) -- (60:2) -- (0:2) -- cycle -- (300:2)}{1}{0}{1}{0}\end{scope}
\begin{scope}[xshift=7.5cm,yshift=-8cm]\lozengetile{(0:0) -- (60:2) -- (0:2) -- (300:2) -- (60:2) -- (0:2) -- cycle -- (300:2)}{1}{1}{1}{1}\end{scope}
\begin{scope}[yshift=-12cm]\hexagontile{I}{I}{I}{0}\end{scope}
\begin{scope}[xshift=4.5cm, yshift=-12cm]\hexagontile{I}{I}{I}{1}\end{scope}
\begin{scope}[xshift=7.5cm,yshift=-12cm]\lozengetile{(0:0) -- (0:2) -- (300:2) -- cycle -- (300:2) -- (0:1)}{}{}{I}{I}\end{scope}
\end{tikzpicture}
\end{center}
\caption{A 15 prototile set of tiles that encodes the behaviour of the Rule 30 ECA in the lower half-plane.}
\label{fig:15TileRule30}
\end{figure}
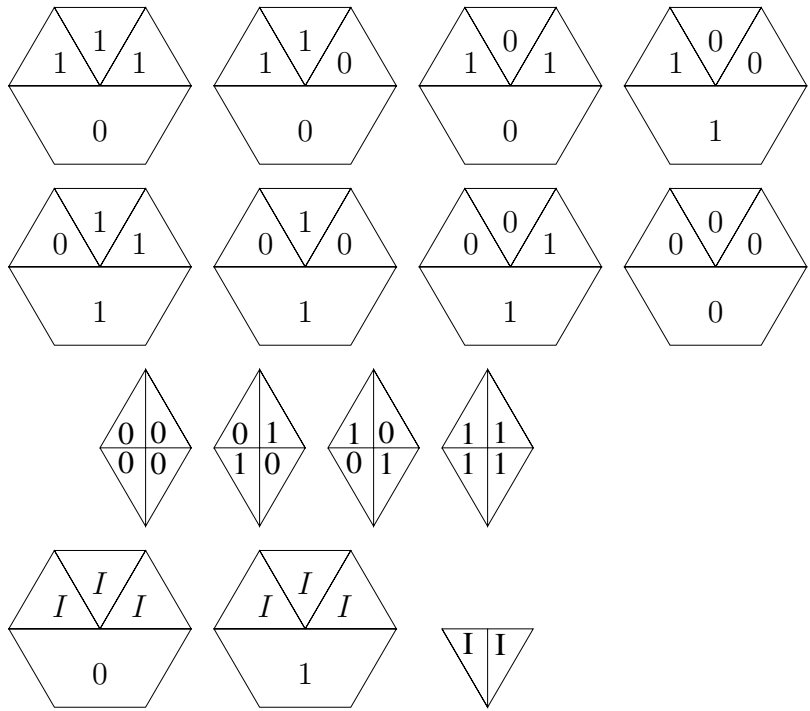

Our proof of this theorem is unusual as it makes use of a non-standard planar tiling made up of hexagon and lozenge tiles - something that the author has not seen at all in the literature. This particular prototile arrangement lends itself to 3-ary iterated functions and dynamical systems slightly better than Wang tiles. Hence, they are included in this thesis as objects for potential further consideration.

\begin{corollary}[C. 2019]
There are chaotic ECA prototile sets of size 15.
\end{corollary}

\begin{proof}
This is immediate from the known properties of Rule 30, 90, etc. given in \cite{CattaneoFM00} - specifically, we can simply code these ECAs into prototiles using the scheme above and obtain a fixed-size prototile set that can code the required behaviour on a given input, such an input being given by an initial row of `I'-tiles from our original construction.
\end{proof}

\begin{corollary}[C. 2019]
There are Turing Complete prototile sets of size 15.
\end{corollary}

\begin{proof}
This corollary is immediate from the Turing completeness of Rule 110 \cite{CookUniversality} and the theorem \ref{thm:ECA-15-Hex} by the same argument given for 1. We note that we have to perform the following steps to obtain the result. Given a Turing Machine with index $e$ and a given input $x$:

\begin{enumerate}
\item Convert $\varphi_e$ to a cyclic tag system, to get $Tag_e$.
\item For $\varphi_e(x)$ we take $Tag_e$ and code this and $x$ into a single row input for our ECA.
\item Code this into the initial row `I'-tiles from our construction.
\end{enumerate}

With this done, we can allow our tiling to proceed row by row, and note that this codes each successive stage of the computation $\varphi_e(x)$ via the mapping above.
\end{proof}

\begin{figure}[b]
\begin{center}
\begin{tikzpicture}[scale=0.6]
\begin{scope}[xshift=-3cm,yshift=1.725cm]\lozengetile{(0:0) -- (0:2) -- (300:2) -- cycle -- (300:2) -- (0:1)}{}{}{$I$}{$I$}\end{scope}
\begin{scope}[xshift=1cm,yshift=1.725cm]\lozengetile{(0:0) -- (0:2) -- (300:2) -- cycle -- (300:2) -- (0:1)}{}{}{$I$}{$I$}\end{scope}
\begin{scope}[xshift=5cm,yshift=1.725cm]\lozengetile{(0:0) -- (0:2) -- (300:2) -- cycle -- (300:2) -- (0:1)}{}{}{$I$}{$I$}\end{scope}
\begin{scope}[xshift=9cm,yshift=1.725cm]\lozengetile{(0:0) -- (0:2) -- (300:2) -- cycle -- (300:2) -- (0:1)}{}{}{$I$}{$I$}\end{scope}
\begin{scope}[xshift=13cm,yshift=1.725cm]\lozengetile{(0:0) -- (0:2) -- (300:2) -- cycle -- (300:2) -- (0:1)}{}{}{$I$}{$I$}\end{scope}
\begin{scope}[xshift=17cm,yshift=1.725cm]\lozengetile{(0:0) -- (0:2) -- (300:2) -- cycle -- (300:2) -- (0:1)}{}{}{$I$}{$I$}\end{scope}

\begin{scope}[xshift=0cm]\hexagontile{I}{I}{I}{0}\end{scope}
\begin{scope}[xshift=4cm]\hexagontile{I}{I}{I}{0}\end{scope}
\begin{scope}[xshift=8cm]\hexagontile{I}{I}{I}{1}\end{scope}
\begin{scope}[xshift=12cm]\hexagontile{I}{I}{I}{0}\end{scope}
\begin{scope}[xshift=16cm]\hexagontile{I}{I}{I}{0}\end{scope}
\begin{scope}[xshift=1cm,yshift=-1.75cm]\lozengetile{(0:0) -- (60:2) -- (0:2) -- (300:2) -- (60:2) -- (0:2) -- cycle -- (300:2)}{0}{0}{0}{0}\end{scope}
\begin{scope}[xshift=5cm,yshift=-1.75cm]\lozengetile{(0:0) -- (60:2) -- (0:2) -- (300:2) -- (60:2) -- (0:2) -- cycle -- (300:2)}{0}{1}{0}{1}\end{scope}
\begin{scope}[xshift=9cm,yshift=-1.75cm]\lozengetile{(0:0) -- (60:2) -- (0:2) -- (300:2) -- (60:2) -- (0:2) -- cycle -- (300:2)}{1}{0}{1}{0}\end{scope}
\begin{scope}[xshift=13cm,yshift=-1.75cm]\lozengetile{(0:0) -- (60:2) -- (0:2) -- (300:2) -- (60:2) -- (0:2) -- cycle -- (300:2)}{0}{0}{0}{0}\end{scope}

\begin{scope}[xshift=0cm,yshift=-3.5cm]\hexagontile{0}{0}{0}{0}\end{scope}
\begin{scope}[xshift=4cm,yshift=-3.5cm]\hexagontile{0}{0}{1}{1}\end{scope}
\begin{scope}[xshift=8cm,yshift=-3.5cm]\hexagontile{0}{1}{0}{1}\end{scope}
\begin{scope}[xshift=12cm,yshift=-3.5cm]\hexagontile{1}{0}{0}{1}\end{scope}
\begin{scope}[xshift=16cm,yshift=-3.5cm]\hexagontile{0}{0}{0}{0}\end{scope}

\begin{scope}[xshift=1cm,yshift=-5.25cm]\lozengetile{(0:0) -- (60:2) -- (0:2) -- (300:2) -- (60:2) -- (0:2) -- cycle -- (300:2)}{0}{1}{0}{1}\end{scope}
\begin{scope}[xshift=5cm,yshift=-5.25cm]\lozengetile{(0:0) -- (60:2) -- (0:2) -- (300:2) -- (60:2) -- (0:2) -- cycle -- (300:2)}{1}{1}{1}{1}\end{scope}
\begin{scope}[xshift=9cm,yshift=-5.25cm]\lozengetile{(0:0) -- (60:2) -- (0:2) -- (300:2) -- (60:2) -- (0:2) -- cycle -- (300:2)}{1}{1}{1}{1}\end{scope}
\begin{scope}[xshift=13cm,yshift=-5.25cm]\lozengetile{(0:0) -- (60:2) -- (0:2) -- (300:2) -- (60:2) -- (0:2) -- cycle -- (300:2)}{1}{0}{1}{0}\end{scope}

\begin{scope}[xshift=0cm,yshift=-7cm]\hexagontile{0}{0}{1}{1}\end{scope}
\begin{scope}[xshift=4cm,yshift=-7cm]\hexagontile{0}{1}{1}{1}\end{scope}
\begin{scope}[xshift=8cm,yshift=-7cm]\hexagontile{1}{1}{1}{0}\end{scope}
\begin{scope}[xshift=12cm,yshift=-7cm]\hexagontile{1}{1}{0}{0}\end{scope}
\begin{scope}[xshift=16cm,yshift=-7cm]\hexagontile{1}{0}{0}{1}\end{scope}

\begin{scope}[xshift=1cm,yshift=-8.75cm]\lozengetile{(0:0) -- (60:2) -- (0:2) -- (300:2) -- (60:2) -- (0:2) -- cycle -- (300:2)}{1}{1}{1}{1}\end{scope}
\begin{scope}[xshift=5cm,yshift=-8.75cm]\lozengetile{(0:0) -- (60:2) -- (0:2) -- (300:2) -- (60:2) -- (0:2) -- cycle -- (300:2)}{1}{0}{1}{0}\end{scope}
\begin{scope}[xshift=9cm,yshift=-8.75cm]\lozengetile{(0:0) -- (60:2) -- (0:2) -- (300:2) -- (60:2) -- (0:2) -- cycle -- (300:2)}{0}{0}{0}{0}\end{scope}
\begin{scope}[xshift=13cm,yshift=-8.75cm]\lozengetile{(0:0) -- (60:2) -- (0:2) -- (300:2) -- (60:2) -- (0:2) -- cycle -- (300:2)}{0}{1}{0}{1}\end{scope}

\end{tikzpicture}
\end{center}
\label{fig:ExHexLozTiling}
\caption{Example few rows of a hexagon and lozenge tiling of Rule 30.}
\end{figure}

We include in figure \ref{fig:ExHexLozTiling} as a worked example of the initial few stages and columns of a Rule 30 ECA Hexagon and Lozenge tiling, demonstrating the function of the initializer tiles, the ECA hexagons, and the connecting lozenge tiles to demonstrate how an ECA can be encoded into a tiling of the plane.

\begin{conjecture}[C. 2019]
There exist ECA prototile sets of 8 tiles.
\end{conjecture}

By \cite{Rao2015} these cannot be formed from Wang tiles - this would mean that there is an aperiodic prototile set of fewer than 8 tiles, which they proved to not be the case. As such, a tiling of 8 tiles must be some other planar repeating tessellation with colours applied to different edges or areas in order to represent a prototile set of 8 tiles.

\chapter{Conclusion}
\label{chap7}
\setcounter{equation}{0}
\renewcommand{\theequation}{\thechapter.\arabic{equation}}


\epigraph{Nevertheless, I repeat; we are only at the beginning. I am only a beginner. I was successful in digging up buried monuments from the substrata of the mind. But where I have discovered a few temples, others may discover a continent.}{\textit{S. Freud, \\ in an interview with G. S. Viereck.}}

Here we give an overview of the conclusions from the work presented in this thesis, and give summary of some of the open questions arising from this research.

\section{Conclusions from Results}

In Chapter 3 we presented our first results concerning the relationship between computability and tiling problems. We extended results due to Harel in \cite{Harel1986} to the general Domino Problem for infinite prototile sets. These results follow the general intuition due to Berger in \cite{berger1966} that the Domino Problem for finite prototile sets is $\Sigma^0_1$/$\Pi^0_1$ complete, so expecting that $TILE/\neg TILE$ is equivalent to $\Sigma^1_1$/$\Pi^1_1$ does fit the general intuition regarding this class of tiling problems.

We next discussed, in chapter 4 the question of whether tilings from a given prototile set are periodic or aperiodic. From this outset we found a rather unusual set for which the problems of (a)periodicity for infinite prototile sets are complete - $(\Pi^1_1 \wedge \Sigma^1_1)$ - which is a rare class of problems. Indeed, it is entirely possible that this may be weakened in subsequent work to one side of this conjunction.

The fact that $ATile_{FIN} \in \Pi^0_1$ is surprising, given we did not even have a proven existence of such prototile sets until the mid-60's. However, we state the conjecture (below) that $PTile_{FIN}$ is unlikely to be arithmetical owing to the requirement to quantify over all possible tilings for a given prototile set, despite its bound on lengths of their possible periodicity vectors.

The Weihrauch reductions presented in Chapter 5 are the first that we know of concerning tiling problems. They directly use material from previous chapters in order to show that the Domino Problems we have defined and studied are all bounded above by the closed choice principle for Baire Space, with some equivalences also being found. These give further detail to our picture of the computability aspects of Domino Problems, fleshing out the overall picture beyond the conventional view.

Finally, our results in Chapter 6 paint a picture regarding how to code tilings of 3-ary functions, using ECAs as our example. This is, sometimes, a more natural formulation of a problem, and as such the presentation of this hexagon-lozenge tiling may be useful outside of this particular class of automaton coding into prototile sets.

\section{Open Problems and Further Work}

There remain some interesting open problems that arise both from the literature surrounding this thesis, and from results in the thesis itself. 

From \cite{Rao2015} we have the following conjecture:

\begin{conjecture}
All the aperiodic Wang prototile sets generated by Kari's method are \emph{minimal aperiodic}.
\end{conjecture}

This result holds for all given prototile sets derived and demonstrated in the literature, but we did not make any progress regarding the resolution of this problem. It does, however, make a lot of sense, and would be a good result to complete the picture painted by Rao \etal. 

Recall $PTile_{FIN}$ is the set of finite prototile sets for whom all tilings are periodic, we stated the following conjecture: 

\begin{conjecture}
$PTile_{FIN}$ is not arithmetical.
\end{conjecture}

This is motivated by the need to at some point quantify over the entire class of tilings for some finite prototile set $S$ in order to assert that $S \in PTile_{FIN}$, and this need seems unavoidable. However this is not something we have yet been able to show in general. The possible vectors are bounded, which may belie some clever trick for making $PTile_{FIN}$ arithmetical, but this is thus far elusive.

Lastly, recall that $ATile_{FIN} \in \Pi^0_1$, it would seem natural to derive some notion of measure on a prototile set's tilings, in order to derive the following conjecture - an analogue of Kucera's key result (see \cite{Downey2010} for an exposition):

\begin{conjecture}[C. 2019]
For a notion of positive measure on $\mathcal{S}$-tilings, for some prototile set $\mathcal{S}$, if a tiling $T$ has positive measure:
\begin{itemize}
\item $T$ is aperiodic.
\item $T$ encodes some Martin-L\"of Random.
\end{itemize}
\end{conjecture}

However, the work to identify a suitable notion of measure was not yet undertaken. We suspect that this can be achieved by means of analysis on the `colour density' for coloured edges/Want tile quadrants.

It is also worth noting that the following conjecture is unresolved:
\begin{conjecture}
The tiling method due to Socolar in \cite{Socolar1988} does indeed lead to total planar aperiodic tilings.
\end{conjecture}
It is our strong opinion that this is true by means of an application of WKL to some additional machinery added to the construction that is presented. However, the details have not yet been worked out to see if this can be achieved.

Finally, we have our conjecture from chapter 6:

\begin{conjecture}
	There exist ECA prototile sets of 8 tiles.
\end{conjecture}

As noted there, this cannot be formed of Wang tiles, but there is likely some way of cutting a planar representation of a given ECA into a regular single prototile per part of each rule. Shapes for this result would probably resemble interlocking tilings that look like a double conjoined `H', as detailed in \cite{GrunbaumTP}.

Finally, we note that the work in Chapter 5 on Weihrauch reducibility for tiling problems as principles has the capability to be taken much further. We alluded to one, for which we gave a definition of $WIPT$, accompanied by the following estimate of $C_{\omega^\omega} \leq_W C_{2^\omega} \star C_\omega \star WIPT$.

Indeed, we consider that there are many further applications for tiling problems, in particular for dimensionality $\geq 2$ and for non-Euclidian planar tilings.

A good starting point for the latter is the result due to Beauquier, Muller, and Schupp in \cite{Beauquier1999}. Here, they showed that a tiling problem known as ``the Bar Problem'' - the question of whether a plane that has holes in it can be covered with $(1 \times n)$ `bars' - is $NP$-complete in the Euclidian plane, however in the hyperbolic plane it becomes polynomial time.

Overall, we hope that we have demonstrated some interesting results regarding tiling problems, and laid down some framework and exposition that encourages future results.

%
%
%

\bibliographystyle{amsplain}
\bibliography{refs}

\printindex

\end{document}